\documentclass[10pt,a4paper]{amsart}

\usepackage{amsthm,amsmath,amssymb,amsxtra,amscd,enumerate}
\usepackage[all]{xy}
\usepackage{longtable}
\usepackage{epic,eepic}
\usepackage{float}
\numberwithin{equation}{section}
\usepackage{hyperref}%

\theoremstyle{plain}
\newtheorem{thm}{Theorem}[section]
\newtheorem{prop}[thm]{Proposition}
\newtheorem{lem}[thm]{Lemma}
\newtheorem{de}[thm]{Definition}
\newtheorem{rem}[thm]{Remark}
\newtheorem{cor}[thm]{Corollary}

\newtheorem{set}[thm]{Setting}
\newtheorem{ex}[thm]{Example}

\newcommand{\bb}{\mathbb}
\newcommand{\cal}{\mathcal}

\DeclareMathOperator*{\Hom}{Hom}
\DeclareMathOperator{\Ad}{Ad}
\DeclareMathOperator{\ad}{ad}

\title[discretely decomposable restriction of $A_{\mathfrak q}(\lambda)$]{Classification of discretely decomposable $A_{\mathfrak q}(\lambda)$
with respect to reductive symmetric pairs}
\thanks{%
}

\author[Toshiyuki Kobayashi]{Toshiyuki Kobayashi$^*$}
\address[Toshiyuki Kobayashi]{Kavli IPMU and 
        Graduate School of Mathematical Sciences, 
        The University of Tokyo, 3-8-1 Komaba, Meguro, 
        153-8914 Tokyo, Japan}
\thanks{* Partially supported by
        Grant-in-Aid for Scientific Research (B) (22340026), Japan
        Society for the Promotion of Science}
\author[Yoshiki Oshima]{Yoshiki Oshima$^{**}$}
\thanks{** Partially supported by
        Grant-in-Aid for JSPS Fellows}
\address[Yoshiki Oshima]{Graduate School of Mathematical Sciences, 
        The University of Tokyo, 3-8-1 Komaba, Meguro, 153-8914 Tokyo, Japan}

\begin{document}

\maketitle


\begin{abstract}
We give a classification of the triples
$(\mathfrak{g},\mathfrak{g}',\mathfrak{q})$ 
such that Zuckerman's derived functor $(\mathfrak{g},K)$-module
$A_{\mathfrak{q}}(\lambda)$ for a $\theta$-stable parabolic subalgebra
$\mathfrak{q}$ is discretely decomposable with respect to a reductive
symmetric pair $(\mathfrak{g},\mathfrak{g}')$.
The proof is based on the criterion for discretely decomposable
restrictions by the first author and on Berger's classification of
reductive symmetric pairs.
\end{abstract}

\section{Introduction}
\label{sec:intro}

Branching problems in representation theory
 ask how an irreducible representation decomposes 
when restricted to a subgroup (or a subalgebra).  

In the category of unitary representations of a locally compact group $G'$,
 one can describe irreducible decompositions
 by means of the direct integrals of Hilbert spaces.  
The object of our study is the restriction 
 of an irreducible unitary representation $\pi$ of $G$ to its
 subgroup $G'$, in particular when $G$ and $G'$ are both reductive Lie groups.
Then the irreducible decomposition is unique; 
 however,
 it may contain continuous spectrum
 in the direct integral of Hilbert spaces. 

For a reductive Lie group $G$, 
 we can consider branching problems also
 in the category of $({\mathfrak {g}}, K)$-modules.  
If the underlying $(\mathfrak{g},K)$-module $\pi_K$ is
 discretely decomposable as a $(\mathfrak{g}',K')$-module
(see Definition \ref{de:2.1}),
then the branching laws of the restrictions of the unitary
 representation $\pi$ to $G'$ and the $(\mathfrak{g},K)$-module $\pi_K$ 
 to $(\frak{g}',K')$ are essentially 
 the same in the following sense:
\begin{alignat*}{3}
&\pi|_{G'}
&&\simeq \sideset{}{^\oplus}\sum_{\tau\in\widehat{G'}}
         m_\pi(\tau)\tau
&&\quad\text{(Hilbert direct sum)},
\\
&\pi_K|_{(\mathfrak{g}',K')}
&&\simeq \bigoplus_{\tau\in\widehat{G'}}
         m_\pi(\tau)\tau_{K'}
&&\quad\text{(algebraic direct sum)},
\end{alignat*}
where $\widehat{G'}$ is the set of equivalence classes of irreducible
unitary representations of $G'$,
and $\tau_{K'}$ is the underlying $(\mathfrak{g}',K')$-module of
$\tau$.
The key ingredient here 
is that the natural map
\begin{equation}\label{eqn:Hom}
\Hom\nolimits_{(\mathfrak{g}',K')}(\tau_{K'},\pi_K)\to
\Hom\nolimits_{G',\text{continuous}}(\tau,\pi)
\end{equation}
is bijective, and therefore the dimensions of the spaces of homomorphisms 
coincide, giving the same multiplicity $m_\pi(\tau)$ in the branching laws.

It should be noted that \eqref{eqn:Hom} is not surjective in general and
 that the restriction of an irreducible and unitarizable
 $({\mathfrak {g}},K)$-module $\pi_K$ may not be decomposed
 into an algebraic direct sum
 of irreducible $({\mathfrak {g}}',K')$-modules.  
Such a phenomenon happens whenever continuous spectrum appears
 in the branching law of the restriction of the
 unitary representation $\pi$ to ${G'}$.

The aim of this article is to give a classification
 of the triples $({\mathfrak {g}}, {\mathfrak {g}}', \pi_K)$
 such that the $({\mathfrak {g}}, K)$-module $\pi_K$ is discretely decomposable
 as a $({\mathfrak {g}}', K')$-module
 in the setting
 where 
\begin{align*}
&\text{$({\mathfrak {g}}, {\mathfrak {g}}')$ is a reductive symmetric pair,}
\\
&\text{$\pi_K$ is Zuckerman's derived functor module $A_{\frak {q}}(\lambda)$.}  
\end{align*}

The condition for discrete decomposability does not change if we replace $K$ and $K'$ 
 by their finite covering groups or subgroups of finite index.
Thus, we may and do assume that $K$ is connected and 
$K'=K^\sigma$ (or equivalently $G'=G^\sigma$), 
where $\sigma$ is an involution of $G$ leaving $K$ stable.
Further,
 the condition for discrete decomposability of $A_{\mathfrak{q}}(\lambda)$
 depends on a $\theta$-stable
 parabolic subalgebra ${\mathfrak {q}}$, 
but is independent of the parameter $\lambda$ in the good range.  

Our main result is Theorem \ref{ddclass} 
 with Tables \ref{holparablist}, \ref{holpairlist}, 
 \ref{dslist} and \ref{isolist}. 
 They give a classification of the triples $({\mathfrak {g}}, {\mathfrak {g}}^{\sigma}, {\mathfrak {q}})$
 for which $A_{\mathfrak {q}}(\lambda)$
 is discretely decomposable 
 as a $({\mathfrak {g}}^{\sigma}, K^{\sigma})$-module.  
The list is described up to the conjugacy of 
$K \times K$ 
 as we explain at the beginning of Section \ref{sec:class}.
We find that quite a large part of such triples 
 $({\mathfrak{g}}, {\mathfrak{g}}^{\sigma}, {\mathfrak{q}})$
 appear as a \lq{family}\rq\
 containing $({\mathfrak{g}}, {\mathfrak{g}}^{\sigma}, {\mathfrak{b}})$ 
with ${\mathfrak{b}}$  a $\theta$-stable Borel subalgebra.  
We call them {\it discrete series type} (see Tables \ref{holparablist}, \ref{holpairlist} 
and \ref{dslist}), 
which include {\it holomorphic type} as a special case
 (see Proposition \ref{holdd}).
Moreover, there are some other triples, which we refer to as 
{\it isolated type} (see Table \ref{isolist}).

The tensor product of two representations
 is an example of the restriction 
 with respect to symmetric pairs.  
Thus,
a very special case of our theorem includes the classification
 of two discrete series representations $\pi_1$ and $\pi_2$ of $G'$
 such that the tensor product representation $\pi_1 \otimes \pi_2$ decomposes discretely 
 (see Corollary \ref{dstensor}).  

There exist irreducible symmetric pairs $(\frak{g},\frak{g}^\sigma)$ for
 which any non-trivial $A_\frak{q}(\lambda)$ is not discretely decomposable.
We give a classification of all such pairs $(\frak{g},\frak{g}^\sigma)$
 in Theorem \ref{notdd}.

The proof is based on the 
criterion for the discretely decomposable restriction
established in \cite{kob94,kob98i,kob98ii}, 
 see Theorem \ref{crit},
and on the classification of reductive symmetric pairs 
$(\mathfrak{g},\mathfrak{g}')$ by Berger \cite{ber}
up to outer automorphisms of $\mathfrak{g}$.


\section{Discretely Decomposable $A_\frak{q}(\lambda)$ for Symmetric Pairs}
\label{sec:dd}

Let $G$ be a connected real reductive Lie group. 
We write $\frak {g}$ for the Lie algebra of $G$ and $\frak{g}_\bb{C}$
for its complexification.  
Analogous notation will be used for other Lie algebras.

Let $\sigma$ be an involutive automorphism of $G$, 
and we set $G^\sigma :=\{g\in G:\sigma g=g\}$.
Then $(G,G^\sigma)$ forms a reductive symmetric pair.
Take a Cartan involution $\theta$ of $G$ which commutes with $\sigma$.
Then $K:=G^\theta$ and $K^\sigma=K\cap G^\sigma$ are maximal compact subgroups 
of $G$ and $G^\sigma$, respectively.
We let $\theta$ and $\sigma$ also denote the induced involutions 
on $\frak{g}$ and their complex linear extensions to 
$\frak{g}_\bb{C}$.
The Cartan decompositions are denoted by $\frak{g}=\frak{k}+\frak{p}$ and 
$\frak{g}^\sigma=\frak{k}^\sigma+\frak{p}^\sigma$, respectively.

We recall from \cite{kob98ii}
 the following basic notion,
 which we shall apply to branching problems
 in the category of $({\mathfrak {g}},K)$-modules.  

\begin{de}\label{de:2.1}
{\rm 
We say that a $(\frak{g},K)$-module 
$V$ is {\it discretely decomposable} if 
there exists an increasing filtration $\{V_n\}$ such that 
$V=\bigcup_{n=0}^{\infty} V_n$ and 
each $V_n$ is of finite length as a 
$(\frak{g},K)$-module.
}
\end{de}

\begin{rem}[see {\cite[Lemma 1.3]{kob98ii}}]
{\rm 
Suppose that $V$ is a unitarizable $(\frak{g},K)$-module.
Then $V$ is discretely decomposable if and only if 
$V$ is isomorphic to 
the algebraic direct sum of irreducible $(\frak{g},K)$-modules.
}
\end{rem}

Next,
 let us fix some notation concerning Zuckerman's derived functor modules
 $A_{\mathfrak{q}}(\lambda)$.  
Suppose $\frak{q}$ is a $\theta$-stable parabolic subalgebra of 
$\frak{g}_\bb{C}$.
The normalizer $L=N_G(\frak{q})$ of $\frak{q}$ is a connected reductive 
subgroup of $G$.
Hence a unitary character $\bb{C}_\lambda$ of $L$ is determined by 
its differential $\lambda\in\sqrt{-1}\frak{l}^*$. 
Associated to the data $(\frak{q},\lambda)$, one defines 
Zuckerman's derived functor module
 $A_\frak{q}(\lambda)$ as in \cite[(5.6)]{KnVo}.
In our normalization, $A_\frak{q}(0)$ is a unitarizable $(\frak {g}, K)$-module
 with non-zero $(\frak{g},K)$-cohomologies,
 and in particular,
 has the same infinitesimal character 
as the trivial one-dimensional representation $\bb{C}$ of $\frak{g}$.
We note that if $\frak{q}=\frak{g}_\bb{C}$, then $L=G$ and 
 $A_\frak{q}(\lambda)$ is one-dimensional.

Take a fundamental Cartan subalgebra $\frak{h}$ of $\frak{l}$  
and choose a positive root system $\Delta^+(\frak{g}_\bb{C},\frak{h}_\bb{C})$
such that the set $\Delta(\frak{q},\frak{h}_\bb{C})$ 
of roots for $\frak{q}$ contains all the positive roots, and
 set $\Delta^+(\frak{l}_\bb{C},\frak{h}_\bb{C})
 :=\Delta(\frak{l}_\bb{C},\frak{h}_\bb{C})
 \cap\Delta^+(\frak{g}_\bb{C},\frak{h}_\bb{C})$.
Let $\frak{u}$ be the nilradical of $\frak{q}$.
Denote by $\rho, \rho_{\frak{l}}$, and $\rho(\frak{u})\in\frak{h}_\bb{C}^*$
 half the sum of roots in
 $\Delta^+(\frak{g}_\bb{C},\frak{h}_\bb{C})$,
 $\Delta^+(\frak{l}_\bb{C},\frak{h}_\bb{C})$,
 and $\Delta(\frak{u},\frak{h}_\bb{C})$, respectively.
Let $\langle\cdot,\cdot\rangle$ be an invariant bilinear form 
on $\frak{h}_\bb{C}^*$ 
that is positive definite on the real span of the roots.
Following the terminology \cite[Definitions 0.49 and 0.52]{KnVo},
we say for a unitary character $\bb{C}_\lambda$ of $L$, 
$\lambda$ is in the {\it good range} if 
\[{\rm Re} \langle \lambda+\rho,\alpha\rangle>0\quad 
\alpha\in\Delta(\frak{u},\frak{h}_\bb{C}), \]
and in the {\it weakly fair range} if 
\[{\rm Re} \langle \lambda+\rho(\frak{u}),\alpha\rangle\geq 0\quad 
\alpha\in\Delta(\frak{u},\frak{h}_\bb{C}).\]
The $K$-finite Hermitian dual of the $(\frak{g},K)$-module
 $A_{\frak{q}}(\lambda)$ in the normalization here is isomorphic
 to the cohomologically induced module
 ${\cal R}_\frak{q}^S(\bb{C}_{\nu})$ with
 $S=\dim_\bb{C}(\frak{u}\cap \frak{k}_\bb{C})$ and
 $\nu=\lambda+\rho(\frak{u})$
 in the normalization of \cite{kob94}.
Accordingly, the good range (resp.\ the weakly fair range)
 amounts to the condition on $\nu$ as 
\[{\rm Re} \langle \nu+\rho_\frak{l},\alpha\rangle> 0\quad 
\alpha\in\Delta(\frak{u},\frak{h}_\bb{C}) \quad 
 ({\text{resp. }} {\rm Re} \langle \nu,\alpha\rangle\geq 0\quad 
\alpha\in\Delta(\frak{u},\frak{h}_\bb{C})).\]
We pin down some basic properties of the $({\mathfrak{g}}, K)$-module 
 $A_{\frak {q}}(\lambda)$ (\cite[Chapters VIII and IX]{KnVo}).
\begin{thm}
If $\lambda$ is in the weakly fair range, $A_\frak{q}(\lambda)$
is unitarizable or zero.
If $\lambda$ is in the good range, $A_\frak{q}(\lambda)$ 
is non-zero and irreducible.
\end{thm}

\begin{thm}
\label{aqds}
Suppose that ${\rm rank}\,\frak{g}_\bb{C}={\rm rank}\,\frak{k}_\bb{C}$.
If $\frak{q}$ is a $\theta$-stable Borel subalgebra of $\frak{g}_\bb{C}$ 
and if $\lambda$ is in the good range, 
then $A_\frak{q}(\lambda)$ is isomorphic to the 
underlying $(\frak{g},K)$-module of a discrete series representation of $G$.
Conversely the underlying $(\frak{g},K)$-module of 
any discrete series representation of $G$ is isomorphic to 
$A_\frak{q}(\lambda)$ for some $\theta$-stable Borel subalgebra $\frak{q}$ 
and $\lambda$ in the good range.
\end{thm}

The goal of this article is to give a classification of 
the triples $(G,G^\sigma,\frak{q})$ such that 
the $({\mathfrak{g}},K)$-module $A_\frak{q}(\lambda)$ is 
discretely decomposable 
as a $(\frak{g}^\sigma,K^\sigma)$-module. 
Since the discrete decomposability depends only on 
the triple $(\frak{g},\frak{g}^\sigma,\frak{q})$ of Lie algebras 
 and not on the Lie group $G$, 
our classification will be given in terms of the Lie algebras.

To pursue the classification,
 we prepare some further basic setups:
\begin{de}
\label{def:irrsp}
{\rm{
We say the pair $(\frak{g},\frak{g}^\sigma)$ is an 
{\it irreducible symmetric pair} if one of the following holds.
\begin{enumerate}
\item[(1)]
$\frak{g}$ is simple.
\item[(2)]
$\frak{g}'$ is simple and $\frak{g}\simeq\frak{g}'\oplus\frak{g}'$;
$\sigma$ acts by switching the factors.
\end{enumerate}
}}
\end{de}
Let $\frak{g}=\frak{g}'\oplus\frak{g}'$ and $\varphi$ 
a non-trivial automorphism of $\frak{g}'$. 
Then there is also a symmetric pair $(\frak{g},\frak{g}^\sigma)$
defined by the involution 
$\sigma(x,y):=(\varphi(y),\varphi^{-1}(x))$ for $x,y\in\frak{g}'$.
For the simplicity of the exposition, we exclude this case from 
the definition of irreducible pairs.
This does not lose any generality for our purpose 
because we have an isomorphism 
$A_{\frak{q}'_1\oplus\frak{q}'_2}(\lambda_1,\lambda_2)|_{\frak{g}^\sigma}
\simeq 
A_{\frak{q}'_1\oplus \varphi(\frak{q}'_2)}
(\lambda_1, (\varphi^*)^{-1}\lambda_2)|_{{\rm diag}(\frak{g}')}$
via the isomorphism
$\frak{g}^\sigma\simeq {\rm diag}(\frak{g}')$, 
$(x,\varphi^{-1}(x))\mapsto (x, x)$.
Here, $\frak{q}'_1,\frak{q}'_2$ are parabolic subalgebra of $\frak{g}'_\bb{C}$ 
and ${\rm diag}(\frak{g}')$ is the diagonal in 
$\frak{g}=\frak{g}'\oplus \frak{g}'$.
Therefore the discrete decomposability for 
the triple $(\frak{g},\frak{g}^{\sigma},\frak{q}'_1\oplus\frak{q}'_2)$ 
is equivalent to 
that for the triple 
$(\frak{g},{\rm diag}(\frak{g}'),\frak{q}'_1\oplus\varphi(\frak{q}'_2))$. 
We shall treat the latter case in Section \ref{sec:tensor}.

We should remark that our definition differs from 
the one in Berger \cite{ber}, where the pair $(\frak{g},\frak{g}^\sigma)$ 
was called irreducible if $\frak{g}^{-\sigma}$ is 
an irreducible $\frak{g}^{\sigma}$-module. 
For example, 
$(\frak{sl}(n,\bb{R}),\frak{sl}(m,\bb{R})\oplus\frak{sl}(n-m,\bb{R})
\oplus\bb{R})$ is 
an irreducible pair for the Definition \ref{def:irrsp}, 
while it is not for the definition of \cite{ber}.
Both definitions are the same for Riemannian symmetric pairs.

Any semisimple symmetric pair is isomorphic to the direct sum of irreducible 
symmetric pairs.  
In particular,
 branching problems of $A_\frak{q}(\lambda)$ 
 with respect to reductive symmetric pairs
 reduce to the case of irreducible symmetric pairs
 because any $\theta$-stable parabolic subalgebra ${\mathfrak{q}}$
 is obviously written as the direct sum
 of $\theta$-stable parabolic subalgebras 
 of each factor.  

To describe $\theta$-stable parabolic subalgebras
 of ${\mathfrak {g}}_{\mathbb{C}}$, 
 it is convenient to use the following convention:
\begin{de}
\label{def:qa}
{\rm{
We say that a parabolic subalgebra $\frak{q}$ of $\frak{g}_\bb{C}$ 
is given by a vector $a\in\sqrt{-1}\frak{k}$ if 
$\frak{q}$ is the sum of non-negative eigenspaces of $\ad(a)$.
}}
\end{de}
Then ${\mathfrak {q}}$ is a $\theta$-stable parabolic subalgebra
 with a Levi decomposition $\frak{q}=\frak{l}_{\mathbb{C}}+\frak{u}$,
 where $\frak{l}_{\mathbb{C}}$ and $\frak{u}$ are the sums
 of zero and positive eigenspaces of 
$\ad(a)$, respectively.
Note that any $\theta$-stable parabolic subalgebras are obtained in this way.  

Needless to say,
 the defining element $a$ of a $\theta$-stable parabolic subalgebra 
${\mathfrak {q}}$
 is not unique.
However, 
 we adopt this convention
 in our classification
 (Tables \ref{holparablist}, \ref{dslist} and \ref{isolist})
 because it is not hard
 to compute ${\mathfrak{q}}$
 and $L=N_G({\mathfrak{q}})$
 from the defining element $a$
 by using the root system.  

Replacing $\frak{q}$ by $\Ad(k)\frak{q}$ for $k\in K$ if necessary,
we restrict ourselves to consider the following setting.

\begin{set}
\label{setq}
{\rm

(1)
Suppose that $(\frak{g},\frak{g}^\sigma)$ 
is an irreducible symmetric pair and 
the involution $\sigma$ commutes with a Cartan involution $\theta$.
Fix a $\sigma$-stable Cartan subalgebra 
$\frak{t}=\frak{t}^{\sigma}+\frak{t}^{-\sigma}$ of $\frak{k}$ 
such that $\frak{t}^{-\sigma}$ is maximal abelian in $\frak{k}^{-\sigma}$.
Choose a positive system $\Delta^+(\frak{k}_\bb{C},\frak{t}_\bb{C})$
that is compatible with 
some positive system 
of the restricted root system 
$\Sigma^+(\frak{k}_\bb{C},\sqrt{-1}\frak{t}^{-\sigma})$.

\noindent (2)
Let $\frak{q}$ be a $\theta$-stable parabolic subalgebra 
of $\frak{g}_\bb{C}$. We assume that $\frak{q}$ is given by
a $\Delta^+(\frak{k}_\bb{C},\frak{t}_\bb{C})$-dominant vector
$a\in\sqrt{-1}\frak{t}$.
}
\end{set}

Since $\frak{t}$ is $\sigma$-stable, $\sigma$ acts on the 
complexification $\frak{t}_\bb{C}$ and also on 
the dual space $\frak{t}_\bb{C}^*$, which is denoted by 
the same letter $\sigma$.
We note
 that ${\mathfrak {p}}_{\mathbb{C}}$
 and the nilradical ${\mathfrak {u}}$ of ${\mathfrak{q}}$
 are ${\mathfrak {t}}_{\mathbb{C}}$-invariant subspaces.  
We write $\Delta({\mathfrak {p}}_{\mathbb{C}}, {\mathfrak {t}}_{\mathbb{C}})$, 
 $\Delta({\mathfrak {u}} \cap {\mathfrak {p}}_{\mathbb{C}}, {\mathfrak {t}}_{\mathbb{C}})$
 for the sets of the weights of ${\mathfrak {t}}_{\mathbb{C}}$
 in ${\mathfrak {p}}_{\mathbb{C}}$, 
 ${\mathfrak {u}} \cap {\mathfrak {p}}_{\mathbb{C}}$, 
respectively. 
Here is a summary on equivalent conditions 
for discretely decomposable restrictions of $A_{\mathfrak {q}}(\lambda)$
 with respect to reductive symmetric pairs.  
We shall use the condition (iii) for our classification
 of the triples $({\mathfrak {g}}, {\mathfrak {g}}^{\sigma}, {\mathfrak {q}})$.  

\begin{thm}
\label{crit}
In Setting \ref{setq}, 
the following eight conditions on the triple $(\mathfrak{g},\mathfrak{g}^\sigma,\mathfrak{q})$
are equivalent: 
\begin{enumerate}
\item[{\rm (i)}]
$A_\frak{q}(\lambda)$ is non-zero and 
discretely decomposable as a $(\frak{g}^\sigma,K^\sigma)$-module 
for some $\lambda$ in the weakly fair range.
\item[{\rm (i$'$)}]
$A_\frak{q}(\lambda)$ is discretely decomposable 
as a $(\frak{g}^\sigma,K^\sigma)$-module 
for any $\lambda$ in the weakly fair range.
\item[{\rm (ii)}]
$\bb{R}_+\langle\frak{u}\cap\frak{p}_\bb{C}\rangle
\cap\sqrt{-1}(\frak{t}^{-\sigma})^*=\{0\}$.
Here, we define 
\[\bb{R}_+\langle\frak{u}\cap\frak{p}_\bb{C}\rangle
:=\left\{\sum_{\alpha\in\Delta(\frak{u}\cap\frak{p}_\bb{C},\frak{t}_\bb{C})}
n_\alpha \alpha:n_\alpha\in\bb{R}_{\geq 0}\right\}.\]
\item[{\rm (ii$'$)}]
There exists $b\in\sqrt{-1}\frak{t}^{\sigma}$ such that 
$\langle 
{\rm pr}_+(\bb{R}_+\langle\frak{u}\cap\frak{p}_\bb{C}\rangle), 
b\rangle>0$,
where ${\rm pr}_+: \sqrt{-1} \frak{t}^* \to 
\sqrt{-1} (\frak{t}^{\sigma})^*$ is 
the restriction map.
\item[{\rm (iii)}]
$\sigma\alpha(a)\geq 0$ 
whenever 
$\alpha\in\Delta(\frak{p}_\bb{C},\frak{t}_\bb{C})$
 satisfies $\alpha(a)>0$.   
\item[{\rm (iii$'$)}]
$\sigma(\frak{u}\cap\frak{p}_\bb{C})\subset \frak{q}.$
\item[{\rm (iv)}]
Let ${\cal V}(A_\frak{q}(\lambda))$ be the associated variety of 
$A_\frak{q}(\lambda)$ and ${\rm pr}_+:\frak{g}_\bb{C}^*
\to (\frak{g}^\sigma_\bb{C})^*$ the 
restriction map.
Then ${\rm pr}_+{\cal V}(A_\frak{q}(\lambda))$ is contained 
in the nilpotent cone of $(\frak{g}^\sigma_\bb{C})^*$
for any $\lambda$ in the weakly fair range.
\item[{\rm (v)}]
Each $K^\sigma$-type occurs in $A_\frak{q}(\lambda)$ 
with finite multiplicity 
for any $\lambda$ in the weakly fair range.
\end{enumerate}
\end{thm}

If one of, and hence any of, 
 these equivalent conditions holds, 
we say that the triple $(\frak{g},\frak{g}^\sigma,\frak{q})$ satisfies 
the {\it discrete decomposability condition}.

\begin{proof}
The equivalence of (i), (i$'$), (ii), (iii$'$), (iv), and (v)
 was established in \cite{kob94, kob98i, kob98ii}.  
To be more precise,
 the implication (ii) $\Rightarrow$ (v) was proved in \cite{kob94}, 
 and an alternative proof based on micro-local analysis
 was given in \cite{kob98i}.  
The opposite direction (v) $\Rightarrow$ (i$'$) $\Rightarrow$ 
(i) $\Rightarrow$ (iv) 
$\Rightarrow$ (iii$'$) $\Rightarrow$ (ii) was proved in \cite{kob98ii}.  
The conditions (ii$'$) and (iii) are just reformulations of (ii) and (iii$'$), 
respectively.
\end{proof}

We end this section with a number of direct consequences of 
Theorem \ref{crit}, namely, 
one equivalent condition (Proposition \ref{ass}), 
two sufficient conditions 
(Propositions \ref{incl}, \ref{holdd}) 
and two necessary conditions (Propositions \ref{split}, \ref{highlow}) 
for the discrete decomposability of 
$A_\frak{q}(\lambda)$ as a $(\frak{g}^{\sigma},K^{\sigma})$-module.

Suppose that an involution $\sigma$ of $G$ commutes 
with a Cartan involution $\theta$. 
Then the composition $\theta\sigma$ becomes another involution of $G$.  
The symmetric pair $(\frak{g},\frak{g}^{\theta\sigma})$ 
is called the {\it associated pair} of $(\frak{g},\frak{g}^\sigma)$.

Since ${\sigma} ={\theta \sigma}$ on $\frak{t}$, 
we get from the condition (iii) in Theorem \ref{crit}
 the following proposition:
\begin{prop}
\label{ass}
For $\lambda$ in the weakly fair range, 
$A_\frak{q}(\lambda)$ is discretely decomposable as a 
$(\frak{g}^\sigma,K^\sigma)$-module if and only if 
it is discretely decomposable 
as a $(\frak{g}^{\theta\sigma},K^{\theta\sigma})$-module.
\end{prop}
The following proposition is a direct consequence of 
the condition (ii) in Theorem \ref{crit}:

\begin{prop}
\label{incl}
Let $\frak{q}_1$ and $\frak{q}_2$ be $\theta$-stable parabolic subalgebras
of $\frak{g}_\bb{C}$ such that $\frak{q}_1\subset\frak{q}_2$.
If $(\frak{g},\frak{g}^\sigma,\frak{q}_1)$ 
satisfies the discrete decomposability condition, 
then so does $(\frak{g},\frak{g}^\sigma,\frak{q}_2)$.
\end{prop}

Yet another easy consequence of Theorem \ref{crit}
 concerns the triples
 $({\mathfrak {g}}, {\mathfrak {g}}^{\sigma}, {\mathfrak {q}})$
 for {\it{holomorphic}} ${\mathfrak {q}}$ 
 (Definition \ref{def:holoparab})
 defined for a Hermitian Lie algebra ${\mathfrak {g}}$
 below: 

\begin{de}
{\rm 
Let ${\mathfrak {g}}={\mathfrak {k}}+{\mathfrak {p}}$
 be a real non-compact simple Lie algebra.  
We say $\frak{g}$ is 
of {\it Hermitian type} 
and the symmetric pair $(\frak{g},\frak{k})$ is a 
{\it Hermitian symmetric pair}
if the center $\frak{z}_K$ of $\frak{k}$ is one-dimensional.
}
\end{de}

If $\frak{g}$ is of Hermitian type, 
then $\frak{p}_\bb{C}$,
 regarded as a $K$-module by the adjoint action, 
decomposes into the direct sum 
 of two irreducible submodules, 
 say,  
$\frak{p}_\bb{C}=\frak{p}_+ + \frak{p}_-$.
Then the Riemannian symmetric space
$G/K$ becomes a Hermitian symmetric space
by choosing $\frak{p}_-$ as a holomorphic tangent space at the base point. 

\begin{de}
\label{def:holoparab}
{\rm 
Suppose that $\frak{g}$ is a simple Lie algebra 
of Hermitian type.
A $\theta$-stable parabolic subalgebra 
$\frak{q}$ of $\frak{g}_\bb{C}$ 
is said to be {\it holomorphic} (resp.\ {\it anti-holomorphic}) 
if $\frak{q}\supset\frak{p}_+$
(resp.\ $\frak{q}\supset\frak{p}_-$). 
}
\end{de}

See Table \ref{holparablist} for the conditions on a defining element $a$ 
for parabolic subalgebra ${\mathfrak {q}}$ 
to be holomorphic or anti-holomorphic.  

If a $\theta$-stable parabolic subalgebra $\frak{q}$ is holomorphic 
and if $A_\frak{q}(\lambda)$ is non-zero and irreducible
(in particular, if $\lambda$ is in the good range), then 
$A_\frak{q}(\lambda)$ is a lowest weight module 
with respect to a 
Borel subalgebra containing $\frak{p}_+$.
Similarly, if $\frak{q}$ is anti-holomorphic, 
$A_\frak{q}(\lambda)$ is a highest weight module. 
If $\frak{q}\cap\frak{p}_\bb{C}=\frak{p}_+$
(resp.\ $\frak{q}\cap\frak{p}_\bb{C}=\frak{p}_-$) 
and $\lambda$ is in the good range, then 
$A_\frak{q}(\lambda)$ is 
the underlying $(\frak{g},K)$-module of 
a holomorphic (resp.\ anti-holomorphic) discrete series 
representation of $G$.

\begin{de}
{\rm 
Suppose that $\frak{g}$ is a simple Lie algebra of Hermitian type, 
so the center $\frak{z}_K$ of $\frak{k}$ is one-dimensional.
We say a symmetric pair 
$(\frak{g},\frak{g}^\sigma)$ 
is of {\it holomorphic type}
if $\frak{z}_K\subset\frak{g}^\sigma$, 
or equivalently if $\sigma$ induces a holomorphic involution 
 on the Hermitian symmetric space $G/K$.
}
\end{de}

It follows immediately from $\frak{k}^\sigma=\frak{k}^{\theta\sigma}$
 that the pair $(\frak{g},\frak{g}^\sigma)$
 is of holomorphic type if and only if the associated pair 
 $(\frak{g},\frak{g}^{\theta\sigma})$ is of holomorphic type.
See Table \ref{holpairlist} for the classification
 of symmetric pairs $({\mathfrak {g}}, {\mathfrak {g}}^{\sigma})$
 of holomorphic type.  

\begin{ex}
{\rm 
Let $\frak{g}=\frak{su}(2,2)\simeq\frak{so}(4,2)$.
Suppose we are in Setting \ref{setq},
and retain the notation of Setting \ref{sumn} for 
$\frak{t}$ and $e_i$.
In particular, 
$\frak{q}$ is given by $a=a_1e_1+a_2e_2+a_3e_3+a_4e_4$ 
with $a_1\geq a_2$ and $a_3\geq a_4$.
Figure \ref{fig1} follows the notation in \cite{kob94}.
It shows 18 $\theta$-stable parabolic subalgebras of $\frak{g}_\bb{C}$, 
which form a complete set of representatives of 
$\frak{q}$ up to $K$-conjugacy and the equivalence relation 
among the $\theta$-stable parabolic subalgebras $\frak{q}_1\sim \frak{q}_2$ 
defined by a $(\frak{g},K)$-isomorphism 
$A_{\frak{q}_1}(0)\simeq A_{\frak{q}_2}(0)$.
The correspondence is: 
$X_1\leftrightarrow a_1 > a_2 > a_3 > a_4,$ 
$X_2\leftrightarrow a_1 > a_3 > a_2 > a_4,$ 
$X_3\leftrightarrow a_1 > a_3 > a_4 > a_2,$ 
$X_4\leftrightarrow a_3 > a_1 > a_2 > a_4,$ 
$X_5\leftrightarrow a_3 > a_1 > a_4 > a_2,$ 
$X_6\leftrightarrow a_3 > a_4 > a_1 > a_2,$ 
$Y_1\leftrightarrow a_1 > a_2 = a_3 > a_4,$ 
$Y_2\leftrightarrow a_1 > a_3 > a_2 = a_4,$ 
$Y_3\leftrightarrow a_1 = a_3 > a_2 > a_4,$ 
$Y_4\leftrightarrow a_3 > a_1 > a_2 = a_4,$ 
$Y_5\leftrightarrow a_1 = a_3 > a_4 > a_2,$ 
$Y_6\leftrightarrow a_3 > a_1 = a_4 > a_2,$ 
$Z_1\leftrightarrow a_1 > a_2 = a_3 = a_4,$ 
$Z_2\leftrightarrow a_1 = a_2 = a_3 > a_4,$ 
$Z_3\leftrightarrow a_3 > a_1 = a_2 = a_4,$ 
$Z_4\leftrightarrow a_1 = a_3 = a_4 > a_2,$ 
$W  \leftrightarrow a_1 = a_3 > a_2 = a_4,$ 
$U  \leftrightarrow a_1 = a_2 = a_3 = a_4.$ 
We see that $X_1,\dots,X_6$ yield $\theta$-stable Borel subalgebras and 
$X_1,X_6,Y_1,Y_6,Z_1,Z_2,Z_3,Z_4,U$ yield
holomorphic or anti-holomorphic parabolic subalgebras.
}
\end{ex}

\begin{figure}[H]
\centering
\setlength{\unitlength}{0.0004in}
\begingroup\makeatletter\ifx\SetFigFont\undefined%
\gdef\SetFigFont#1#2#3#4#5{%
  \reset@font\fontsize{#1}{#2pt}%
  \fontfamily{#3}\fontseries{#4}\fontshape{#5}%
  \selectfont}%
\fi\endgroup%
{\renewcommand{\dashlinestretch}{30}
\begin{picture}(6696,5761)(0,-10)
\put(1548,5384){\ellipse{680}{680}}
\put(2748,5384){\ellipse{680}{680}}
\put(3948,5384){\ellipse{680}{680}}
\put(5148,5384){\ellipse{680}{680}}
\put(648,943){\ellipse{680}{680}}
\put(348,5384){\ellipse{680}{680}}
\put(6348,5384){\ellipse{680}{680}}
\put(1623,2609){\blacken\ellipse{140}{140}}
\put(1623,2609){\ellipse{140}{140}}
\put(3348,2609){\blacken\ellipse{140}{140}}
\put(3348,2609){\ellipse{140}{140}}
\put(5073,2609){\blacken\ellipse{140}{140}}
\put(5073,2609){\ellipse{140}{140}}
\path(6348,4684)(6048,4159)(6648,4159)(6348,4684)
\path(948,3634)(648,3109)(1248,3109)(948,3634)
\path(2148,3634)(1848,3109)(2448,3109)(2148,3634)
\path(5748,3634)(5448,3109)(6048,3109)(5748,3634)
\path(1548,4159)(1056,3460)
\path(1548,4159)(1056,3460)
\path(411,4141)(840,3463)
\path(411,4141)(840,3463)
\path(522,4150)(2034,3454)
\path(522,4150)(2034,3454)
\path(2721,4187)(2253,3448)
\path(2721,4187)(2253,3448)
\path(2058,3100)(1703,2709)
\path(2058,3100)(1703,2709)
\path(1059,3106)(1530,2717)
\path(1059,3106)(1530,2717)
\path(1757,2551)(3237,1960)
\path(1757,2551)(3237,1960)
\path(2786,4182)(3225,3563)
\path(2786,4182)(3225,3563)
\path(3948,4159)(3468,3578)
\path(3948,4159)(3468,3578)
\path(5148,4159)(3561,3494)
\path(5148,4159)(3561,3494)
\path(6183,4156)(4671,3451)
\path(6183,4156)(4671,3451)
\path(4023,4159)(4467,3484)
\path(4023,4159)(4467,3484)
\path(4662,3094)(4988,2724)
\path(4662,3094)(4988,2724)
\path(6300,4144)(5856,3451)
\path(6300,4144)(5856,3451)
\path(5625,3097)(5190,2732)
\path(5625,3097)(5190,2732)
\path(4929,2566)(3459,1966)
\path(4929,2566)(3459,1966)
\path(4548,3634)(4248,3109)(4848,3109)(4548,3634)
\path(5223,4159)(5658,3466)
\path(5223,4159)(5658,3466)
\path(3348,3109)(3338,2762)
\path(3348,3109)(3338,2762)
\path(646,537)(346,12)(946,12)(646,537)
\path(1551,5041)(2756,4522)
\path(1551,5041)(2756,4522)
\path(2745,5044)(1623,4534)
\path(2745,5044)(1623,4534)
\path(5148,4534)(2745,5044)
\path(5148,4534)(2745,5044)
\path(4098,4534)(5133,5032)
\path(4098,4534)(5133,5032)
\path(5153,5039)(6267,4573)
\path(5153,5039)(6267,4573)
\path(6348,5734)(6048,5209)(6648,5209)(6348,5734)
\path(3348,2134)(3048,1609)(3648,1609)(3348,2134)
\path(3338,2484)(3348,2152)
\path(3338,2484)(3348,2152)
\path(6348,5034)(6348,4696)
\path(6348,5034)(6348,4696)
\path(5148,5032)(5148,4534)
\path(5148,5032)(5148,4534)
\path(3948,5038)(3948,4534)
\path(3948,5038)(3948,4534)
\path(2811,4517)(3936,5035)
\path(2811,4517)(3936,5035)
\path(1548,5041)(1548,4534)
\path(1548,5041)(1548,4534)
\path(348,5044)(348,4696)
\path(348,5044)(348,4696)
\path(429,4573)(1545,5044)
\path(429,4573)(1545,5044)
\path(348,4684)(48,4159)(648,4159)(348,4684)
\path(348,5734)(48,5209)(648,5209)(348,5734)
\path(1623,4159)(3144,3485)
\path(1623,4159)(3144,3485)
\put(1423,4259){\makebox(0,0)[lb]{\smash{{{\SetFigFont{8}{9.6}{\familydefault}{\mddefault}{\updefault}$Y_2$}}}}}
\put(2623,4259){\makebox(0,0)[lb]{\smash{{{\SetFigFont{8}{9.6}{\familydefault}{\mddefault}{\updefault}$Y_3$}}}}}
\put(3823,4259){\makebox(0,0)[lb]{\smash{{{\SetFigFont{8}{9.6}{\familydefault}{\mddefault}{\updefault}$Y_4$}}}}}
\put(5023,4259){\makebox(0,0)[lb]{\smash{{{\SetFigFont{8}{9.6}{\familydefault}{\mddefault}{\updefault}$Y_5$}}}}}
\put(6223,4259){\makebox(0,0)[lb]{\smash{{{\SetFigFont{8}{9.6}{\familydefault}{\mddefault}{\updefault}$Y_6$}}}}}
\put(823,3209){\makebox(0,0)[lb]{\smash{{{\SetFigFont{8}{9.6}{\familydefault}{\mddefault}{\updefault}$Z_1$}}}}}
\put(2023,3209){\makebox(0,0)[lb]{\smash{{{\SetFigFont{8}{9.6}{\familydefault}{\mddefault}{\updefault}$Z_2$}}}}}
\put(3223,3209){\makebox(0,0)[lb]{\smash{{{\SetFigFont{8}{9.6}{\familydefault}{\mddefault}{\updefault}$W$}}}}}
\put(4423,3209){\makebox(0,0)[lb]{\smash{{{\SetFigFont{8}{9.6}{\familydefault}{\mddefault}{\updefault}$Z_3$}}}}}
\put(5623,3209){\makebox(0,0)[lb]{\smash{{{\SetFigFont{8}{9.6}{\familydefault}{\mddefault}{\updefault}$Z_4$}}}}}
\put(1098,809){\makebox(0,0)[lb]{\smash{{{\SetFigFont{10}{12.0}{\familydefault}{\mddefault}{\updefault}: $\theta$-stable Borel subalgebra}}}}}
\put(1098,259){\makebox(0,0)[lb]{\smash{{{\SetFigFont{10}{12.0}{\familydefault}{\mddefault}{\updefault}: holomorphic or anti-holomorphic parabolic subalgebra}}}}}
\put(3223,1709){\makebox(0,0)[lb]{\smash{{{\SetFigFont{8}{9.6}{\familydefault}{\mddefault}{\updefault}$U$}}}}}
\put(223,4259){\makebox(0,0)[lb]{\smash{{{\SetFigFont{8}{9.6}{\familydefault}{\mddefault}{\updefault}$Y_1$}}}}}
\put(223,5309){\makebox(0,0)[lb]{\smash{{{\SetFigFont{8}{9.6}{\familydefault}{\mddefault}{\updefault}$\!X_1$}}}}}
\put(1423,5309){\makebox(0,0)[lb]{\smash{{{\SetFigFont{8}{9.6}{\familydefault}{\mddefault}{\updefault}$\!X_2$}}}}}
\put(2623,5309){\makebox(0,0)[lb]{\smash{{{\SetFigFont{8}{9.6}{\familydefault}{\mddefault}{\updefault}$\!X_3$}}}}}
\put(3823,5309){\makebox(0,0)[lb]{\smash{{{\SetFigFont{8}{9.6}{\familydefault}{\mddefault}{\updefault}$\!X_4$}}}}}
\put(5023,5309){\makebox(0,0)[lb]{\smash{{{\SetFigFont{8}{9.6}{\familydefault}{\mddefault}{\updefault}$X_5$}}}}}
\put(6223,5309){\makebox(0,0)[lb]{\smash{{{\SetFigFont{8}{9.6}{\familydefault}{\mddefault}{\updefault}$\!X_6$}}}}}
\end{picture}
}
\caption{}
\label{fig1}
\end{figure}

The following theorem can be deduced from 
\cite[Theorem 7.4]{kob98iii}.
For the convenience of the reader, we give an alternative proof by using 
the criterion, Theorem \ref{crit} (iii).
\begin{prop}
\label{holdd}
Suppose that a symmetric pair $(\frak{g},\frak{g}^\sigma)$ 
is of holomorphic type and 
a parabolic subalgebra $\frak{q}$ of $\frak{g}_\bb{C}$
is holomorphic or anti-holomorphic.
Then $A_\frak{q}(\lambda)$ is discretely decomposable 
as a $(\frak{g}^\sigma,K^\sigma)$-module for any $\lambda$ 
in the weakly fair range.
\end{prop}
\begin{proof}
Choose $z\in\sqrt{-1}\frak{z}_K$ such that 
$\Delta(\frak{p}_+,\frak{t}_\bb{C}) 
= \{\alpha\in\Delta(\frak{p}_\bb{C},\frak{t}_\bb{C}) 
: \alpha(z) > 0 \}$. 
We observe $\frak{u}\cap \frak{p}_\bb{C}\subset\frak{p}_+$ 
if the $\theta$-stable parabolic subalgebra 
$\frak{q}=\frak{l}_\bb{C}+\frak{u}$ 
is holomorphic. 
Thus, if $\alpha(a)>0$ for $\alpha\in\Delta(\frak{p}_\bb{C},\frak{t}_\bb{C})$, 
then $\alpha\in\Delta(\frak{p}_+,\frak{t}_\bb{C})$.
Since $\sigma(z)=z$, 
the $\sigma$-action on $\frak{t}_\bb{C}^*$ stabilizes 
$\Delta(\frak{p}_+,\frak{t}_\bb{C})$. 
Then $\sigma\alpha\in\Delta(\frak{p}_+,\frak{t}_\bb{C})$ and hence 
$\sigma\alpha(a)\geq 0$.
Thus, Theorem \ref{crit} (iii) is satisfied.
\end{proof}

Conversely, Theorem \ref{crit} gives a simple, 
necessary condition on a pair $(\frak{g},\frak{g}^\sigma)$ 
such that at least one infinite dimensional $A_\frak{q}(\lambda)$ 
is discretely decomposable
 as a $({\mathfrak {g}}^{\sigma}, K^{\sigma})$-module.
\begin{prop}
\label{split}
Let $\frak{g}$ be a simple non-compact Lie algebra and 
$\sigma$ an involution of $\frak{g}$ commuting with $\theta$.
Suppose that 
$\lambda$ is in the weakly fair range, $\frak{q}\neq\frak{g}_\bb{C}$, 
and $A_\frak{q}(\lambda)$ is non-zero.
If $A_\frak{q}(\lambda)$ is 
 discretely decomposable as a $(\frak{g}^\sigma,K^\sigma)$-module, 
 then $\frak{t}^{\sigma} \ne 0$, 
or equivalently $\frak{k}^\sigma+\sqrt{-1}\frak{k}^{-\sigma}$ 
is not a split real form of 
$\frak{k}_\bb{C}$.
\end{prop}

\begin{proof}
Suppose $\frak{t}^{\sigma}=0$.  
Then $\sigma$ acts by $-1$ on $\frak{t}$ and hence on 
$\Delta(\frak{p}_\bb{C},\frak{t}_\bb{C})$.
Therefore if $\alpha(a)>0$
for some $\alpha\in\Delta(\frak{p}_\bb{C}, \frak{t}_\bb{C})$, 
then $\sigma\alpha(a)<0$.
By Theorem \ref{crit}, 
$A_\frak{q}(\lambda)$ is 
not discretely decomposable.
If $\alpha(a)=0$
for all $\alpha\in\Delta(\frak{p}_\bb{C},\frak{t}_\bb{C})$, 
then $\frak{p}_\bb{C}\subset\frak{q}$. 
Therefore $\frak{q}$ must coincide with $\frak{g}_\bb{C}$, 
which is not the case.

Finally, we note that $\frak{k}^\sigma+\sqrt{-1}\frak{k}^{-\sigma}$ is a 
real form of $\frak{k}_\bb{C}$, where $\sigma$ acts as 
a Cartan involution.
Thus, we see from Setting \ref{setq} (1) that $\frak{t}^\sigma=0$ 
if and only if $\frak{k}^\sigma+\sqrt{-1}\frak{k}^{-\sigma}$ is 
a split real form of $\frak{k}_\bb{C}$.
\end{proof}

The following proposition presents also 
 a necessary condition 
for $A_\frak{q}(\lambda)$ to be discretely decomposable, 
which is stronger than 
the one in Proposition \ref{split}.
Let $\alpha_0$ be the highest weight of the irreducible representation 
of $\frak{k}_\bb{C}$ on $\frak{p}_\bb{C}$ (if $\frak{g}$ is not of Hermitian) 
or on $\frak{p}_+$ (if $\frak{g}$ is of Hermitian).

\begin{prop}
\label{highlow}
Let $\frak{g}$ be a simple non-compact Lie algebra and 
$\sigma$ an involution of $\frak{g}$ commuting with $\theta$.
Suppose that $\frak{q}\neq\frak{g}_\bb{C}$, 
$\lambda$ is in the weakly fair range, 
and $A_\frak{q}(\lambda)$ is non-zero. 
If one of the following three assumptions hold, 
then $A_\frak{q}(\lambda)$ is not discretely decomposable 
as a $(\frak{g}^\sigma,K^\sigma)$-module. 
\begin{enumerate}
\item[$(1)$]
$\frak{g}$ is not of Hermitian type 
and $-\sigma\alpha_0$ is $\Delta^+(\frak{k}_\bb{C},\frak{t}_\bb{C})$-dominant. 
\item[$(2)$]
$\frak{g}$ is of Hermitian type, 
$(\frak{g},\frak{g}^{\sigma})$ is not of holomorphic type, 
and $-\sigma\alpha_0$ is $\Delta^+(\frak{k}_\bb{C},\frak{t}_\bb{C})$-dominant. 
\item[$(3)$]
$\frak{g}$ is of Hermitian type, 
$(\frak{g},\frak{g}^{\sigma})$ is of holomorphic type, 
$-\sigma\alpha_0$ is $\Delta^+(\frak{k}_\bb{C},\frak{t}_\bb{C})$-dominant, 
and $\frak{q}$ is neither holomorphic nor anti-holomorphic.
\end{enumerate}
\end{prop}

\begin{proof}
We assume that the parabolic subalgebra $\frak{q}$ is given by 
a $\Delta^+(\frak{k}_\bb{C},\frak{t}_\bb{C})$-dominant vector $a$.

(1): 
Since $-\sigma\alpha_0$ is an extremal weight of 
$\Delta(\frak{p}_\bb{C},\frak{t}_\bb{C})$ and  
$-\sigma\alpha_0$ is dominant, 
$-\sigma\alpha_0$ is the highest weight of $\frak{p}_\bb{C}$ and 
hence $-\sigma\alpha_0=\alpha_0$.
If $\alpha_0(a)\leq 0$, then $\alpha(a)\leq 0$ for all 
$\alpha\in \Delta(\frak{p}_\bb{C},\frak{t}_\bb{C})$. 
Hence $\alpha(a)=0$ for all 
$\alpha\in \Delta(\frak{p}_\bb{C},\frak{t}_\bb{C})$ and then 
$\frak{q}=\frak{g}_\bb{C}$,
contradicting our assumption.
We therefore have $\alpha_0(a)>0$ and 
$\sigma\alpha_0(a)=-\alpha_0(a)<0$ so the condition (iii) 
in Theorem \ref{crit} fails.

(2): 
Similarly to the proof of the case (1), 
$-\sigma\alpha_0$ must be the highest weight of $\frak{p}_+$ and 
hence $-\sigma\alpha_0=\alpha_0$.
Let $\alpha'_0$ be the highest weight of $\frak{p}_-$. 
Then $-\sigma\alpha'_0$ is dominant and hence $-\sigma\alpha'_0=\alpha'_0$.
If $\alpha_0(a)\leq 0$ and $\alpha'_0(a)\leq 0$, 
then $\alpha(a)\leq 0$ for all 
$\alpha\in \Delta(\frak{p}_\bb{C},\frak{t}_\bb{C})$. 
This implies 
$\frak{q}=\frak{g}_\bb{C}$,
contradicting our assumption.
Hence we must have $\alpha_0(a)>0$ or 
$\alpha'_0(a)> 0$.
If $\alpha_0(a)>0$, then $\sigma\alpha_0(a)<0$ so
the condition (iii) in Theorem \ref{crit} fails.
Similarly for the case $\alpha'_0(a)>0$.

(3): 
Similarly to the proof of the case (1), 
$\sigma\alpha_0$ is the lowest weight of $\frak{p}_+$.
If $\alpha_0(a)\leq 0$, then $\alpha(a)\leq 0$ for all 
$\alpha\in\Delta(\frak{p}_+, \frak{t}_\bb{C})$. 
This implies that 
$\frak{q}$ is anti-holomorphic, contradicting our assumption.
In the same way, $\sigma\alpha_0(a)\geq 0$ implies that $\frak{q}$ is 
holomorphic, a contradiction.
Therefore $\alpha_0(a)> 0$ and 
$\sigma\alpha_0(a)<0$ so the condition (iii) in 
Theorem \ref{crit} fails.
\end{proof}

The key assumption of Proposition \ref{highlow} is 
that $-\sigma\alpha_0$ is dominant.
In order to give a simple criterion to verify this, 
we consider the Satake diagram of 
the reductive Lie algebra 
$\frak{k}^{\sigma}+\sqrt{-1}\frak{k}^{-\sigma}$, 
which is a real form of $\frak{k}_\bb{C}$ 
(see \cite[Chapter X]{hel} for the Satake diagram).
Each vertex is associated to a simple root of 
$\Delta^+(\frak{k}_\bb{C},\frak{t}_\bb{C})$.
Then we add a vertex $\star$, indicating the 
highest weight $\alpha_0 (\in\frak{t}_\bb{C}^*)$ of 
$\frak{p}_\bb{C}$ or $\frak{p}_+$.
We connect this new vertex to the vertex associated to $\alpha_i$ 
if $\langle \alpha_0,\alpha_i\rangle >0$.
We can immediately tell whether $-\sigma\alpha_0$ is dominant 
from this diagram: 

\begin{prop}
\label{dom}
$-\sigma\alpha_0$ is $\Delta^+(\frak{k}_\bb{C},\frak{t}_\bb{C})$-dominant 
if and only if no black circle is connected to the new vertex $\star$.
\end{prop}
\begin{proof}
Write $\Delta^+=\Delta^+(\frak{k}_\bb{C},\frak{t}_\bb{C})$ for simplicity.
Suppose that the vertex $\star$ is connected to the black circle 
associated to a simple root $\alpha_i$.
Then 
$\langle -\sigma\alpha_0,\alpha_i \rangle
=-\langle \alpha_0,\sigma\alpha_i \rangle
=-\langle \alpha_0,\alpha_i \rangle<0$ and 
hence $-\sigma\alpha_0$ is not 
$\Delta^+$-dominant.

Conversely, assume that there is no black circle 
connected to the vertex $\star$.
Suppose that $\alpha\in -\sigma\Delta^+$. 
Then $-\sigma\alpha\in \Delta^+$ and hence 
$\langle -\sigma\alpha_0, \alpha\rangle
=\langle \alpha_0, -\sigma\alpha\rangle\geq 0$.
Suppose that $\alpha\in\Delta^+\setminus-\sigma\Delta^+$.
Since $\Delta^+$ is compatible with 
a positive restricted root system 
$\Sigma^+(\frak{k}_\bb{C},\sqrt{-1}\frak{t}^{-\sigma})$, 
it follows that 
$\sigma\alpha=\alpha$ and $\alpha$ can be written as a linear sum of 
roots associated to black circles.
Our assumption implies that $\alpha_0$ is orthogonal to 
any roots associated to black circle and hence orthogonal to $\alpha$.
Thus, $\langle -\sigma\alpha_0,\alpha\rangle\geq 0$ 
for all $\alpha\in \Delta^+$.
\end{proof}

Owing to Proposition \ref{dom}, 
we can classify all the pairs $(\frak{g},\frak{g}^{\sigma})$ such that 
$-\sigma\alpha_0$ is $\Delta^+(\frak{k}_\bb{C},\frak{t}_\bb{C})$-dominant
and $\frak{t}^{\sigma}\neq 0$.
See Appendix \ref{sec:diagram}
for the list of the diagrams for all such pairs.


\section{Discretely decomposable tensor product}
\label{sec:tensor}

The tensor product of two representations is a special case of 
the restriction with respect to a symmetric pair, namely, 
it is regarded as the restriction of an outer tensor product 
representation of the direct sum 
${\mathfrak {g}}={\mathfrak {g}}' \oplus {\mathfrak {g}}'$, 
when restricted to the subalgebra $\frak{g}^{\sigma}:={\rm diag}(\frak{g}')$.
In this section we discuss when the tensor product
of $(\frak{g}',K')$-modules 
$A_{\frak{q}'_1}(\lambda_1)\otimes A_{\frak{q}'_2}(\lambda_2)$ 
decomposes discretely. 
This is a branching problem of the $(\frak{g},K)$-module 
$A_\frak{q}(\lambda)$ with respect to 
$(\frak{g}^{\sigma},K^{\sigma})
:=({\rm diag}(\frak{g}'), {\rm diag}(K'))$,
 where 
$K=K'\times K'$, 
$
{\frak {q}}={\frak {q}'_1}\oplus {\frak {q}'_2}
$
 and $\lambda=(\lambda_1, \lambda_2)$.

\begin{thm}
\label{tensor}
Let $\frak{g}'$ be a non-compact simple Lie algebra.
Let $\frak{q}'_1$ and $\frak{q}'_2$ be $\theta$-stable parabolic 
subalgebras of $\frak{g}'_\bb{C}$, not equal to $\frak{g}'_\bb{C}$.
Then the following three conditions on $\frak{q}'_1$ and $\frak{q}'_2$ 
are equivalent.
\begin{itemize}
\item[$({\rm i})$]
The tensor product 
$A_{\frak{q}'_1}(\lambda_1)\otimes A_{\frak{q}'_2}(\lambda_2)$
is non-zero and 
discretely decomposable as a $(\frak{g}',K')$-module
for some $\lambda_1$ and $\lambda_2$
in the weakly fair range.
\item[$({\rm i}')$]
The tensor product 
$A_{\frak{q}'_1}(\lambda_1)\otimes A_{\frak{q}'_2}(\lambda_2)$
is discretely decomposable as a $(\frak{g}',K')$-module
for any $\lambda_1$ and $\lambda_2$
in the weakly fair range.
\item[$({\rm ii})$]
$\frak{g}'$ is of 
Hermitian type and 
both $\frak{q}'_1$ and $\frak{q}'_2$ are simultaneously holomorphic or 
anti-holomorphic. 
\end{itemize}
\end{thm}

\begin{proof}
Let $\frak{t}'$ be a Cartan subalgebra of $\frak{k}'$.
Fix a positive system $\Delta^+(\frak{k}'_\bb{C},\frak{t}'_\bb{C})$.
Suppose that $\frak{q}'_1$ and $\frak{q}'_2$
are given by 
$a_1\in\sqrt{-1}\frak{t}'$ and $a_2\in\sqrt{-1}\frak{t}'$, 
respectively.
We set ${\mathfrak {g}}={\mathfrak {g}}' \oplus {\mathfrak {g}}'$, 
 ${\mathfrak {k}}={\mathfrak {k}}' \oplus {\mathfrak {k}}'$,
 $\frak{t}=\frak{t}'\oplus\frak{t'}$,
 and 
 ${\mathfrak {q}}={\mathfrak {q}}_1' \oplus {\mathfrak {q}}_2'$.  
Then $\frak{t}$ is a Cartan subalgebra of $\frak{k}$ and
${\mathfrak {q}}$ is a $\theta$-stable parabolic subalgebra of ${\mathfrak {g}}_{\mathbb{C}}$.  
We define the involution $\sigma$ of $\frak{g}$ as 
$\sigma(x,y):=(y,x)$ for $x,y\in\frak{g}'$.
Let $\Delta^+(\frak{k}_\bb{C},\frak{t}_\bb{C})$ 
be the union of $\Delta^+(\frak{k}'_\bb{C},\frak{t}'_\bb{C})$ 
in the first factor and $-\Delta^+(\frak{k}'_\bb{C},\frak{t}'_\bb{C})$ 
in the second factor, so the condition of 
Setting \ref{setq} (1) is satisfied.
We assume that the defining element 
$a=(a_1,a_2)$ of $\frak{q}$ 
is $\Delta^+(\frak{k}_\bb{C},\frak{t}_\bb{C})$-dominant. 
This means that $a_1$ and $-a_2$ are
$\Delta^+(\frak{k}'_\bb{C},\frak{t}'_\bb{C})$-dominant.
Then the condition (iii) in Theorem \ref{crit} amounts to that 
\[
({\rm i}'')\quad
\alpha(a_2)\geq 0 
{\it \ \ whenever\ } \alpha\in\Delta(\frak{p}'_\bb{C},\frak{t}'_\bb{C})
 {\it\  satisfies\ } \alpha(a_1)>0.\]
(i$''$) implies that $\alpha(a_1)\geq 0$ 
whenever $\alpha\in\Delta(\frak{p}'_\bb{C},\frak{t}'_\bb{C})$
 satisfies $\alpha(a_2)>0$.

By Theorem \ref{crit}, 
it suffices to prove that (i$''$) is equivalent to (ii).

(ii) $\Rightarrow$ (i$''$):
This is similar to Proposition \ref{holdd}.
Suppose that $(\frak{g}',\frak{k}')$ is a Hermitian symmetric pair.
We assume $\frak{q}'_1$ and $\frak{q}'_2$ are holomorphic 
with respect to $\frak{p}'_\bb{C}=\frak{p}'_+ + \frak{p}'_-$.
If $\alpha\in\Delta(\frak{p}'_\bb{C},\frak{t}'_\bb{C})$ satisfies 
 $\alpha(a_1)>0$,  
then $\alpha\in\Delta(\frak{p}'_+,\frak{t}'_\bb{C})$.
Since $\Delta(\frak{p}'_+,\frak{t}'_\bb{C})\subset 
\Delta(\frak{q}'_2,\frak{t}_\bb{C})$, 
it follows that $\alpha(a_2)\geq 0$ and 
hence (i$''$) holds.
The same argument works when $\frak{q}'_1$ and $\frak{q}'_2$ are 
anti-holomorphic.

(i$''$) $\Rightarrow$ (ii):
Suppose that $(\frak{g}',\frak{k}')$ is not 
a Hermitian symmetric pair.  
This means that $\frak{k}'_\bb{C}$ is semisimple
 and acts irreducibly on 
$\frak{p}'_\bb{C}$ by the adjoint action.
Let us show $\alpha_0(a_1) >0$
 and $\alpha_0(a_2)<0$
 if $\alpha_0\in \Delta(\frak{p}'_\bb{C},\frak{t}'_\bb{C})$ is 
 the highest weight of $\frak{p}'_\bb{C}$
with respect to $\Delta^+(\frak{k}'_\bb{C},\frak{t}'_\bb{C})$.
First we observe that $\alpha_0(a_1)\geq 0$ because $a_1$ is dominant 
and $\frak{k}'_\bb{C}$ is semisimple.
If $\alpha_0(a_1)=0$, 
then we would have $\mathfrak {p}_{\mathbb{C}}' \subset {\mathfrak {q}}_1'$, 
 which would result in $\frak{q}'_1=\frak{g}'_\bb{C}$, 
 contradicting our assumption.
Therefore $\alpha_0(a_1)>0$.
In the same way, we have $\alpha_0(a_2)<0$.
Hence (i$''$) fails.

Suppose now that $(\frak{g}',\frak{k}')$ is a Hermitian 
symmetric pair and fix a decomposition 
$\frak{p}'_\bb{C}=\frak{p}'_+ + \frak{p}'_-$.
We assume that (i$''$) holds.
Let  $\alpha_0\in\Delta(\frak{p}'_+,\frak{t}'_\bb{C})$
be the highest weight 
of $\frak{p}'_+$ 
with respect to $\Delta^+(\frak{k}'_\bb{C},\frak{t}'_\bb{C})$.
Since $\Delta(\frak{p}'_+,\frak{t}'_\bb{C})
=-\Delta(\frak{p}'_-,\frak{t}'_\bb{C})$, 
we see that $-\alpha_0$ is the lowest weight of $\frak{p}'_-$.

Now we assume that $\frak{q}'_1$ is not anti-holomorphic, namely 
$\frak{p}'_-\not\subset \frak{q}'_1$.
Then $\alpha_0(a_1)>0$ 
because ${\mathfrak {p}}'_- \not \subset {\mathfrak {q}}'_1$
 and $\alpha_0$ is the highest weight.
Then (i$''$) 
implies that $\alpha_0(a_2)\geq 0$.
Since $-a_2$ is dominant, $\alpha(-a_2)\leq \alpha_0(-a_2)\leq 0$
for every $\alpha\in\Delta(\frak{p}'_+,\frak{t}'_\bb{C})$.
Therefore $\frak{q}'_2$ is holomorphic. 
In particular, $\frak{q}'_2$ is not anti-holomorphic, which in turn implies 
that $\frak{q}'_1$ is holomorphic by the same argument.

Likewise,
 if we assume that $\frak{q}'_1$ is not holomorphic, we see 
that $\frak{q}'_1$ and $\frak{q}'_2$ are anti-holomorphic.
\end{proof}

In view of Theorem \ref{aqds} and Theorem \ref{tensor}, 
we can tell when the tensor product of 
two discrete series representations decomposes discretely.

\begin{cor}
\label{dstensor}
Suppose that $V_1$ and $V_2$ are the underlying $(\frak{g}',K')$-modules of 
discrete series representations.
Then $V_1\otimes V_2$ is discretely decomposable as 
a $(\frak{g}',K')$-module if and only if 
they are simultaneously holomorphic (or anti-holomorphic) discrete series 
representations.
\end{cor}


\section{Classification of Discretely Decomposable $A_\frak{q}(\lambda)$}
\label{sec:class}

The classification of the triples $(\frak{g},\frak{g}^\sigma,\frak{q})$ 
goes as follows.
The tensor product case was treated in Section \ref{sec:tensor}. 
Consider the case where $\frak{g}$ is simple.
We fix a simple Lie algebra $\frak{g}$ with a Cartan involution $\theta$.
Suppose that $(\frak{g},\frak{g}^{\sigma_1},\frak{q}_1)$ and 
$(\frak{g},\frak{g}^{\sigma_2},\frak{q}_2)$ are triples  
such that ${\rm Ad}(k)\sigma_1{\rm Ad}(k^{-1})=\sigma_2$ 
and ${\rm Ad}(k')\frak{q}_1=\frak{q}_2$ for 
$k,k'\in K$.
Then there is an isomorphism 
$A_{\frak{q}_1}(\lambda_1)|_{\frak{g}^{\sigma_1}}
\simeq A_{\frak{q}_2}(\lambda_2)|_{\frak{g}^{\sigma_2}}$ 
via the isomorphism 
${\rm Ad}(k):\frak{g}^{\sigma_1}\to\frak{g}^{\sigma_2}$
if ${\rm Ad}^*(k')\lambda_1=\lambda_2$.
In this sense the branching problems 
with respect to $(\frak{g},\frak{g}^{\sigma_1},\frak{q}_1)$
and $(\frak{g},\frak{g}^{\sigma_2},\frak{q}_2)$
are equivalent. 
Thus, we will classify 
the triples $(\frak{g},\frak{g}^\sigma,\frak{q})$
with the discrete decomposability condition
up to the adjoint action ${\rm Ad}(K)\times {\rm Ad}(K)$. 
Here, $\sigma$ is an involution commuting with $\theta$ 
and $\frak{q}$ is a $\theta$-stable parabolic subalgebra 
of $\frak{g}_\bb{C}$.

Retain the notation and the assumption in Setting \ref{setq}. 
In particular, 
the parabolic subalgebra $\frak{q}$ is given by 
a $\Delta^+(\frak{k}_\bb{C},\frak{t}_\bb{C})$-dominant vector 
$a\in\sqrt{-1}\frak{t}$.
The classification of $(\frak{g},\frak{g}^\sigma,\frak{q})$ with
the discrete decomposability condition is given 
as conditions on the coordinates $a_i$ of $a$.

\begin{thm}
\label{ddclass}
Let $(\frak{g},\frak{g}^\sigma)$ be an 
irreducible symmetric pair such that $\sigma$ commutes 
with $\theta$ 
and let $\frak{q}$ be a $\theta$-stable 
parabolic subalgebra of $\frak{g}_\bb{C}$, 
not equal to $\frak{g}_\bb{C}$.
Suppose that $\lambda$ is in the weakly fair range and 
that $A_\frak{q}(\lambda)$ is non-zero.
Then $A_\frak{q}(\lambda)$ is discretely decomposable 
as a $(\frak{g}^\sigma,K^\sigma)$-module if and only if 
one of the following conditions on 
the triple $(\frak{g},\frak{g}^\sigma,\frak{q})$ holds.
\begin{enumerate}
\item[$(1)$]
$\frak{g}$ is compact.
\item[$(2)$]
$\sigma=\theta$.
\item[$(3)$]
$\frak{g}=\frak{g}'\oplus\frak{g}'$ and 
$\frak{q}=\frak{q}'_1\oplus\frak{q}'_2$.
Further, $\frak{g}'$ is of Hermitian type and 
both of the parabolic subalgebras
$\frak{q}'_1$ and $\frak{q}'_2$ of $\frak{g}'_\bb{C}$ are holomorphic, or 
they are anti-holomorphic 
(see Table \ref{holparablist} for 
holomorphic and anti-holomorphic parabolic subalgebras).
\item[$(4)$]
The symmetric pair $(\frak{g},\frak{g}^\sigma)$ 
is of holomorphic type (see Table \ref{holpairlist} for the classification) 
and the parabolic subalgebra $\frak{q}$ is either holomorphic 
or anti-holomorphic.
\item[$(5)$]
The triple $(\frak{g},\frak{g}^\sigma,\frak{q})$ is isomorphic to
one of those listed in Table \ref{dslist} or
in Table \ref{isolist}, 
where the parabolic subalgebra $\frak{q}$ is given by the conditions on $a$.
\end{enumerate}
In Tables \ref{holparablist}, \ref{dslist}, and \ref{isolist}, 
we have assumed that the defining element $a$ of $\frak{q}$ 
is dominant with respect to 
$\Delta^+(\frak{k}_\bb{C},\frak{t}_\bb{C})$ 
(see Appendix \ref{sec:setup} 
for concrete conditions on the coordinates of $a$) 
and list only additional conditions 
for the discrete decomposability.
\end{thm}

\begin{proof}
If $\frak{g}$ is compact, namely, 
if $\frak{g}$ is isomorphic to the Lie algebra of a compact Lie group, then 
the discrete decomposability follows obviously.

We divide irreducible symmetric pairs $(\frak{g},\frak{g}^{\sigma})$ 
into the following four cases.

\noindent
Case 1: In the tensor product case, a necessary and sufficient condition 
for the discrete decomposability 
was obtained in Theorem \ref{tensor}.

In the rest of the proof, we assume that $\frak{g}$ is non-compact and simple.

\noindent
Case 2: Suppose that $\frak{t}^\sigma=0$ or the assumption 
(1) or (2) in Proposition \ref{highlow} is satisfied 
for a symmetric pair $(\frak{g},\frak{g}^\sigma)$.
Then it follows from Propositions \ref{split} and \ref{highlow} 
that the triple $(\frak{g},\frak{g}^{\sigma},\frak{q})$ 
does not satisfy the discrete decomposability condition for 
any $\theta$-stable parabolic subalgebra $\frak{q}$ other than 
$\frak{g}_\bb{C}$. 

\noindent
Case 3: Suppose that  the assumption 
(3) in Proposition \ref{highlow} is satisfied 
for a symmetric pair $(\frak{g},\frak{g}^\sigma)$.
Then the triple $(\frak{g},\frak{g}^{\sigma},\frak{q})$ 
satisfies the discrete decomposability condition 
if and only if $\frak{q}$ is holomorphic or anti-holomorphic. 

We can verify which irreducible pairs $(\frak{g},\frak{g}^\sigma)$
belong to Case 2 or Case 3.
The condition $\frak{t}^{\sigma}=0$ holds if and only if 
$\frak{k}^{\sigma}+\sqrt{-1}\frak{k}^{-\sigma}$ is a split real form 
of $\frak{k}_\bb{C}$, so it is easily verified.
The case $\frak{t}^\sigma\neq 0$ is less easy.
We give a list of all the pairs 
$(\frak{g},\frak{g}^{\sigma})$ such that 
$-\sigma\alpha_0$ is $\Delta^+(\frak{k}_\bb{C},\frak{t}_\bb{C})$-dominant
and $\frak{t}^{\sigma}\neq 0$ in
Appendix \ref{sec:diagram}.
The verification of 
the dominancy of $-\sigma\alpha_0$ 
is reduced to a simple combinatorial problem by using the Satake diagram 
as we noted in the end of Section \ref{sec:dd}.

\noindent
Case 4: 
The classification of the triples $(\frak{g},\frak{g}^{\sigma},\frak{q})$ 
with the discrete decomposability condition 
for the remaining symmetric pairs $(\frak{g},\frak{g}^{\sigma})$ 
is more delicate.
For this, we apply the criterion, Theorem \ref{crit} ${\rm (iii)}$. 
This criterion reduces to simple computations for 
only the pair $(\frak{k},\frak{k}^\sigma)$ 
and the set of weights $\Delta(\frak{p}_\bb{C},\frak{t}_\bb{C})$.
We then carry out the computation in a case-by-case way.

To be more precise, we classify 
the $K$-conjugacy classes of symmetric pairs 
$(\frak{g},\frak{g}^{\sigma})$, 
building on Berger's classification of symmetric pairs (\cite{ber}). 
We postpone this until Section \ref{sec:Kconj}.

In Setting \ref{setq}, 
we gave a symmetric pair $(\frak{g},\frak{g}^\sigma)$, followed by  
the choice of a Cartan subalgebra $\frak{t}$ of $\frak{k}$ and 
a positive system $\Delta^+(\frak{k}_\bb{C},\frak{t}_\bb{C})$
that satisfy the compatibility condition with respect to $\sigma$ and 
finally we set a $\theta$-stable parabolic subalgebra $\frak{q}$ 
given by a $\Delta^+(\frak{k}_\bb{C},\frak{t}_\bb{C})$-dominant 
vector $a\in\sqrt{-1}\frak{t}$.
In the following, however, we do this in a different order.
We fix $\frak{t}$ and $\Delta^+(\frak{k}_\bb{C},\frak{t}_\bb{C})$ 
before $\sigma$ is given. 
This does not lose the generality 
because all the pairs $(\frak{t},\Delta^+(\frak{k}_\bb{C},\frak{t}_\bb{C}))$ 
are $K$-conjugate (recall that 
we treat $\sigma$ and $\frak{q}$ up to $K\times K$-conjugacy).
Then choose $\sigma$ that
satisfies the conditions in Setting \ref{setq} (1) 
with respect to 
$(\frak{t},\Delta^+(\frak{k}_\bb{C},\frak{t}_\bb{C}))$.
Each $K$-conjugacy class of 
$\theta$-stable parabolic subalgebras of $\frak{g}_\bb{C}$ 
has a unique representative $\frak{q}$ 
which is given by a dominant vector $a\in\sqrt{-1}\frak{t}$.

Let $\frak{g}$ be a non-compact simple Lie algebra.
Choose coordinates $e_i$ of $\frak{t}_\bb{C}$ and 
write the defining element $a$ 
of $\frak{q}$ as $a=\sum a_i e_i$ 
(see Appendix \ref{sec:setup}).
Fix a positive system $\Delta^+(\frak{k}_\bb{C},\frak{t}_\bb{C})$. 
We assume that $a$ is 
$\Delta^+(\frak{k}_\bb{C},\frak{t}_\bb{C})$-dominant.
For a given $K$-conjugacy class of symmetric pairs 
$(\frak{g},\frak{g}^\sigma)$, 
we choose a representative $\sigma$ that satisfies 
the conditions in Setting $\ref{setq}$ (1).
We describe the restriction of $\sigma$ to $\frak{t}$ 
and then the $\sigma$-action on the set of weights 
$\Delta(\frak{p}_\bb{C},\frak{t}_\bb{C})$.
Now the condition Theorem \ref{crit} ${\rm (iii)}$ amounts to conditions on 
the coordinates $a_i$.

We illustrate computations in the following two examples.
Other cases are verified similarly. 
\end{proof}

\begin{ex}
{\rm 
Let $(\frak{g},\frak{g}^\sigma)
  =(\frak{su}(m,n),\,\frak{su}(m,k)\oplus\frak{su}(n-k)\oplus \frak{u}(1))$
for $k,n-k\geq 1$.
We fix $\frak{t}$, $\{\epsilon_i\}$, and $\{e_i\}$ as in Setting \ref{sumn}.
Choose $\sigma$ that satisfies the conditions in Setting \ref{setq} (1), 
so the restriction of $\sigma$ to $\frak{t}_\bb{C}$ can be written as 
$\sigma(e_i)=e_{\sigma(i)}$
for $1\leq i \leq m+n$, 
where 
\begin{align*}
&\sigma(i)=i & &\quad{\rm for}\ 1\leq i\leq m,\\
&\sigma(m+j)=m+n-j+1 & &\quad{\rm for}\ 1\leq j\leq \min\{k,n-k\} 
 \ {\rm or}\ \max\{k,n-k\}< j\leq n, \\
&\sigma(m+j)=m+j & &\quad{\rm for}\ \min\{k,n-k\}< j \leq \max\{k,n-k\}.
\end{align*}
Suppose that $\frak{q}$ is given by a dominant vector 
$a=a_1e_1+\dots +a_{m+n}e_{m+n}\in\sqrt{-1}\frak{t}$, namely 
$a_1\geq \dots \geq a_m$ 
and $a_{m+1}\geq \dots \geq a_{m+n}$ as in Setting \ref{sumn}.
If the condition $({\rm iii})$ in Theorem \ref{crit} 
is satisfied, 
then $a_i-a_{m+n}>0$ implies $a_i-a_{m+1}\geq 0$ for 
$1\leq i\leq m$.
As a consequence, we see that the triple 
$(\frak{su}(m,n),\,
 \frak{su}(m,k)\oplus\frak{su}(n-k)\oplus \frak{u}(1),
 \,\frak{q})$ satisfies the discrete decomposability condition
if and only if  
\begin{enumerate}
\item[(1)] $a_{m+n}\geq a_1$, 
\item[(2)] there exists an integer $1\leq l\leq m-1$ such that \\
$a_1\geq \cdots \geq a_l \geq a_{m+1}\geq \cdots \geq a_{m+n}
\geq a_{l+1}\geq \cdots \geq a_{m}$, or 
\item[(3)] $a_m\geq a_{m+1}$.
\end{enumerate}
These triples are listed in Table \ref{dslist}.
}
\end{ex}

\begin{ex}
{\rm 
Let $(\frak{g},\frak{g}^\sigma)
  =(\frak{f}_{4(-20)},\,\frak{so}(8,1))$.
Here, the exceptional Lie algebra 
$\frak{f}_{4(-20)}$ is a real form of $\frak{f}_4^\bb{C}$ with real rank one.
We fix $\frak{t}$, $\{\epsilon_i\}$, and $\{e_i\}$ as in Setting \ref{f42}.
Choose $\sigma$ that satisfies the conditions in Setting \ref{setq} (1), 
so the restriction of $\sigma$ to $\frak{t}_\bb{C}$ can be written as 
\begin{align*}
&\sigma e_1=-e_1,\\
&\sigma e_{i}=e_{i}\quad\quad \ {\rm for}\ 
 2\leq i\leq  4.
\end{align*}
Suppose that $\frak{q}$ is given by 
$a=a_1e_1+\dots +a_{4}e_{4}\in\sqrt{-1}\frak{t}$, namely 
$a_1\geq \dots \geq a_4\geq 0$ as in Setting \ref{f42}. 
If the condition $({\rm iii})$ in Theorem \ref{crit} is satisfied, then 
$\frac{1}{2}(\epsilon_1+\epsilon_2-\epsilon_3-\epsilon_4)(a)\leq 0$ 
or $\frac{1}{2}(-\epsilon_1+\epsilon_2-\epsilon_3-\epsilon_4)(a) \geq 0$.
The former implies $a_1=a_2=a_3=a_4$ and 
the latter implies $a_1=a_2\geq a_3=a_4=0$.
Hence the triple
 $(\frak{f}_{4(-20)},\,\frak{so}(8,1),\,\frak{q})$
 satisfies the discrete decomposability condition
if and only if 
$(a_1,a_2,a_3,a_4)=(s,s,s,s)$ or $(s,s,0,0)$ for $s\geq 0$.
These triples are listed in Table \ref{isolist}.
}
\end{ex}

From our classification result, we see that: 

\begin{cor}
\label{borel}
In the setting of Theorem \ref{ddclass}, 
suppose that $\frak{q}$ is a Borel subalgebra of $\frak{g}_\bb{C}$.
If $A_\frak{q}(\lambda)$ is discretely decomposable as a 
$(\frak{g}^{\sigma},K^{\sigma})$-module, 
then $\sigma=\theta$ or 
${\rm rank}\,\frak{g}_\bb{C}={\rm rank}\,\frak{k}_\bb{C}$.
In particular, 
$A_\frak{q}(\lambda)$ is isomorphic to the underlying 
$(\frak{g},K)$-module of a discrete series representation 
in the latter case as far as $\lambda$ is in the good range.
\end{cor}

\begin{rem}
{\rm 
The triples $(\frak{g},\frak{g}^\sigma,\frak{q})$ in 
Table \ref{dslist} have the following property: 
there exists a $\theta$-stable
Borel subalgebra $\frak{b}$ contained in $\frak{q}$
such that $(\frak{g},\frak{g}^\sigma,\frak{b})$ also satisfies the 
discrete decomposability condition.
This is also the case for (1), (2), (3), and (4) in Theorem \ref{ddclass}.
Then Proposition \ref{incl} implies that every $\theta$-stable 
parabolic subalgebra containing $\frak{b}$ satisfies the discrete 
decomposability condition.
We call such triples $(\frak{g},\frak{g}^\sigma,\frak{q})$ 
{\it discrete series type}. 
The triples in Table \ref{dslist} together with (1), (2), (3), and (4) in 
Theorem \ref{ddclass} give all triples of discrete series type.
}
\end{rem}

\begin{rem}
{\rm 
The remaining case is (5) in Theorem \ref{ddclass} for Table \ref{isolist}.
We call triples $(\frak{g},\frak{g}^\sigma,\frak{q})$ in Table \ref{isolist} 
{\it isolated type}.
For generic $m,n$ and $k$, discrete series type and isolated type are
exclusive.
However, for particular $m,n$ or $k$ there may be overlaps (see
Remark \ref{overlap} (6), (7)).
}
\end{rem}

\begin{rem}
\label{overlap}
{\rm 
We did not intend to write the cases 
$(1)$ to $(5)$ in Theorem \ref{ddclass} 
in an exclusive way. 
Also there are some overlaps among the tables.
What follows from (2) to (5) below show overlaps between 
the cases $(4)$ and $(5)$ in Theorem \ref{ddclass}.
(6) and (7) discuss some 
overlaps between Tables \ref{dslist} and \ref{isolist}.

\begin{itemize}
\item[(1)]
Table \ref{holpairlist} includes the case 
$\sigma=\theta$ with $\frak{g}$ Hermitian.

\item[(2)]
The symmetric pair 
$(\frak{su}(m,n),\,\frak{su}(m,k)\oplus\frak{su}(n-k)\oplus\frak{u}(1))$ 
in Table \ref{dslist} is of holomorphic type and 
the parabolic subalgebra $\frak{q}$ is holomorphic (anti-holomorphic) if 
$a_m\geq a_{m+1}$ ($a_{m+n}\geq a_1$).

\item[(3)]
The symmetric pairs 
$(\frak{so}(2m,2n),\,\frak{so}(2m,k)\oplus\frak{so}(2n-k))$ and 
$(\frak{so}(2m,2n+1),\,\frak{so}(2m,k)\oplus\frak{so}(2n-k+1))$ 
in Table \ref{dslist}
are of holomorphic type if $m=1$ and 
$\frak{q}$ is holomorphic or anti-holomorphic if 
$m=1,\ |a_1|\geq |a_2|$.

\item[(4)]
The symmetric pair 
$(\frak{so}(2m,2n),\,\frak{u}(m,n))$ in Table \ref{dslist}
is of holomorphic type if $m=1$ or $n=1$ and 
the parabolic subalgebra $\frak{q}$ is holomorphic or anti-holomorphic if 
$m=1,\ a=a_1e_1$ or $n=1,\ a=a_{m+1}e_{m+1}$.

\item[(5)]
The symmetric pairs 
$(\frak{so}^*(2n),\,\frak{so}^*(2n-2)\oplus\frak{so}(2))$ and 
$(\frak{so}^*(2n),\,\frak{u}(n-1,1))$
in Table \ref{isolist}
are of holomorphic type and 
the parabolic subalgebra $\frak{q}$ is holomorphic or anti-holomorphic if 
$k=1$ or $n-1$.

\item[(6)]
The triple
$(\frak{so}(2m,2n),\,\frak{u}(m,n),\frak{q})$ for 
$m=2$ and $a=a_1e_1$ in Table \ref{isolist}
is also listed in Table \ref{dslist}.

\item[(7)]
The triple
$(\frak{su}(2m,2n),\,\frak{sp}(m,n),\frak{q})$ for 
$m=1$ and 
$(a_1,a_2;\,a_3,\dots,a_{2n+2})
=(s,0;\,t,0,\dots,0) (s\geq t)$, $(0,-s;\,0,\dots,0,-t) (s\geq t)$, 
$(s,-t;\,0,\dots,0)(s,t\geq 0) \mod \bb{I}_{2n+2}$
in Table \ref{isolist}
are also listed in Table \ref{dslist}.

\item[(8)]
There are also coincidences of Lie algebras with small rank such as 
$\frak{sp}(2,\bb{R})\simeq\frak{so}(2,3)$, 
$\frak{so}(2,4)\simeq\frak{su}(2,2)$, 
$\frak{so}(3,3)\simeq\frak{sl}(4,\bb{R})$,
and 
$\frak{so}^*(6)\simeq \frak{su}(1,3)$.
\end{itemize}
}
\end{rem}

\begin{rem}
\label{exh}
{\rm 
Our classification of the triples $(\frak{g},\frak{g}^\sigma,\frak{q})$ 
is carried out up to $K \times K$-conjugacy 
as we noted in the beginning of this section. 
In some cases, there exist 
more than one $K$-conjugacy classes 
of $(\frak{g},\frak{g}^\sigma)$ for a given 
Lie algebra isomorphism class of $\frak{g}^\sigma$.
To save space, we did not distinguish some of different 
$K$-conjugacy classes in the tables 
if the discrete decomposability conditions with respect to them are the same.
For given Lie algebras $\frak{g}$ and $\frak{g}'$, we define 
${\cal S}$, ${\cal T}$ and $\phi$ as we shall explain in (\ref{Krest}).
The elements of ${\cal S}$ correspond to the $K$-conjugacy classes of 
involutions $\sigma$ of $\frak{g}$ such that $\theta\sigma=\sigma\theta$
and $\frak{g}^\sigma$ is isomorphic to $\frak{g}'$.
The discrete decomposability condition only depends on 
their images in ${\cal T}$ by $\phi$.
\begin{itemize}
\item[(1)]
Let $\frak{g}=\frak{so}^*(8)$ and 
$\frak{g}'=\frak{so}^*(2)\oplus\frak{so}^*(6)\simeq \frak{u}(1,3)$.
There are two $K$-conjugacy classes of 
 involutions $\sigma$
such that $\theta\sigma=\sigma\theta$ and 
$\frak{g}^\sigma$ is isomorphic to $\frak{g}'$, 
and one is the associated pair of the other.
We list the two pairs 
$(\frak{so}^*(8),\frak{so}^*(2)\oplus\frak{so}^*(6))$ 
and 
$(\frak{so}^*(8),\frak{u}(1,3))$
in Table \ref{holpairlist}, 
which are not $K$-conjugate to each other.
This is the case (1) in Proposition \ref{Kconj}.
\item[(2)]
Let $\frak{g}=\frak{su}(m,n)$.
Among two types of symmetric pairs 
$(\frak{su}(m,n),\frak{su}(m,k)\oplus\frak{su}(n-k)\oplus\frak{u}(1))$
and $(\frak{su}(m,n),\frak{su}(m-k)\oplus\frak{su}(k,n)\oplus\frak{u}(1))$, 
we list only the former type in Tables \ref{dslist} and \ref{isolist} 
because the latter type can be treated by interchanging $m$ and $n$.
Similarly for $\frak{g}=\frak{so}(m,n)$ or $\frak{g}=\frak{sp}(m,n)$.
\item[(3)]
Let $\frak{g}=\frak{so}(2m,2n)$ and $\frak{g}'=\frak{u}(m,n)$.
Consider $K$-conjugacy classes of involutions $\sigma$
such that $\theta\sigma=\sigma\theta$ and 
$\frak{g}^\sigma$ is isomorphic to $\frak{g}'$.
Then $\frak{g}^{\theta\sigma}$ is also isomorphic to 
$\frak{g}'$.

If both $m$ and $n$ are odd, then 
$|{\rm Im}\ \phi|$=1 and $|{\cal S}|=2$.
This is the case (1) in Proposition \ref{Kconj}.

If $m$ is even and $n$ is odd, then 
$|{\rm Im}\ \phi|$=2 and $|{\cal S}|=2$.
For every $\sigma$, the case (3) in Proposition \ref{Kconj}
occurs.
The same holds if $m$ is odd and $n$ is even.

If both $m$ and $n$ are even, then 
$|{\rm Im}\ \phi|$=4 and $|{\cal S}|=4$.
For every $\sigma$, the case (3) in Proposition \ref{Kconj}
occurs.

It turns out that 
the discrete decomposability condition depends on the 
$K$-conjugacy classes of $\sigma$ only if $m=2$ (or $n=2$). 
For $m=2$ and $n\neq 2$, we write in Table \ref{dslist} as 
\begin{align*}
&\frak{g}^{\sigma}=\frak{u}(2,n)_1 \quad {\rm if}\ 
\sigma(e_1)=-e_2,\\
&\frak{g}^{\sigma}=\frak{u}(2,n)_2 \quad {\rm if}\ 
\sigma(e_1)=e_2,
\end{align*}
(if $m=2$ and $n=2k$ $(k>1)$, there are four $K$-conjugacy classes, 
so we group them two and two).
For $m=n=2$, we write in Table \ref{dslist} as 
\begin{align*}
&\frak{g}^{\sigma}=\frak{u}(2,2)_{11} \quad {\rm if}\ 
\sigma(e_1)=-e_2\ {\rm and}\ \sigma(e_3)=-e_4, \\
&\frak{g}^{\sigma}=\frak{u}(2,2)_{12} \quad {\rm if}\ 
\sigma(e_1)=-e_2\ {\rm and}\ \sigma(e_3)=e_4, \\
&\frak{g}^{\sigma}=\frak{u}(2,2)_{21} \quad {\rm if}\ 
\sigma(e_1)=e_2\ {\rm and}\ \sigma(e_3)=-e_4, \\
&\frak{g}^{\sigma}=\frak{u}(2,2)_{22} \quad {\rm if}\ 
\sigma(e_1)=e_2\ {\rm and}\ \sigma(e_3)=e_4.
\end{align*}
\end{itemize}
}
\end{rem}

Let $\frak{g}$ be a simple non-compact Lie algebra 
and $\sigma(\neq \theta)$ an involution commuting with $\theta$.
We illustrate by examples how to obtain all $\theta$-stable parabolic subalgebras $\frak{q}$ 
of $\frak{g}_\bb{C}$ such that $(\frak{g},\frak{g}^\sigma,\frak{q})$ 
satisfy the discrete decomposability condition.

\begin{ex}
{\rm 
Let $(\frak{g},\frak{g}^\sigma)=(\frak{so}(4,2),\frak{u}(2,1))$.
Fix a Cartan subalgebra $\frak{t}$ of 
$\frak{k}=\frak{so}(4)\oplus\frak{so}(2)$, 
a positive system $\Delta^+(\frak{k}_\bb{C},\frak{t}_\bb{C})$, 
and a basis of $\frak{t}_\bb{C}$ as in Setting \ref{somn}.
We use the letters $a'_i,e'_i$ instead of $a_i,e_i$ in Setting \ref{somn}.
Suppose that $\frak{g}^\sigma=\frak{u}(2,1)_1$ in the notation of 
Remark $\ref{exh}$.
We assume $\frak{q}$ is given by 
a $\Delta^+(\frak{k}_\bb{C},\frak{t}_\bb{C})$-dominant vector 
$a=a'_1e'_1+a'_2e'_2+a'_3e'_3$ as in Setting \ref{setq} (2).
According to Table \ref{holpairlist}, the pair $(\frak{g},\frak{g}^\sigma)$ is of holomorphic type.
Hence all holomorphic or anti-holomorphic parabolic subalgebras $\frak{q}$ 
satisfy the discrete decomposability condition.
According to Table \ref{holparablist}, $\frak{q}$ is holomorphic or anti-holomorphic if and only if $|a'_3|\geq a'_1$.
The pair $(\frak{so}(4,2),\frak{u}(2,1)_1)$ is listed in Table \ref{dslist}.
This says that $\frak{q}$ satisfies the discrete decomposability condition 
if $-a'_2\geq|a'_3|$.
The pair $(\frak{so}(4,2),\frak{u}(2,1))$ is also listed 
in Table \ref{isolist}.
This says that $\frak{q}$ satisfies the discrete decomposability condition 
if $a=a'_1e'_1$ or $a=a'_3e'_3$.
We also have an isomorphism 
$$
(\frak{so}(4,2),\frak{u}(2,1))\simeq
(\frak{su}(2,2),\frak{su}(2,1)\oplus\frak{su}(1)\oplus\frak{u}(1)).
$$
Regard $\frak{t}$ as a Cartan subalgebra of $\frak{su}(2,2)$ and 
define $a_i$ and $e_i$ as in Setting \ref{sumn}.
Then we have 
\[
a'_1=\frac{1}{2}(a_1-a_2+a_3-a_4),\ a'_2=\frac{1}{2}(-a_1+a_2+a_3-a_4),\ 
a'_3=\frac{1}{2}(a_1+a_2-a_3-a_4).\]
The pair $(\frak{su}(2,2),\frak{su}(2,1)\oplus\frak{su}(1)\oplus\frak{u}(1))$
is listed in Table \ref{dslist} and in Table \ref{isolist}. 
However, it turns out that no parabolic subalgebra other than 
that obtained in the previous argument 
satisfies the discrete decomposability condition.
As a consequence, 
a parabolic subalgebra 
$\frak{q}$ satisfies the discrete decomposability condition 
if and only if $\frak{q}$ is given by $a$ for 
$|a'_3|\geq a'_1$ or $-a'_2\geq|a'_3|$ 
under the assumptions in Settings \ref{setq} and \ref{somn}.
They correspond to 
$X_1,X_3,X_6,Y_1,Y_2,Y_5,Y_6,Z_1,Z_2,Z_3,Z_4,W$, or $U$. 
In all cases, the triples $(\frak{g},\frak{g}^\sigma,\frak{q})$
are of discrete series type.
}
\end{ex}

\begin{figure}[H]
\setlength{\unitlength}{0.0004in}
\begingroup\makeatletter\ifx\SetFigFont\undefined%
\gdef\SetFigFont#1#2#3#4#5{%
  \reset@font\fontsize{#1}{#2pt}%
  \fontfamily{#3}\fontseries{#4}\fontshape{#5}%
  \selectfont}%
\fi\endgroup%
{\renewcommand{\dashlinestretch}{30}
\begin{picture}(6696,5160)(0,-10)
\put(6348,4788){\ellipse{680}{680}}
\put(3338,1200){\ellipse{680}{680}}
\put(348,4788){\ellipse{680}{680}}
\put(648,347){\ellipse{680}{680}}
\put(1623,2013){\blacken\ellipse{140}{140}}
\put(1623,2013){\ellipse{140}{140}}
\put(3348,2013){\blacken\ellipse{140}{140}}
\put(3348,2013){\ellipse{140}{140}}
\put(5073,2013){\blacken\ellipse{140}{140}}
\put(5073,2013){\ellipse{140}{140}}
\put(2748,4798){\ellipse{680}{680}}
\put(971,2696){\ellipse{680}{680}}
\put(2138,2696){\ellipse{680}{680}}
\put(3373,2696){\ellipse{680}{680}}
\put(4569,2696){\ellipse{680}{680}}
\put(348,3741){\ellipse{680}{680}}
\put(6348,3741){\ellipse{680}{680}}
\put(5769,2696){\ellipse{680}{680}}
\put(5168,3741){\ellipse{680}{680}}
\put(1623,3741){\ellipse{680}{680}}
\path(1469,3421)(1152,2997)
\path(1469,3421)(1152,2997)
\path(492,3417)(773,2963)
\path(492,3417)(773,2963)
\path(601,3514)(1894,2933)
\path(601,3514)(1894,2933)
\path(2721,3591)(2340,2975)
\path(2721,3591)(2340,2975)
\path(1829,3466)(3081,2918)
\path(1829,3466)(3081,2918)
\path(1943,2401)(1692,2122)
\path(1943,2401)(1692,2122)
\path(1182,2423)(1519,2107)
\path(1182,2423)(1519,2107)
\path(1757,1955)(3072,1420)
\path(1757,1955)(3072,1420)
\path(2786,3586)(3201,3001)
\path(2786,3586)(3201,3001)
\path(3948,3563)(3509,3016)
\path(3948,3563)(3509,3016)
\path(4927,3466)(3637,2948)
\path(4927,3466)(3637,2948)
\path(6089,3516)(4824,2930)
\path(6089,3516)(4824,2930)
\path(4023,3563)(4404,2990)
\path(4023,3563)(4404,2990)
\path(4772,2412)(4992,2137)
\path(4772,2412)(4992,2137)
\path(6224,3426)(5953,2997)
\path(6224,3426)(5953,2997)
\path(5537,2442)(5187,2122)
\path(5537,2442)(5187,2122)
\path(4929,1970)(3582,1431)
\path(4929,1970)(3582,1431)
\path(5316,3406)(5586,2986)
\path(5316,3406)(5586,2986)
\path(1551,4445)(2756,3926)
\path(1551,4445)(2756,3926)
\path(2745,4448)(1859,4020)
\path(2745,4448)(1859,4020)
\path(4919,3983)(2745,4448)
\path(4919,3983)(2745,4448)
\path(4098,3938)(5133,4436)
\path(4098,3938)(5133,4436)
\path(5153,4443)(6157,4026)
\path(5153,4443)(6157,4026)
\path(3342,1890)(3348,1556)
\path(3342,1890)(3348,1556)
\path(6348,4438)(6348,4100)
\path(6348,4438)(6348,4100)
\path(3948,4442)(3948,3938)
\path(3948,4442)(3948,3938)
\path(2811,3921)(3936,4439)
\path(2811,3921)(3936,4439)
\path(348,4448)(348,4100)
\path(348,4448)(348,4100)
\path(552,4020)(1545,4448)
\path(552,4020)(1545,4448)
\path(5148,4436)(5148,4096)
\path(5148,4436)(5148,4096)
\path(1548,4445)(1548,4096)
\path(1548,4445)(1548,4096)
\path(3348,2333)(3348,2146)
\path(3348,2333)(3348,2146)
\put(1423,3663){\makebox(0,0)[lb]{\smash{{{\SetFigFont{8}{9.6}{\familydefault}{\mddefault}{\updefault}$Y_2$}}}}}
\put(2623,3663){\makebox(0,0)[lb]{\smash{{{\SetFigFont{8}{9.6}{\familydefault}{\mddefault}{\updefault}$Y_3$}}}}}
\put(3823,3663){\makebox(0,0)[lb]{\smash{{{\SetFigFont{8}{9.6}{\familydefault}{\mddefault}{\updefault}$Y_4$}}}}}
\put(5023,3663){\makebox(0,0)[lb]{\smash{{{\SetFigFont{8}{9.6}{\familydefault}{\mddefault}{\updefault}$Y_5$}}}}}
\put(6223,3663){\makebox(0,0)[lb]{\smash{{{\SetFigFont{8}{9.6}{\familydefault}{\mddefault}{\updefault}$Y_6$}}}}}
\put(823,2613){\makebox(0,0)[lb]{\smash{{{\SetFigFont{8}{9.6}{\familydefault}{\mddefault}{\updefault}$Z_1$}}}}}
\put(2023,2613){\makebox(0,0)[lb]{\smash{{{\SetFigFont{8}{9.6}{\familydefault}{\mddefault}{\updefault}$Z_2$}}}}}
\put(3223,2613){\makebox(0,0)[lb]{\smash{{{\SetFigFont{8}{9.6}{\familydefault}{\mddefault}{\updefault}$W$}}}}}
\put(4423,2613){\makebox(0,0)[lb]{\smash{{{\SetFigFont{8}{9.6}{\familydefault}{\mddefault}{\updefault}$Z_3$}}}}}
\put(5623,2613){\makebox(0,0)[lb]{\smash{{{\SetFigFont{8}{9.6}{\familydefault}{\mddefault}{\updefault}$Z_4$}}}}}
\put(3223,1113){\makebox(0,0)[lb]{\smash{{{\SetFigFont{8}{9.6}{\familydefault}{\mddefault}{\updefault}$U$}}}}}
\put(223,3663){\makebox(0,0)[lb]{\smash{{{\SetFigFont{8}{9.6}{\familydefault}{\mddefault}{\updefault}$Y_1$}}}}}
\put(223,4713){\makebox(0,0)[lb]{\smash{{{\SetFigFont{8}{9.6}{\familydefault}{\mddefault}{\updefault}$\!X_1$}}}}}
\put(1423,4713){\makebox(0,0)[lb]{\smash{{{\SetFigFont{8}{9.6}{\familydefault}{\mddefault}{\updefault}$\!X_2$}}}}}
\put(2623,4713){\makebox(0,0)[lb]{\smash{{{\SetFigFont{8}{9.6}{\familydefault}{\mddefault}{\updefault}$\!X_3$}}}}}
\put(3823,4713){\makebox(0,0)[lb]{\smash{{{\SetFigFont{8}{9.6}{\familydefault}{\mddefault}{\updefault}$\!X_4$}}}}}
\put(5023,4713){\makebox(0,0)[lb]{\smash{{{\SetFigFont{8}{9.6}{\familydefault}{\mddefault}{\updefault}$X_5$}}}}}
\put(6223,4713){\makebox(0,0)[lb]{\smash{{{\SetFigFont{8}{9.6}{\familydefault}{\mddefault}{\updefault}$\!X_6$}}}}}
\put(1098,313){\makebox(0,0)[lb]{\smash{{{\SetFigFont{10}{12.0}{\familydefault}{\mddefault}{\updefault}: $A_{\mathfrak{q}}(\lambda)|_{\mathfrak{u}(2,1)}$ is discretely decomposable}}}}}
\end{picture}
}
\caption{}
\label{fig2}
\end{figure}

\begin{ex}
{\rm 
Let $(\frak{g},\frak{g}^\sigma)=(\frak{su}(2,2),\frak{sp}(1,1))$.
Fix a Cartan subalgebra $\frak{t}$ of 
$\frak{k}=\frak{su}(2)\oplus\frak{su}(2)\oplus\frak{u}(1)$, 
a positive system $\Delta^+(\frak{k}_\bb{C},\frak{t}_\bb{C})$, 
and $e_i$ as in Setting \ref{sumn}.
We assume $\frak{q}$ is given by 
$\Delta^+(\frak{k}_\bb{C},\frak{t}_\bb{C})$-dominant vector 
$a=a_1e_1+a_2e_2+a_3e_3+a_4e_4$ as in Setting \ref{setq} (2).
The pair $(\frak{su}(2,2),\frak{sp}(1,1))$ is listed in Table \ref{dslist}.
This says that $\frak{q}$ satisfies the discrete decomposability condition 
if $a_1\geq a_3\geq a_4\geq a_2$ or $a_3\geq a_1\geq a_2\geq a_4$.
The pair $(\frak{su}(2,2),\frak{sp}(1,1))$ is also listed 
in Table \ref{isolist} and 
we have an isomorphism 
$$
(\frak{su}(2,2),\frak{sp}(1,1))\simeq
(\frak{so}(4,2),\frak{so}(4,1)).
$$
It turns out that $\frak{q}$ satisfies the discrete decomposability 
condition if and only if $a_1\geq a_3\geq a_4\geq a_2$ or 
$a_3\geq a_1\geq a_2\geq a_4$ 
under the assumptions in Settings \ref{setq} and \ref{sumn}.
They correspond to 
$X_3,X_4,Y_2,Y_3,Y_4,Y_5,Z_1,Z_2,Z_3,Z_4,W$, or $U$. 
In all cases, the triples $(\frak{g},\frak{g}^\sigma,\frak{q})$
are of discrete series type.
}
\end{ex}

\begin{figure}[H]
\setlength{\unitlength}{0.0004in}
\begingroup\makeatletter\ifx\SetFigFont\undefined%
\gdef\SetFigFont#1#2#3#4#5{%
  \reset@font\fontsize{#1}{#2pt}%
  \fontfamily{#3}\fontseries{#4}\fontshape{#5}%
  \selectfont}%
\fi\endgroup%
{\renewcommand{\dashlinestretch}{30}
\begin{picture}(6137,5150)(0,-10)
\put(3763,3741){\ellipse{680}{680}}
\put(2563,3741){\ellipse{680}{680}}
\put(4963,3741){\ellipse{680}{680}}
\put(1363,3741){\ellipse{680}{680}}
\put(2563,4788){\ellipse{680}{680}}
\put(3763,4788){\ellipse{680}{680}}
\put(3115,1200){\ellipse{680}{680}}
\put(3115,2696){\ellipse{680}{680}}
\put(1915,2696){\ellipse{680}{680}}
\put(4315,2696){\ellipse{680}{680}}
\put(5515,2696){\ellipse{680}{680}}
\put(715,2696){\ellipse{680}{680}}
\put(425,347){\ellipse{680}{680}}
\put(1400,2013){\blacken\ellipse{140}{140}}
\put(1400,2013){\ellipse{140}{140}}
\put(3125,2013){\blacken\ellipse{140}{140}}
\put(3125,2013){\ellipse{140}{140}}
\put(4850,2013){\blacken\ellipse{140}{140}}
\put(4850,2013){\ellipse{140}{140}}
\path(1222,3424)(913,2981)
\path(1222,3424)(913,2981)
\path(269,3417)(535,2990)
\path(269,3417)(535,2990)
\path(378,3514)(1671,2933)
\path(378,3514)(1671,2933)
\path(2401,3442)(2117,2975)
\path(2401,3442)(2117,2975)
\path(1585,3481)(2854,2920)
\path(1585,3481)(2854,2920)
\path(1720,2401)(1477,2116)
\path(1720,2401)(1477,2116)
\path(959,2423)(1297,2116)
\path(959,2423)(1297,2116)
\path(1534,1955)(2849,1420)
\path(1534,1955)(2849,1420)
\path(2686,3418)(2977,3004)
\path(2686,3418)(2977,3004)
\path(3619,3430)(3265,2995)
\path(3619,3430)(3265,2995)
\path(4747,3484)(3379,2911)
\path(4747,3484)(3379,2911)
\path(5960,3560)(4601,2930)
\path(5960,3560)(4601,2930)
\path(3901,3415)(4181,2990)
\path(3901,3415)(4181,2990)
\path(4537,2419)(4769,2131)
\path(4537,2419)(4769,2131)
\path(6077,3548)(5730,2997)
\path(6077,3548)(5730,2997)
\path(5314,2442)(4964,2108)
\path(5314,2442)(4964,2108)
\path(4706,1970)(3359,1431)
\path(4706,1970)(3359,1431)
\path(5098,3418)(5363,2986)
\path(5098,3418)(5363,2986)
\path(1328,4445)(2350,4003)
\path(1328,4445)(2350,4003)
\path(2522,4448)(1564,4012)
\path(2522,4448)(1564,4012)
\path(4723,3973)(2522,4448)
\path(4723,3973)(2522,4448)
\path(3988,3991)(4910,4436)
\path(3988,3991)(4910,4436)
\path(4930,4443)(6044,3977)
\path(4930,4443)(6044,3977)
\path(6125,4438)(6125,4100)
\path(6125,4438)(6125,4100)
\path(4925,4436)(4924,4090)
\path(4925,4436)(4924,4090)
\path(3725,4442)(3724,4084)
\path(3725,4442)(3724,4084)
\path(2770,4000)(3713,4439)
\path(2770,4000)(3713,4439)
\path(1325,4445)(1324,4081)
\path(1325,4445)(1324,4081)
\path(125,4448)(125,4100)
\path(125,4448)(125,4100)
\path(329,4020)(1322,4448)
\path(329,4020)(1322,4448)
\path(3125,2359)(3125,2138)
\path(3125,2359)(3125,2138)
\path(3125,1883)(3125,1556)
\path(3125,1883)(3125,1556)
\put(1200,3663){\makebox(0,0)[lb]{\smash{{{\SetFigFont{8}{9.6}{\familydefault}{\mddefault}{\updefault}$Y_2$}}}}}
\put(2400,3663){\makebox(0,0)[lb]{\smash{{{\SetFigFont{8}{9.6}{\familydefault}{\mddefault}{\updefault}$Y_3$}}}}}
\put(3600,3663){\makebox(0,0)[lb]{\smash{{{\SetFigFont{8}{9.6}{\familydefault}{\mddefault}{\updefault}$Y_4$}}}}}
\put(4800,3663){\makebox(0,0)[lb]{\smash{{{\SetFigFont{8}{9.6}{\familydefault}{\mddefault}{\updefault}$Y_5$}}}}}
\put(6000,3663){\makebox(0,0)[lb]{\smash{{{\SetFigFont{8}{9.6}{\familydefault}{\mddefault}{\updefault}$Y_6$}}}}}
\put(600,2613){\makebox(0,0)[lb]{\smash{{{\SetFigFont{8}{9.6}{\familydefault}{\mddefault}{\updefault}$Z_1$}}}}}
\put(3000,2613){\makebox(0,0)[lb]{\smash{{{\SetFigFont{8}{9.6}{\familydefault}{\mddefault}{\updefault}$W$}}}}}
\put(4200,2613){\makebox(0,0)[lb]{\smash{{{\SetFigFont{8}{9.6}{\familydefault}{\mddefault}{\updefault}$Z_3$}}}}}
\put(5400,2613){\makebox(0,0)[lb]{\smash{{{\SetFigFont{8}{9.6}{\familydefault}{\mddefault}{\updefault}$Z_4$}}}}}
\put(3000,1113){\makebox(0,0)[lb]{\smash{{{\SetFigFont{8}{9.6}{\familydefault}{\mddefault}{\updefault}$U$}}}}}
\put(0,3663){\makebox(0,0)[lb]{\smash{{{\SetFigFont{8}{9.6}{\familydefault}{\mddefault}{\updefault}$Y_1$}}}}}
\put(0,4713){\makebox(0,0)[lb]{\smash{{{\SetFigFont{8}{9.6}{\familydefault}{\mddefault}{\updefault}$\!X_1$}}}}}
\put(1200,4713){\makebox(0,0)[lb]{\smash{{{\SetFigFont{8}{9.6}{\familydefault}{\mddefault}{\updefault}$\!X_2$}}}}}
\put(3600,4713){\makebox(0,0)[lb]{\smash{{{\SetFigFont{8}{9.6}{\familydefault}{\mddefault}{\updefault}$\!X_4$}}}}}
\put(4800,4713){\makebox(0,0)[lb]{\smash{{{\SetFigFont{8}{9.6}{\familydefault}{\mddefault}{\updefault}$X_5$}}}}}
\put(6000,4713){\makebox(0,0)[lb]{\smash{{{\SetFigFont{8}{9.6}{\familydefault}{\mddefault}{\updefault}$\!X_6$}}}}}
\put(1800,2613){\makebox(0,0)[lb]{\smash{{{\SetFigFont{8}{9.6}{\familydefault}{\mddefault}{\updefault}$Z_2$}}}}}
\put(2400,4713){\makebox(0,0)[lb]{\smash{{{\SetFigFont{8}{9.6}{\familydefault}{\mddefault}{\updefault}$\!X_3$}}}}}
\put(875,313){\makebox(0,0)[lb]{\smash{{{\SetFigFont{10}{12.0}{\familydefault}{\mddefault}{\updefault}: $A_{\mathfrak{q}}(\lambda)|_{\mathfrak{sp}(1,1)}$ is discretely decomposable}}}}}
\end{picture}
}
\caption{}
\label{fig3}
\end{figure}

We see from our classification that, in most cases,
 the center of $L=N_G(\frak{q})$ is contained in $K$
 if $A_\frak{q}(\lambda)$ is discretely decomposable as
 a $(\frak{g}^\sigma,K^\sigma)$-module.
We classify the cases where
 $L$ has a split center, or equivalently
 $\lambda$ can be non-zero on $\frak{l}\cap \frak{p}$.
\begin{cor}
Let $(\frak{g},\frak{g}^\sigma)$ be an irreducible symmetric pair
 such that $\sigma\neq \theta$.
Suppose that $A_\frak{q}(\lambda)$ is non-zero
 and discretely decomposable as
 a $(\frak{g}^\sigma,K^\sigma)$-module with $\lambda$
 in the weakly fair range.
Then $L$ has a split center if and only if 
\begin{align*}
(\frak{g},\frak{g}^{\sigma},\frak{l})
=&(\frak{sl}(2n,\bb{C}),\, \frak{sp}(n,\bb{C}),\, 
\frak{sl}(2n-1,\bb{C})\oplus\bb{C}),\\ 
&(\frak{sl}(2n,\bb{C}),\, \frak{su}^*(2n),\, 
\frak{sl}(2n-1,\bb{C})\oplus\bb{C}),\\ 
&(\frak{so}(2n,\bb{C}),\, \frak{so}(2n-1,\bb{C}),\, 
\frak{sl}(n,\bb{C})\oplus\bb{C}), \ {\rm or}\\
&(\frak{so}(2n,\bb{C}),\, \frak{so}(2n-1,1),\, 
\frak{sl}(n,\bb{C})\oplus\bb{C}).
\end{align*}
\end{cor}

\begin{proof}
If ${\rm rank}\,\frak{g}_\bb{C}={\rm rank}\,\frak{k}_\bb{C}$, then 
a fundamental Cartan subalgebra $\frak{h}$ of $\frak{l}$ is 
contained in $\frak{k}$.
In this case, 
 the center of $L$ is contained in $K$.

Suppose that ${\rm rank}\,\frak{g}_\bb{C}>{\rm rank}\,\frak{k}_\bb{C}$.
Then Theorem \ref{ddclass} implies that
 the pair $(\frak{g},\frak{g}^\sigma)$ is
 isomorphic to
 $(\frak{so}(2m+1,2n+1),\, \frak{so}(2m+1,k)\oplus\frak{so}(2n-k+1)),
 (\frak{sl}(2n,\bb{C}),\, \frak{sp}(n,\bb{C}))$,
 $(\frak{sl}(2n,\bb{C}),\, \frak{su}^*(2n))$,
 $(\frak{so}(2n,\bb{C}),\, \frak{so}(2n-1,\bb{C}))$, or
 $(\frak{so}(2n,\bb{C}),\, \frak{so}(2n-1,1))$.

Let $(\frak{g},\frak{g}^\sigma)
=(\frak{so}(2m+1,2n+1),\, \frak{so}(2m+1,k)\oplus\frak{so}(2n-k+1))$. 
By Theorem \ref{ddclass}, we may assume that
 $\frak{q}$ is given by
 $a$ with $a_{m+1}=\cdots=a_{m+n}=0$ (see Table \ref{isolist}).
Then $\frak{l}$ is a direct sum of $\frak{so}(2l-1,2n+1)(l\geq 1)$
 and compact factors.
Hence the center of $L$ is contained in $K$.

For the remaining four pairs $(\frak{g},\frak{g}^\sigma)$, 
 we have $\frak{l}=\frak{sl}(2n-1,\bb{C})\oplus\bb{C}$
 if $\frak{g}=\frak{sl}(2n,\bb{C})$ and
 $\frak{l}=\frak{sl}(n,\bb{C})\oplus\bb{C}$
 if $\frak{g}=\frak{so}(2n,\bb{C})$ 
 (see Table \ref{isolist}). 
Therefore $L$ has a split center in these cases.
\end{proof}

As another consequence of Theorem \ref{ddclass}, 
we get all the pairs $(\frak{g},\frak{g}^\sigma)$ 
which do not have 
discretely decomposable restrictions $A_\frak{q}(\lambda)|_{\frak{g}^\sigma}$.
We use the notation of \cite[Chapter X]{hel} for 
exceptional Lie algebras.

\begin{thm}
\label{notdd}
Let $(\frak{g},\frak{g}^\sigma)$ be an irreducible symmetric pair
such that $\frak{g}$ is non-compact and 
that $\sigma$ $(\neq \theta)$ commutes with $\theta$.
The following two conditions on the pair $(\frak{g},\frak{g}^\sigma)$ 
are equivalent.
\begin{enumerate}
\item[{\rm (i)}]
There is no $\theta$-stable parabolic subalgebra 
$\frak{q}$ $(\neq \frak{g}_\bb{C})$ such that 
the triple $(\frak{g},\frak{g}^\sigma,\frak{q})$ satisfies 
the discrete decomposability condition (see Theorem \ref{crit}).
\item[{\rm (ii)}]
One of the following cases occurs.
\begin{enumerate}
\item[$(1)$]
$\frak{g}\simeq\frak{g}'\oplus\frak{g}'$ with 
$\frak{g}'$ not of Hermitian type.
\item[$(2)$]
The simple Lie algebra 
$\frak{g}$ is isomorphic to 
$\frak{sl}(n,\bb{R}) (n\geq 5)$, 
$\frak{so}(1, n)$, 
$\frak{su}^*(2n)$, 
$\frak{sl}(2n+1,\bb{C})$, 
$\frak{so}(2n+1,\bb{C})$, 
$\frak{sp}(n,\bb{C})$, 
$\frak{g}_{2(2)}$,
$\frak{e}_{6(6)}$,
$\frak{e}_{6(-26)}$,
$\frak{e}_{7(7)}$, 
$\frak{e}_{8(8)}$,
$\frak{g}_{2}^\bb{C}$,
$\frak{f}_{4}^\bb{C}$,
$\frak{e}_{6}^\bb{C}$,
$\frak{e}_{7}^\bb{C}$,
or 
$\frak{e}_{8}^\bb{C}$.
\item[$(3)$]
$\frak{k}^\sigma+\sqrt{-1}\frak{k}^{-\sigma}$ is 
a split real form of $\frak{k}_\bb{C}$.
\item[$(4)$]
The pair $(\frak{g},\frak{g}^\sigma)$ is isomorphic to one 
of those listed in 
Table \ref{notddpair}.
\end{enumerate}
\end{enumerate}
\end{thm}


\section{$K$-conjugacy classes of Reductive Symmetric Pairs}
\label{sec:Kconj}

In \cite{ber}, irreducible symmetric pairs 
$(\frak{g},\frak{g}^\sigma)$ are classified 
up to outer automorphisms of $\mathfrak {g}$.
For our purpose,
 we need its refinement.
To classify the $K$-conjugacy classes of $(\frak{g},\frak{g}^\sigma)$, 
we have to tell
whether or not two symmetric pairs 
$(\frak{g},\frak{g}^{\sigma_1})$ and $(\frak{g},\frak{g}^{\sigma_2})$ 
 are $K$-conjugate to each other
 when $\frak{g}^{\sigma_1}$ is isomorphic to $\frak{g}^{\sigma_2}$ 
by an outer automorphism of ${\mathfrak {g}}$.

For this, we fix a reductive Lie algebra $\frak{g}'$ with a 
Cartan decomposition $\frak{g}'=\frak{k}'+\frak{p}'$.
Denote by ${\cal S}\equiv {\cal S}(\frak{g},\frak{g}')$ 
the set of $K$-conjugacy classes of involutions $\sigma$ of 
$\frak{g}$ such that $\sigma$ commutes with $\theta$ and 
that there is an isomorphism 
$\varphi:\frak{g}^\sigma\xrightarrow{\sim}\frak{g}'$ of Lie algebras 
with $\varphi(\frak{k}^\sigma)=\frak{k}'$.
Similarly, denote by ${\cal T}\equiv {\cal T}(\frak{k},\frak{k}')$ 
the set of $K$-conjugacy classes of involutions $\sigma$ of 
$\frak{k}$ such that $\frak{k}^\sigma$ is isomorphic to ${\frak{k}'}$.  
We allow the case where $\sigma$ is the identity
 in the definition of ${\cal T}$.  
Then the restriction $\sigma \mapsto \sigma|_{\frak {k}}$
 induces a map:
\begin{align}
\label{Krest}
\phi:{\cal S}\to {\cal T}.  
\end{align} 
The aim of this section is to classify the set ${\cal S}$.
This is carried out by studying ${\cal T}$ and $\phi$.

First let us study ${\cal T}\equiv{\cal T}(\frak{k},\frak{k}')$.
There is a one-to-one correspondence between 
${\cal T}$
and the set of $K$-conjugacy
classes of real forms 
$\frak{k}^{\sigma}+\sqrt{-1}\frak{k}^{-\sigma}$ of $\frak{k}_\bb{C}$ 
such that $\frak{k}^\sigma\simeq\frak{k}'$.  
Therefore the elements of ${\cal T}$ correspond to 
the Satake diagrams 
of real forms $\frak{k}_0$ 
of $\frak{k}_\bb{C}$ such that a maximal compact subalgebra of $\frak{k}_0$
is isomorphic to $\frak{k}'$.
For a simple compact Lie algebra $\frak{k}$, 
we see from the list of Satake diagrams (\cite[Chapter X]{hel}) that 
\begin{align*}
|{\cal T}|=3 \quad {\rm if}\ 
(\frak{k},\frak{k}')
\simeq
&(\frak{so}(8),\frak{u}(4))\simeq(\frak{so}(8),\frak{so}(2)\oplus\frak{so}(6)), \\
&(\frak{so}(8),\frak{so}(7)), \\
&(\frak{so}(8),\frak{so}(3)\oplus\frak{so}(5)), \\
|{\cal T}|=2 \quad {\rm if}\ 
(\frak{k},\frak{k}')
\simeq
&(\frak{so}(4n),\frak{u}(2n))\quad (n\geq 3),
\end{align*}
and $|{\cal T}|\leq 1$ if otherwise.
For $\frak{k}$ not simple, there may exist 
outer automorphisms which interchange simple factors.
In such a case, 
 the cardinality of ${\cal T}$ may also be greater than one.

Second we study the map $\phi$.
\begin{prop}
\label{Kconj}
Let $x\in {\cal T}$.
Suppose that the fiber $\phi^{-1}(x)$ is non-empty.
Choose an involution $\sigma$ of $\frak{g}$ which represents 
an element of $\phi^{-1}(x)$.
Then one of the following three cases occurs.
\begin{itemize}
\item[(1)]
$|\phi^{-1}(x)|=2$ and $\{\sigma, \theta\sigma\}$
is a complete set of representatives of $\phi^{-1}(x)$.
In particular, $\frak{g}^\sigma$ and $\frak{g}^{\theta\sigma}$ 
are isomorphic as Lie algebras, but they are not $K$-conjugate to each other.
\item[(2)]
$|\phi^{-1}(x)|=1$ and $\frak{g}^\sigma$
is not isomorphic to $\frak{g}^{\theta\sigma}$
as a Lie algebra.
\item[(3)]
$|\phi^{-1}(x)|=1$ and $\frak{g}^\sigma$ is $K$-conjugate to 
$\frak{g}^{\theta\sigma}$.
\end{itemize}
\end{prop}

Let $y_1, y_2\in\phi^{-1}(x)$.
We can choose two involutions 
$\sigma_1$ and $\sigma_2$ of $\frak{g}$
which represent $y_1$ and $y_2$,  
respectively, such that
$\frak{k}^{\sigma_1}=\frak{k}^{\sigma_2}$.
Therefore, the proof of Proposition \ref{Kconj} reduces to the following lemma.

\begin{lem}
\label{inv}
Let $\sigma_1$ and $\sigma_2$ be involutions of a 
simple Lie algebra $\frak{g}$ 
that commute with a Cartan involution $\theta$.
If $\frak{k}^{\sigma_1}=\frak{k}^{\sigma_2}$, 
then $\sigma_1$ is $K$-conjugate to 
$\sigma_2$ or $\sigma_1=\theta\sigma_2$.
\end{lem}

\begin{proof}
[Proof of the Lemma \ref{inv}.]
Since $\sigma_1=\sigma_2$ on $\frak{k}$, 
the composition $\tau=\sigma_1\sigma_2$ is an automorphism of $\frak{g}$ 
that is the identity map on $\frak{k}$.
Then the restriction 
$\tau|_\frak{p}:\frak{p}\to \frak{p}$
is an isomorphism of ${\rm ad}(\frak{k})$-modules.

Suppose that $\frak{g}$ is not of Hermitian type. 
Then $\frak{p}_\bb{C}$ is a simple $\frak{k}$-module. 
Hence $\tau$ acts on $\frak{p}$ as a scalar.
Because $\tau=1$ on $\frak{k}$ and 
$[\frak{p},\frak{p}]=\frak{k}$, we have 
$\tau=1$ or $-1$ on $\frak{p}$.
Therefore $\sigma_1=\sigma_2$ or
$\sigma_1=\theta\sigma_2$.

Suppose that $\frak{g}$ is of Hermitian type. 
Then $\frak{p}_\bb{C}$ decomposes as a 
$\frak{k}$-module: $\frak{p}_\bb{C}=\frak{p}_++\frak{p}_-$.
We extend $\tau$ to a complex linear automorphism of $\frak{g}_\bb{C}$ 
and use the same letter.
Since $\frak{p}_+$ and $\frak{p}_-$ are non-isomorphic 
simple $\frak{k}$-modules, 
there are constants $c_+,c_-\in\bb{C}$ such that 
$\tau=c_+$ on $\frak{p}_+$ 
and $\tau=c_-$ on $\frak{p}_-$.
In light that $[\frak{p}_+,\frak{p}_-]=\frak{k}$ and 
$\tau=1$ on $\frak{k}$, we have $c_+c_-=1$.
We write $\overline{\frak{p}_+}$ for 
the complex conjugate of $\frak{p}_+$ with respect to 
the real form $\frak{g}$.
Since $\overline{\frak{p}_+}=\frak{p}_-$
and $\tau$ commutes with the complex conjugates, 
we have $\overline{c_+}=c_-$.
Let $z\in\frak{z}_K$ be a non-zero element of 
the center of $\frak{k}$.
Then we can write $\tau={\rm Ad}(\exp(tz))$ for $t\in\bb{R}$.
Since $\sigma_1\tau=\sigma_2$ is an involution, 
it follows that 
$\tau^{-1}=\sigma_1\tau\sigma_1={\rm Ad}(\exp(t\sigma_1 z))$. 
If the symmetric pair $(\frak{g},\frak{g}^{\sigma_1})$ is 
of holomorphic type, then $\sigma_1 z=z$ and 
hence $\tau=\tau^{-1}$. 
Therefore, $c_+=1$ or $-1$ and it follows that 
$\sigma_1=\sigma_2$ or $\sigma_1=\sigma_2\theta$.
If the symmetric pair $(\frak{g},\frak{g}^{\sigma_1})$ is 
not of holomorphic type, then $\sigma_1 z=-z$. 
In this case, 
${\rm Ad}(\exp(-tz/2))\sigma_1{\rm Ad}(\exp(tz/2))=\sigma_2$, so
$\sigma_1$ is $K$-conjugate to $\sigma_2$.
\end{proof}

When $\frak{g}^{\sigma}$ is isomorphic to 
$\frak{g}^{\theta\sigma}$ as a Lie algebra, 
we use a case-by-case analysis to tell 
whether 
$\sigma$ and $\theta\sigma$ are $K$-conjugate 
and we conclude that: 

\begin{prop}
For a symmetric pair $(\frak{g},\frak{g}')$ with $\frak{g}$ simple, 
Proposition \ref{Kconj} (1) occurs if and only if 
$(\frak{g},\frak{g}')$ is isomorphic to 
$(\frak{so}(4m+2,4n+2),\frak{u}(2m+1,2n+1))$.
\end{prop}


\appendix
\section{Setup for $\theta$-stable Parabolic Subalgebras}
\label{sec:setup}

In this appendix we fix a positive system
 $\Delta^+(\frak{k}_\bb{C},\frak{t}_\bb{C})$ with respect to
 a Cartan subalgebra $\frak{t}_\bb{C}$ of $\frak{k}_\bb{C}$
 and present the set of weights $\Delta(\frak{p}_\bb{C},\frak{t}_\bb{C})$
 for each simple Lie algebra $\frak{g}$. 
We also write down the conditions for $a\in\sqrt{-1}\frak{t}$ to be
 $\Delta^+(\frak{k}_\bb{C},\frak{t}_\bb{C})$-dominant
 in terms of the coordinates $a_i$, which are used in
 Tables \ref{holparablist}, \ref{dslist}, and \ref{isolist}.

In what follows, we do not include $\frak{g}$ that has no non-trivial
 triple $(\frak{g},\frak{g}^{\sigma},\frak{q})$ satisfying the
 discrete decomposability condition
 (see Theorem \ref{notdd} (2)).
We will define $\epsilon_i\in\frak{t}_\bb{C}^*$
 and $e_i\in\frak{t}_\bb{C}$.
If $\frak{g}$ is not equal to $\frak{su}(m,n)$,
 $\frak{sl}(2n,\bb{C})$, $\frak{e}_{6(2)}$, or $\frak{e}_{7(-25)}$,
 then $\{\epsilon_i\}$ is a basis of $\frak{t}_\bb{C}^*$
 and $\{e_i\}$ is a dual basis of $\{\epsilon_i\}$.

\begin{set}
\label{sumn}
{\rm
Let $\frak{g}=\frak{su}(m,n)$.
Choose $\epsilon_1,\dots,\epsilon_{m+n} \in \frak{t}_\bb{C}^*$
such that 
\begin{align*}
&\Delta^+(\frak{k}_\bb{C},\frak{t}_\bb{C})
 =\{\epsilon_i-\epsilon_j\}_{1\leq i < j  \leq m}
 \cup\{\epsilon_{m+i}-\epsilon_{m+j}\}_{1\leq i < j \leq n},\\
&\Delta(\frak{p}_\bb{C},\frak{t}_\bb{C})
 =\{\pm(\epsilon_i-\epsilon_{m+j})\}
 _{1\leq i\leq m,\;1\leq j \leq n}.
\end{align*}
Define $e_1,\dots, e_{m+n}\in\frak{t}_\bb{C}$ such that 
$(\epsilon_i-\epsilon_j)(e_k)=\delta_{ik}-\delta_{jk}$
and then $e_1+\dots+e_{m+n}=0$.
The dominant condition on $a=a_1e_1+\cdots+a_{m+n}e_{m+n}
\in \sqrt{-1}\frak{t}$ amounts to that 
$a_1\geq a_2\geq \cdots\geq a_{m}$ and 
$a_{m+1}\geq a_{m+2}\geq \cdots\geq a_{m+n}$.
}
\end{set}

\begin{set}
\label{somn}
{\rm
Let $\frak{g}=\frak{so}(2m,2n)$.
Choose $\epsilon_1,\dots,\epsilon_{m+n} \in \frak{t}_\bb{C}^*$
such that 
\begin{align*}
&\Delta^+(\frak{k}_\bb{C},\frak{t}_\bb{C})
=\{\epsilon_i\pm\epsilon_j\}_{1\leq i < j \leq m}
\cup\{\epsilon_{m+i}\pm\epsilon_{m+j}\}_{1\leq i<j\leq n},\\
&\Delta(\frak{p}_\bb{C},\frak{t}_\bb{C})
=\{\pm\epsilon_i\pm\epsilon_{m+j}\}
_{1\leq i \leq m, \;1\leq j \leq n}.
\end{align*}
Denote by $e_1,\dots, e_{m+n}\in\frak{t}_\bb{C}$ the 
dual basis of $\epsilon_1,\dots,\epsilon_{m+n}$.
The dominant condition on $a=a_1e_1+\cdots+a_{m+n}e_{m+n}
\in \sqrt{-1}\frak{t}$ amounts to that 
$a_1\geq \cdots\geq a_{m-1}\geq |a_m|$ and 
$a_{m+1}\geq \cdots\geq a_{m+n-1}\geq |a_{m+n}|$.}
\end{set}

\begin{set}
{\rm
Let $\frak{g}=\frak{so}(2m,2n+1)$.
Choose $\epsilon_1,\dots,\epsilon_{m+n} \in \frak{t}_\bb{C}^*$
such that 
\begin{align*}
&\Delta^+(\frak{k}_\bb{C},\frak{t}_\bb{C})
=\{\epsilon_i\pm\epsilon_j\}_{1\leq i < j \leq m}
\cup\{\epsilon_{m+i}\pm\epsilon_{m+j}\}_{1\leq i<j\leq n}
\cup\{\epsilon_{m+i}\}_{1\leq i\leq n},\\
&\Delta(\frak{p}_\bb{C},\frak{t}_\bb{C})
=\{\pm\epsilon_i\pm\epsilon_{m+j}\}
_{1\leq i \leq m, \;1\leq j \leq n}
\cup\{\pm\epsilon_i\}_{1\leq i\leq m}.
\end{align*}
Denote by $e_1,\dots, e_{m+n}\in\frak{t}_\bb{C}$ the 
dual basis of $\epsilon_1,\dots,\epsilon_{m+n}$.
The dominant condition on $a=a_1e_1+\cdots+a_{m+n}e_{m+n}
\in \sqrt{-1}\frak{t}$ amounts to that 
$a_1\geq \cdots\geq a_{m-1}\geq |a_m|$ and 
$a_{m+1}\geq \cdots\geq a_{m+n}\geq 0$.
}
\end{set}

\begin{set}
{\rm 
Let $\frak{g}=\frak{so}(2m+1,2n)$.
Choose $\epsilon_1,\dots,\epsilon_{m+n} \in \frak{t}_\bb{C}^*$
such that 
\begin{align*}
&\Delta^+(\frak{k}_\bb{C},\frak{t}_\bb{C})
=\{\epsilon_i\pm\epsilon_j\}_{1\leq i < j \leq m}
\cup\{\epsilon_{m+i}\pm\epsilon_{m+j}\}_{1\leq i<j\leq n}
\cup\{\epsilon_i\}_{1\leq i\leq m},\\
&\Delta(\frak{p}_\bb{C},\frak{t}_\bb{C})
=\{\pm\epsilon_i\pm\epsilon_{m+j}\}
_{1\leq i \leq m, \;1\leq j \leq n}
\cup\{\pm\epsilon_{m+i}\}_{1\leq i\leq n}.
\end{align*}
Denote by $e_1,\dots, e_{m+n}\in\frak{t}_\bb{C}$ the 
dual basis of $\epsilon_1,\dots,\epsilon_{m+n}$.
The dominant condition on 
$a=a_1e_1+\cdots+a_{m+n}e_{m+n}\in \sqrt{-1}\frak{t}$ 
amounts to that $a_1\geq \cdots\geq a_m \geq 0$ and 
$a_{m+1}\geq \cdots\geq a_{m+n-1}\geq |a_{m+n}|$.
}
\end{set}

\begin{set}
{\rm 
Let $\frak{g}=\frak{so}(2m+1,2n+1)$.
Choose $\epsilon_1,\dots,\epsilon_{m+n} \in \frak{t}_\bb{C}^*$
such that 
\begin{align*}
&\Delta^+(\frak{k}_\bb{C},\frak{t}_\bb{C})
=\{\epsilon_i\pm\epsilon_j\}_{1\leq i < j \leq m}
\cup\{\epsilon_{m+i}\pm\epsilon_{m+j}\}_{1\leq i<j\leq n}
\cup\{\epsilon_{i}\}_{1\leq i\leq m}
\cup\{\epsilon_{m+i}\}_{1\leq i\leq n},\\
&\Delta(\frak{p}_\bb{C},\frak{t}_\bb{C})
=\{\pm\epsilon_i\pm\epsilon_{m+j}\}
_{1\leq i \leq m, \;1\leq j \leq n}
\cup\{\pm\epsilon_{i}\}_{1\leq i\leq m}
\cup\{\pm\epsilon_{m+i}\}_{1\leq i\leq n}
\cup\{0\}.
\end{align*}
Denote by $e_1,\dots, e_{m+n}\in\frak{t}_\bb{C}$ the 
dual basis of $\epsilon_1,\dots,\epsilon_{m+n}$.
The dominant condition on $a=a_1e_1+\cdots+a_{m+n}e_{m+n}
\in \sqrt{-1}\frak{t}$ amounts to that 
$a_1\geq \cdots\geq a_m \geq 0$ and 
$a_{m+1}\geq \cdots\geq a_{m+n} \geq 0$.
}
\end{set}

\begin{set}
\label{spmn}
{\rm
Let $\frak{g}=\frak{sp}(m,n)$.
Choose $\epsilon_1,\dots,\epsilon_{m+n} \in \frak{t}_\bb{C}^*$
such that 
\begin{align*}
&\Delta^+(\frak{k}_\bb{C},\frak{t}_\bb{C})\!
=\!\{\epsilon_i\pm\epsilon_j\}_{1\leq i < j \leq m}
\!\cup\!\{\epsilon_{m+i}\pm\epsilon_{m+j}\}_{1\leq i<j\leq n}
\!\cup\!\{2\epsilon_{i}\}_{1\leq i\leq m}
\!\cup\!\{2\epsilon_{m+i}\}_{1\leq i\leq n},\\
&\Delta(\frak{p}_\bb{C},\frak{t}_\bb{C})
=\{\pm\epsilon_i\pm\epsilon_{m+j}\}
_{1\leq i \leq m, \;1\leq j \leq n}.
\end{align*}
Denote by $e_1,\dots, e_{m+n}\in\frak{t}_\bb{C}$ the 
dual basis of $\epsilon_1,\dots,\epsilon_{m+n}$.
The dominant condition on $a=a_1e_1+\cdots+a_{m+n}e_{m+n}
\in \sqrt{-1}\frak{t}$ amounts to that 
$a_1\geq \cdots\geq a_{m}\geq 0$
and $a_{m+1}\geq \cdots\geq a_{m+n}\geq 0$.
}
\end{set}

\begin{set}
\label{sostar}
{\rm
Let $\frak{g}=\frak{so}^*(2n)$.
Choose $\epsilon_1,\dots,\epsilon_{n} \in \frak{t}_\bb{C}^*$
such that 
\begin{align*}
&\Delta^+(\frak{k}_\bb{C},\frak{t}_\bb{C})
=\{\epsilon_i-\epsilon_j\}_{1\leq i < j \leq n},\\
&\Delta(\frak{p}_\bb{C},\frak{t}_\bb{C})
=\{\pm(\epsilon_i+\epsilon_j)\}
_{1\leq i < j \leq n}.
\end{align*}
Denote by $e_1,\dots, e_{n}\in\frak{t}_\bb{C}$ the 
dual basis of $\epsilon_1,\dots,\epsilon_{n}$.
The dominant condition on $a=a_1e_1+\cdots+a_{n}e_{n}\in \sqrt{-1}\frak{t}$ 
amounts to that $a_1\geq \cdots\geq a_n$.
}
\end{set}

\begin{set}
\label{spr}
{\rm
Let $\frak{g}=\frak{sp}(n,\bb{R})$.
Choose $\epsilon_1,\dots,\epsilon_{n} \in \frak{t}_\bb{C}^*$
such that 
\begin{align*}
&\Delta^+(\frak{k}_\bb{C},\frak{t}_\bb{C})
=\{\epsilon_i-\epsilon_j\}_{1\leq i < j \leq n},
\\
&\Delta(\frak{p}_\bb{C},\frak{t}_\bb{C})
=\{\pm 2\epsilon_i\}_{1\leq i\leq n}\cup
\{\pm(\epsilon_i+\epsilon_j)\}
_{1\leq i < j \leq n}.
\end{align*}
Denote by $e_1,\dots, e_{n}\in\frak{t}_\bb{C}$ the 
dual basis of $\epsilon_1,\dots,\epsilon_{n}$.
The dominant condition on $a=a_1e_1+\cdots+a_{n}e_{n}
\in \sqrt{-1}\frak{t}$ amounts to that $a_1\geq \cdots\geq a_n$.
}
\end{set}

\begin{set}
\label{slc}
{\rm
Let $\frak{g}=\frak{sl}(2n,\bb{C})$.
Choose $\epsilon_1,\dots,\epsilon_{2n} \in \frak{t}_\bb{C}^*$ such that 
\begin{align*}
&\Delta^+(\frak{k}_\bb{C},\frak{t}_\bb{C})=
\{\epsilon_i-\epsilon_j\}_{1\leq i<j\leq 2n},
\\
&\Delta(\frak{p}_\bb{C},\frak{t}_\bb{C})
=\{\pm(\epsilon_i-\epsilon_j)\}
_{1\leq i < j \leq 2n}
\cup\{0\}.
\end{align*}
Define $e_1,\dots,e_{2n}\in\frak{t}_\bb{C}$ such that
$(\epsilon_i-\epsilon_j)(e_k)=\delta_{ik}-\delta_{jk}$
and then $e_1+\cdots+e_{2n}=0$.
The dominant condition on $a=a_1e_1+\cdots+a_{2n}e_{2n}
\in \sqrt{-1}\frak{t}$ amounts to that 
$a_1\geq \cdots\geq a_{2n}$.
}
\end{set}

\begin{set}
\label{soc}
{\rm
Let $\frak{g}=\frak{so}(2n,\bb{C})$.
Choose $\epsilon_1,\dots,\epsilon_n \in \frak{t}_\bb{C}^*$ such that 
\begin{align*}
&\Delta^+(\frak{k}_\bb{C},\frak{t}_\bb{C})
=\{\epsilon_i\pm\epsilon_j\}_{1\leq i < j \leq n},\\
&\Delta(\frak{p}_\bb{C},\frak{t}_\bb{C})
=\{\pm\epsilon_i\pm\epsilon_j\}
_{1\leq i<j \leq n}
\cup\{0\}.
\end{align*}
Denote by $e_1,\dots, e_{n}\in\frak{t}_\bb{C}$ the 
dual basis of $\epsilon_1,\dots,\epsilon_{n}$.
The dominant condition on $a=a_1e_1+\cdots+a_{n}e_{n}
\in \sqrt{-1}\frak{t}$ amounts to that 
$a_1\geq \cdots\geq a_{n-1}\geq |a_n|$.
}
\end{set}

For real exceptional Lie algebras, we follow the notation of 
\cite[Chapter X]{hel}.

\begin{set}
\label{f41}
{\rm
Let $\frak{g}=\frak{f}_{4(4)}(\equiv\frak{f}_4^1)$ 
so that $\frak{k}_\bb{C}=\frak{sp}(3,\bb{C})\oplus\frak{sl}(2,\bb{C})$.
Choose $\epsilon_1,\epsilon_2,\epsilon_{3},\epsilon_4 \in \frak{t}_\bb{C}^*$
such that 
\begin{align*}
&\Delta^+(\frak{k}_\bb{C},\frak{t}_\bb{C})
=\{\epsilon_i\pm\epsilon_j\}_{1\leq i < j \leq 3}
\cup\{2\epsilon_{i}\}_{1\leq i\leq 3}
\cup\{2\epsilon_4\},\\
&\Delta(\frak{p}_\bb{C},\frak{t}_\bb{C})
=\{\pm\epsilon_1\pm\epsilon_2\pm\epsilon_3\pm\epsilon_4\}
\cup\{\pm\epsilon_i\pm\epsilon_4\}_{1\leq i\leq 3}.
\end{align*}
Denote by $e_1,e_2,e_3,e_4\in\frak{t}_\bb{C}$ the 
dual basis of $\epsilon_1,\epsilon_2,\epsilon_3,\epsilon_4$.
The dominant condition on $a=a_1e_1+a_2e_2+a_3e_3+a_4e_4
\in \sqrt{-1}\frak{t}$ amounts to that 
$a_1\geq a_2\geq a_3\geq 0$ and $a_4\geq 0$.
}
\end{set}

\begin{set}
\label{f42}
{\rm
Let $\frak{g}=\frak{f}_{4(-20)}(\equiv\frak{f}_4^2)$ 
so that $\frak{k}_\bb{C}=\frak{so}(9,\bb{C})$.
Choose $\epsilon_1,\epsilon_2,\epsilon_{3},\epsilon_4\in\frak{t}_\bb{C}^*$
such that 
\begin{align*}
&\Delta^+(\frak{k}_\bb{C},\frak{t}_\bb{C})
=\{\epsilon_i\pm\epsilon_j\}_{1\leq i < j \leq 4}
\cup\{\epsilon_{i}\}_{1\leq i\leq 4},\\
&\Delta(\frak{p}_\bb{C},\frak{t}_\bb{C})
=\Bigl\{\frac{1}{2}
(\pm\epsilon_1\pm\epsilon_2\pm\epsilon_3\pm\epsilon_4)
\Bigr\}.
\end{align*}
Denote by $e_1,e_2,e_3,e_4\in\frak{t}_\bb{C}$ the 
dual basis of $\epsilon_1,\epsilon_2,\epsilon_3,\epsilon_4$.
The dominant condition on $a=a_1e_1+\cdots+a_{4}e_{4}
\in \sqrt{-1}\frak{t}$ amounts to that 
$a_1\geq a_2\geq a_3\geq a_4\geq 0$.
}
\end{set}

\begin{set}
\label{e62}
{\rm
Let $\frak{g}=\frak{e}_{6(2)}(\equiv\frak{e}_{6}^2)$ 
so that $\frak{k}_\bb{C}=\frak{sl}(6,\bb{C})\oplus\frak{sl}(2,\bb{C})$.
Choose $\epsilon_1,\dots,\epsilon_7 \in \frak{t}_\bb{C}^*$ 
such that 
\begin{align*}
&\Delta^+(\frak{k}_\bb{C},\frak{t}_\bb{C})
=\{\epsilon_i-\epsilon_j\}_{1\leq i < j \leq 6}
\cup\{2\epsilon_7\},\\
&\Delta(\frak{p}_\bb{C},\frak{t}_\bb{C})
=\Bigl\{\frac{1}{2}
\Bigl(\sum_{i=1}^6 (-1)^{k(i)} \epsilon_i\Bigr)
 \pm \epsilon_7:\;
k(i)\in\{0,1\},\, k(1)+\dots+k(6)=3
\Bigr\}.
\end{align*}
Define $e_1,\dots,e_7\in\frak{t}_\bb{C}$ such that 
$(\epsilon_i-\epsilon_j)(e_k)=\delta_{ik}-\delta_{jk}$, 
$\epsilon_7(e_7)=1$, and 
$(\epsilon_i-\epsilon_j)(e_7)=\epsilon_7(e_k)=0$ for $1\leq i,j,k\leq 6$.
Then $e_1+\cdots +e_6=0$.
The dominant condition on $a=a_1e_1+\cdots+a_{7}e_{7}
\in \sqrt{-1}\frak{t}$ amounts to that 
$a_1\geq \cdots \geq a_6$ and $a_7\geq 0$.
}
\end{set}

\begin{set}
\label{e63}
{\rm
Let $\frak{g}=\frak{e}_{6(-14)}(\equiv\frak{e}_6^3)$ 
so that $\frak{k}_\bb{C}=\frak{so}(10,\bb{C})\oplus\bb{C}$.
Choose $\epsilon_1,\dots,\epsilon_6 \in \frak{t}_\bb{C}^*$ 
such that 
\begin{align*}
&\Delta^+(\frak{k}_\bb{C},\frak{t}_\bb{C})
=\{\epsilon_i\pm\epsilon_j\}_{1\leq i < j \leq 5},\\
&\Delta(\frak{p}_\bb{C},\frak{t}_\bb{C})
=\Bigl\{\frac{1}{2}
\Bigl(\sum_{i=1}^6 {(-1)^{k(i)}}\epsilon_i\Bigr)
:\;
k(1)+\cdots+k(6)\;{\rm odd}
\Bigr\}
.
\end{align*}
Denote by $e_1,\dots,e_6\in\frak{t}_\bb{C}$ the 
dual basis of $\epsilon_1,\dots,\epsilon_6$.
The dominant condition on $a=a_1e_1+\cdots+a_{6}e_{6}
\in \sqrt{-1}\frak{t}$ amounts to that 
$a_1\geq \cdots \geq a_4\geq |a_5|$.
}
\end{set}

\begin{set}
\label{e72}
{\rm
Let $\frak{g}=\frak{e}_{7(-5)}(\equiv\frak{e}_7^2)$ 
so that $\frak{k}_\bb{C}=\frak{so}(12,\bb{C})\oplus\frak{sl}(2,\bb{C})$.
Choose $\epsilon_1,\dots,\epsilon_7 \in \frak{t}_\bb{C}^*$ 
such that 
\begin{align*}
&\Delta^+(\frak{k}_\bb{C},\frak{t}_\bb{C})
=\{\epsilon_i\pm\epsilon_j\}_{1\leq i < j \leq 6}
\cup\{2\epsilon_7\},\\
&\Delta(\frak{p}_\bb{C},\frak{t}_\bb{C})
=\Bigl\{\frac{1}{2}
\Bigl(\sum_{i=1}^6 {(-1)^{k(i)}}\epsilon_i\Bigr)\pm \epsilon_7:\;
k(1)+\cdots+k(6)\,{\rm odd}
\Bigr\}
.
\end{align*}
Denote by $e_1,\dots,e_7\in\frak{t}_\bb{C}$ the 
dual basis of $\epsilon_1,\dots,\epsilon_7$.
The dominant condition on $a=a_1e_1+\cdots+a_{7}e_{7}
\in \sqrt{-1}\frak{t}$ amounts to that 
$a_1\geq \cdots \geq a_5\geq |a_6|$ and $a_7\geq 0$.
}
\end{set}

\begin{set}
\label{e73}
{\rm
Let $\frak{g}=\frak{e}_{7(-25)}(\equiv\frak{e}_7^3)$ 
so that $\frak{k}_\bb{C}=\frak{e}_6^\bb{C}\oplus\bb{C}$.
Choose $\epsilon_1,\dots,\epsilon_8 \in \frak{t}_\bb{C}^*$ 
such that 
\begin{align*}
&\Delta^+(\frak{k}_\bb{C},\frak{t}_\bb{C})
=\{\epsilon_i\pm\epsilon_j\}_{1\leq j < i \leq 5}\\
&\qquad\cup\Bigl\{
\frac{1}{2}\Bigl(\epsilon_8-\epsilon_7-\epsilon_6
+\sum_{i=1}^5 {(-1)^{k(i)}}\epsilon_i\Bigr):\;
k(1)+\cdots+k(5)\,{\rm even}
\Bigr\},
\\
&\Delta(\frak{p}_\bb{C},\frak{t}_\bb{C})
=\{\pm\epsilon_6\pm \epsilon_i\}_{1\leq i\leq 5}
\cup\{\pm(\epsilon_8-\epsilon_7)\}\\
&\qquad\cup\Bigl\{
\pm\frac{1}{2}\Bigl(\epsilon_8-\epsilon_7+\epsilon_6
+\sum_{i=1}^5 {(-1)^{k(i)}}\epsilon_i\Bigr):\;
k(1)+\cdots+k(5)\,{\rm odd}
\Bigr\}.
\end{align*}
Define $e_1,\dots,e_{8}\in\frak{t}_\bb{C}$ such that
$\epsilon_i(e_j)=\delta_{ij}$
for $1\leq i \leq 6,\ 1\leq j \leq 8$
and that 
$(\epsilon_8-\epsilon_7)(e_i)=\delta_{i8}-\delta_{i7}$ for $1\leq i \leq 8$.
Then $e_8+e_7=0$.
The dominant condition on $a=a_1e_1+\cdots+a_{8}e_{8}
\in \sqrt{-1}\frak{t}$ amounts to that 
$a_5\geq \cdots \geq a_2\geq |a_1|$ and 
$a_8-a_7-a_6-a_5-a_4-a_3-a_2+a_1\geq 0$.
}
\end{set}

\begin{set}
\label{e82}
{\rm
Let $\frak{g}=\frak{e}_{8(-24)}(\equiv\frak{e}_8^2)$ 
so that $\frak{k}_\bb{C}=\frak{e}_7^\bb{C}\oplus\frak{sl}(2,\bb{C})$.
Choose $\epsilon_1,\dots,\epsilon_8\in \frak{t}_\bb{C}^*$  
such that 
\begin{align*}
&\Delta^+(\frak{k}_\bb{C},\frak{t}_\bb{C})
=\{\epsilon_i\pm\epsilon_j\}_{1\leq j < i \leq 6}
\cup\{\epsilon_8\pm\epsilon_7\}\\
&\qquad\cup\Bigl\{
\frac{1}{2}\Bigl(\epsilon_8-\epsilon_7+
\sum_{i=1}^6 {(-1)^{k(i)}}\epsilon_i\Bigr):\;
k(1)+\cdots+k(6)\,{\rm odd}
\Bigr\},
\\
&\Delta(\frak{p}_\bb{C},\frak{t}_\bb{C})
=\{\pm\epsilon_7\pm \epsilon_i\}_{1\leq i\leq 6}
\cup
\{\pm\epsilon_8\pm \epsilon_i\}_{1\leq i\leq 6}\\
&\qquad\cup\Bigl\{
\pm\frac{1}{2}\Bigl(\epsilon_8+\epsilon_7+
\sum_{i=1}^6 {(-1)^{k(i)}}\epsilon_i\Bigr):\;
k(1)+\cdots+k(6)\,{\rm even}
\Bigr\}
.
\end{align*}
Denote by $e_1,\dots,e_8\in\frak{t}_\bb{C}$ the 
dual basis of $\epsilon_1,\dots,\epsilon_8$.
The dominant condition on $a=a_1e_1+\cdots+a_{8}e_{8}
\in \sqrt{-1}\frak{t}$ amounts to that 
$a_6\geq \cdots \geq a_2\geq |a_1|$ and 
$a_8-a_7-a_6-a_5-a_4-a_3-a_2+a_1\geq 0$.
}
\end{set}


\section{List of Symmetric Pairs
Satisfying the Assumption of Proposition \ref{highlow}}
\label{sec:diagram}

In this appendix we assume that 
$\frak{g}$ is a non-compact simple Lie algebra and 
classify all the irreducible symmetric pairs 
$(\frak{g},\frak{g}^{\sigma})$ 
satisfying the following two conditions: 
\begin{enumerate}
\item[(1)] 
$-\sigma\alpha_0$ is dominant 
with respect to $\Delta^+(\frak{k}_\bb{C},\frak{t}_\bb{C})$,
\item[(2)]
$\frak{t}^{\sigma}\neq 0$, 
\end{enumerate}
where $\alpha_0$ is the highest weight of $\frak{p}_\bb{C}$ 
with respect to $\Delta^+(\frak{k}_\bb{C},\frak{t}_\bb{C})$ 
(if $\frak{g}$ is not of Hermitian type) or that of 
$\frak{p}_+$ (if $\frak{g}$ is of Hermitian type). 
The condition (1) is the key assumption in Proposition \ref{highlow}.
Recall we have assumed 
that $\Delta^+(\frak{k}_\bb{C},\frak{t}_\bb{C})$ and 
$\sigma$ satisfy the compatibility condition of Setting \ref{setq} (1).
If $\frak{t}^{\sigma}=0$, then we can apply Proposition \ref{split} 
and we see there is no parabolic subalgebra other than $\frak{g}_\bb{C}$ 
satisfying the discrete decomposability condition. 
To save space, we do not list the pairs
 $(\frak{g},\frak{g}^\sigma)$ with $\frak{t}^{\sigma}=0$ because 
we can easily verify the condition $\frak{t}^{\sigma}=0$, which is equivalent 
to that $\frak{k}^{\sigma}+\sqrt{-1}\frak{k}^{-\sigma}$ is 
a split real form of $\frak{k}_\bb{C}$. 

In view of Proposition \ref{dom}, the classification 
of such pairs
$(\frak{g},\frak{g}^\sigma)$ is carried out by using diagrams.
We thus write the Satake diagram of 
$\frak{k}^{\sigma}+\sqrt{-1}\frak{k}^{-\sigma}$ and 
add a vertex $\star$, which is associated to the weight $\alpha_0$ 
as explained in the paragraph before Proposition \ref{dom}.

In \ref{subsec:nothol}, we list all the pairs 
satisfying the assumption (1) or (2) in Proposition \ref{highlow}
and $\frak{t}^{\sigma}\neq 0$.
For these pairs, no parabolic subalgebra other than $\frak{g}_\bb{C}$ 
satisfies the discrete decomposability condition. 
In \ref{subsec:hol}, we list all the pairs 
satisfying the assumption (3) of Proposition \ref{highlow}
and $\frak{t}^{\sigma}\neq 0$.
By Propositions \ref{holdd} and \ref{highlow}, 
the triple $(\frak{g},\frak{g}^{\sigma},\frak{q})$ 
satisfies the discrete decomposability 
condition if and only if $\frak{q}$ is holomorphic or anti-holomorphic
(Definition \ref{def:holoparab}).

In what follows, 
it is convenient to use the following symbol: 

\begin{align*}
\begin{xy}
\ar@{-} (0,0); (15,0) *+[F]{\qquad p \qquad }="A"
\ar@{-} "A"; (30,0) *{}
\end{xy}
:=
\begin{cases}
\begin{xy}
(17.5,-4.5) *{p}
\ar@{-} (-5,2); (0,2) *{\bullet}="A"
\ar@{-} "A"; (10,2) *{\bullet}="B"
\ar@{-} "B"; (15,2) *{}="C"
\ar@{.} "C"; (20,2) *{}="D"
\ar@{-} "D"; (25,2) *{\bullet}="E"
\ar@{-} "E"; (35,2) *{\bullet}="F"
\ar@{-} "F"; (40,2) *{}
$\underbrace{\hspace*{3.5cm}}$
\end{xy}
& (p\geq 0)\\
\begin{xy}
\ar@{-} (0,3);(10,3) *{\circ}="A"
\ar@{-} "A"; (20,3)
\end{xy}
& (p=-1).
\end{cases}
\end{align*}

We note that the diagram for $(\frak{g},\frak{g}^{\sigma})$ is the same as 
that for the associated pair $(\frak{g},\frak{g}^{\theta\sigma})$.

\subsection{Case of non-holomorphic symmetric pairs}
\label{subsec:nothol}

\subsubsection{}
\ 

$m,n-m\geq 1,\quad |n-2m|\geq 2$

$(\frak{g},\frak{g}^{\sigma})
=(\frak{sl}(n,\bb{R}),
\frak{sl}(m,\bb{R}) \oplus \frak{sl}(n-m,\bb{R}) \oplus \bb{R})$

$(\frak{g},\frak{g}^{\theta\sigma})
=(\frak{sl}(n,\bb{R}),\frak{so}(m,n-m))$

Case: $n$ even 

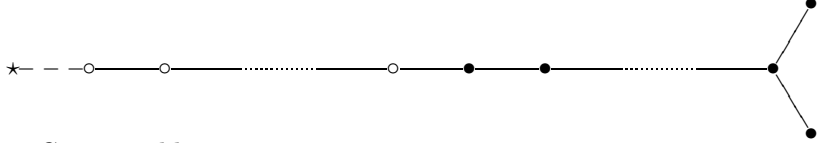
\begin{xy}
\ar@{--} *{\star}; (10,0) *{\circ}="A"
\ar@{-} "A"; (20,0) *{\circ}="B"
\ar@{-} "B"; (30,0) *{}="C"
\ar@{.} "C"; (40,0) *{}="D"
\ar@{-} "D"; (50,0) *{\circ}="E"
\ar@{-} "E"; (60,0) *{\bullet}="F"
\ar@{-} "F"; (70,0) *{\bullet}="G"
\ar@{-} "G"; (80,0) *{}="H"
\ar@{.} "H"; (90,0) *{}="I"
\ar@{-} "I"; (100,0) *{\bullet}="J"
\ar@{-} "J"; (105,-8.6) *{\bullet}="K"
\ar@{-} "J"; (105,8.6) *{\bullet}="L"
\end{xy}

Case: $n$ odd 

\begin{xy}
\ar@{--} *{\star}; (10,0) *{\circ}="A"
\ar@{-} "A"; (20,0) *{\circ}="B"
\ar@{-} "B"; (30,0) *{}="C"
\ar@{.} "C"; (40,0) *{}="D"
\ar@{-} "D"; (50,0) *{\circ}="E"
\ar@{-} "E"; (60,0) *{\bullet}="F"
\ar@{-} "F"; (70,0) *{\bullet}="G"
\ar@{-} "G"; (80,0) *{}="H"
\ar@{.} "H"; (90,0) *{}="I"
\ar@{-} "I"; (100,0) *{\bullet}="J"
\ar@{=>} "J"; (110,0) *{\bullet}="K"
\end{xy}

\subsubsection{}
\ 

$n\geq 2$

$(\frak{g},\frak{g}^{\sigma})
=(\frak{su}(n,n),
\frak{sl}(n,\bb{C})\oplus\bb{R})$

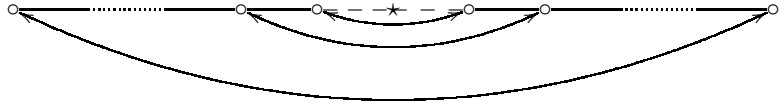
\begin{xy}
\ar@{-} *{\circ}="A"; (10,0) 
\ar@{.} (10,0) ; (20,0)
\ar@{-} (20,0); (30,0) *{\circ}="B"
\ar@{-} "B"   ; (40,0) *{\circ}="C"
\ar@{--}"C"   ; (50,0) *{\star}
\ar@{--}(50,0); (60,0) *{\circ}="D"
\ar@{-} "D"   ; (70,0) *{\circ}="E"
\ar@{-} "E"   ; (80,0) 
\ar@{.} (80,0); (90,0) 
\ar@{-} (90,0); (100,0) *{\circ}="F"
\ar@/_12mm/@{<->} "A"; "F"
\ar@/_5mm/@{<->} "B"; "E"
\ar@/_2mm/@{<->} "C"; "D"
\end{xy}

\subsubsection{}
\ 

$m,n-m\geq 1$

$(\frak{g},\frak{g}^\sigma)
=(\frak{su}^*(2n), 
\frak{su}^*(2m) \oplus \frak{su}^*(2n-2m) \oplus \bb{R})$

$(\frak{g},\frak{g}^{\theta\sigma})
=(\frak{su}^*(2n), \frak{sp}(m,n-m))$

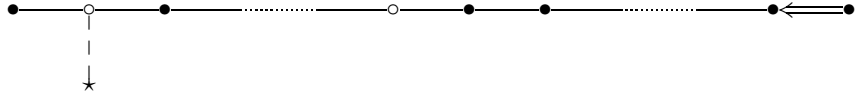
\begin{xy}
\ar@{-} *{\bullet}; (10,0) *{\circ}="A"
\ar@{--} "A"; (10,-10) *{\star}
\ar@{-} "A"; (20,0) *{\bullet}="C"
\ar@{-} "C"; (30,0) 
\ar@{.} (30,0); (40,0) 
\ar@{-} (40,0); (50,0) *{\circ}="B"
\ar@{-} "B"   ; (60,0) *{\bullet}="D"
\ar@{-} "D"   ; (70,0) *{\bullet}="E"
\ar@{-} "E"; (80,0) 
\ar@{.} (80,0); (90,0) 
\ar@{-} (90,0); (100,0) *{\bullet}="F"
\ar@{<=}"F"; (110,0) *{\bullet}
\end{xy}

\subsubsection{}
\ 

$k,l,m-k,n-l\geq 1,\quad {\rm max} \{ |m-2k|,|n-2l|\}\geq 2$

$(\frak{g},\frak{g}^\sigma)
=(\frak{so}(m,n), 
\frak{so}(k,l) \oplus \frak{so}(m-k,n-l))$

Case: $m,n$ even

\begin{xy}
\ar@{-}(0,8.6) *{\bullet}; (5,0) *{\bullet}="A"
\ar@{-}(0,-8.6)*{\bullet}; "A"
\ar@{-} "A"; (10,0) 
\ar@{.} (10,0); (15,0) 
\ar@{-} (15,0); (20,0) *{\bullet}="B"
\ar@{-} "B"   ; (30,0) *{\circ}="C"
\ar@{-} "C"   ; (35,0) 
\ar@{.} (35,0); (40,0) 
\ar@{-} (40,0); (45,0) *{\circ}="D"
\ar@{--}"D"   ; (55,0) *{\star}="E"
\ar@{--}"E"   ; (65,0) *{\circ}="F"
\ar@{-} "F"   ; (70,0)
\ar@{.} (70,0); (75,0) 
\ar@{-} (75,0); (80,0) *{\circ}="G"
\ar@{-} "G"   ; (90,0) *{\bullet}="H"
\ar@{-} "H"   ; (95,0)
\ar@{.} (95,0); (100,0)
\ar@{-} (100,0); (105,0) *{\bullet}="I"
\ar@{-} "I"   ; (110,-8.6) *{\bullet}
\ar@{-} "I"   ; (110,8.6) *{\bullet}
\end{xy}

Case: $m+n$ odd

\begin{xy}
\ar@{-}(0,8.6) *{\bullet}; (5,0) *{\bullet}="A"
\ar@{-}(0,-8.6)*{\bullet}; "A"
\ar@{-} "A"; (10,0) 
\ar@{.} (10,0); (15,0) 
\ar@{-} (15,0); (20,0) *{\bullet}="B"
\ar@{-} "B"   ; (30,0) *{\circ}="C"
\ar@{-} "C"   ; (35,0) 
\ar@{.} (35,0); (40,0) 
\ar@{-} (40,0); (45,0) *{\circ}="D"
\ar@{--}"D"   ; (55,0) *{\star}="E"
\ar@{--}"E"   ; (65,0) *{\circ}="F"
\ar@{-} "F"   ; (70,0)
\ar@{.} (70,0); (75,0) 
\ar@{-} (75,0); (80,0) *{\circ}="G"
\ar@{-} "G"   ; (90,0) *{\bullet}="H"
\ar@{-} "H"   ; (95,0)
\ar@{.} (95,0); (100,0)
\ar@{-} (100,0); (105,0) *{\bullet}="I"
\ar@{=>} "I"   ; (115,0) *{\bullet}
\end{xy}

Case: $m,n$ odd

\begin{xy}
\ar@{<=}(0,0) *{\bullet}; (10,0) *{\bullet}="A"
\ar@{-} "A"; (15,0) 
\ar@{.} (15,0); (20,0) 
\ar@{-} (20,0); (25,0) *{\bullet}="B"
\ar@{-} "B"   ; (35,0) *{\circ}="C"
\ar@{-} "C"   ; (40,0) 
\ar@{.} (40,0); (45,0) 
\ar@{-} (45,0); (50,0) *{\circ}="D"
\ar@{--}"D"   ; (60,0) *{\star}="E"
\ar@{--}"E"   ; (70,0) *{\circ}="F"
\ar@{-} "F"   ; (75,0)
\ar@{.} (75,0); (80,0) 
\ar@{-} (80,0); (85,0) *{\circ}="G"
\ar@{-} "G"   ; (95,0) *{\bullet}="H"
\ar@{-} "H"   ; (100,0)
\ar@{.} (100,0); (105,0)
\ar@{-} (105,0); (110,0) *{\bullet}="I"
\ar@{=>} "I"   ; (120,0) *{\bullet}
\end{xy}

\medskip

\subsubsection{}
\ 

$n\geq 3$

$(\frak{g},\frak{g}^{\sigma})
=(\frak{so}(n,n),\frak{so}(n,\bb{C}))$

$(\frak{g},\frak{g}^{\theta\sigma})
=(\frak{so}(n,n),\frak{gl}(n,\bb{R}))$

Case: $n$ even

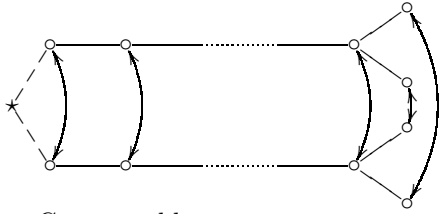
\begin{xy}
\ar@{--} (0,0) *{\star}="A"; (5,-8) *{\circ}="B1"
\ar@{--} "A"; (5,8) *{\circ}="B2"
\ar@{-} "B1" ; (15,-8) *{\circ}="C1"
\ar@{-} "C1" ; (25,-8)
\ar@{.} (25,-8); (35,-8)
\ar@{-} (35,-8); (45,-8) *{\circ}="D1"
\ar@{-} "D1"; (52,-13) *{\circ}="E1"
\ar@{-} "D1"; (52,-3)  *{\circ}="F1"
\ar@{-} "B2" ; (15,8) *{\circ}="C2"
\ar@{-} "C2" ; (25,8)
\ar@{.} (25,8); (35,8)
\ar@{-} (35,8); (45,8) *{\circ}="D2"
\ar@{-} "D2"; (52,13) *{\circ}="E2"
\ar@{-} "D2"; (52,3)  *{\circ}="F2"
\ar@/_/@{<->}"B1"; "B2"
\ar@/_/@{<->}"C1"; "C2"
\ar@/_/@{<->}"D1"; "D2"
\ar@/_{4mm}/@{<->}"E1"; "E2"
\ar@/_{.5mm}/@{<->}"F1"; "F2"
\end{xy}

Case: $n$ odd 

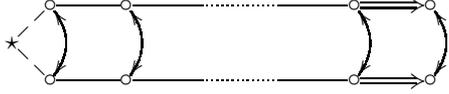
\medskip

\begin{xy}
\ar@{--} (0,0) *{\star}="A"; (5,-5) *{\circ}="B1"
\ar@{--} "A"; (5,5) *{\circ}="B2"
\ar@{-} "B1" ; (15,-5) *{\circ}="C1"
\ar@{-} "C1" ; (25,-5)
\ar@{.} (25,-5); (35,-5)
\ar@{-} (35,-5); (45,-5) *{\circ}="D1"
\ar@{=>} "D1"; (55,-5)  *{\circ}="E1"
\ar@{-} "B2" ; (15,5) *{\circ}="C2"
\ar@{-} "C2" ; (25,5)
\ar@{.} (25,5); (35,5)
\ar@{-} (35,5); (45,5) *{\circ}="D2"
\ar@{=>} "D2"; (55,5)  *{\circ}="E2"
\ar@/_/@{<->}"B1"; "B2"
\ar@/_/@{<->}"C1"; "C2"
\ar@/_/@{<->}"D1"; "D2"
\ar@/_/@{<->}"E1"; "E2"
\end{xy}

\subsubsection{}
\ 

$n\geq 2$

$(\frak{g},\frak{g}^\sigma)
=(\frak{so}^*(4n),\frak{su}^*(2n)\oplus\bb{R})$

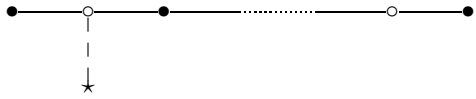
\begin{xy}
\ar@{-} (0,0) *{\bullet}; (10,0) *{\circ}="A"
\ar@{--}  "A"; (10,-10) *{\star}
\ar@{-}  "A"; (20, 0) *{\bullet}="B"
\ar@{-}  "B"; (30,0)
\ar@{.}  (30,0); (40,0)
\ar@{-}  (40,0); (50,0) *{\circ}="C"
\ar@{-} "C"; (60,0) *{\bullet}
\end{xy}

\subsubsection{}
\ 

$n\geq 1$

$(\frak{g},\frak{g}^\sigma)
=(\frak{sp}(n,n),\frak{sp}(n,\bb{C}))$

$(\frak{g},\frak{g}^{\theta\sigma})
=(\frak{sp}(n,n),\frak{su}^*(2n)\oplus\bb{R})$

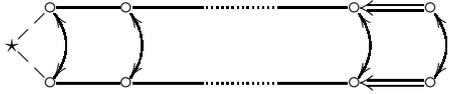
\medskip

\begin{xy}
\ar@{--} (0,0) *{\star}="A"; (5,-5) *{\circ}="B1"
\ar@{--} "A"; (5,5) *{\circ}="B2"
\ar@{-} "B1" ; (15,-5) *{\circ}="C1"
\ar@{-} "C1" ; (25,-5)
\ar@{.} (25,-5); (35,-5)
\ar@{-} (35,-5); (45,-5) *{\circ}="D1"
\ar@{<=} "D1"; (55,-5)  *{\circ}="E1"
\ar@{-} "B2" ; (15,5) *{\circ}="C2"
\ar@{-} "C2" ; (25,5)
\ar@{.} (25,5); (35,5)
\ar@{-} (35,5); (45,5) *{\circ}="D2"
\ar@{<=} "D2"; (55,5)  *{\circ}="E2"
\ar@/_/@{<->}"B1"; "B2"
\ar@/_/@{<->}"C1"; "C2"
\ar@/_/@{<->}"D1"; "D2"
\ar@/_/@{<->}"E1"; "E2"
\end{xy}

\subsubsection{}
\ 

$m,n-m\geq 1$

$(\frak{g},\frak{g}^\sigma)
=(\frak{sl}(n,\bb{C}),
\frak{sl}(m,\bb{C})\oplus\frak{sl}(n-m,\bb{C})\oplus\bb{C})$

$(\frak{g},\frak{g}^{\theta\sigma})
=(\frak{sl}(n,\bb{C}),\frak{su}(m,n-m))$

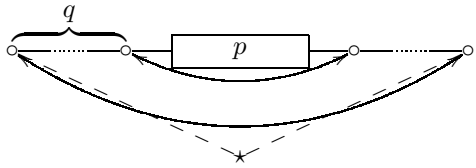
\smallskip

\begin{xy}
(7.5, 4)*{q}
\ar@{-} (0,-1) *{\circ}="A"; (5,-1)
\ar@{.} (5,-1) ; (10,-1)
\ar@{-} (10,-1) ; (15,-1) *{\circ}="B"
\ar@{-} "B"; (30,-1) *+[F]{\qquad p \qquad}="D"
\ar@{-} "D"; (45,-1) *{\circ}="E"
\ar@{-} "E"; (50,-1) *{}
\ar@{.} (50,-1); (55,-1) *{}
\ar@{-} (55,-1); (60,-1) *{\circ}="F"
\ar@{--} (30,-15) *{\star}="C"; "A"
\ar@{--} "C"; "F"
\ar@/_{10mm}/@{<->} "A"; "F"
\ar@/_{4mm}/@{<->} "B"; "E"
$\overbrace{\hspace*{15mm}}$
\end{xy}

\[
p=\begin{cases}
n-2m-1&\text{if $m<n-m$}\\
2m-n-1&\text{if $m>n-m$}\\
-1&\text{if $m=n-m$}
\end{cases}
\qquad 
q=\begin{cases}
m&\text{if $m<n-m$}\\
n-m&\text{if $m>n-m$}\\
m-1&\text{if $m=n-m$}
\end{cases}
\]

\subsubsection{}
\ 

$m,2n-m\geq 2,\quad |2n-2m|\geq 2$

$(\frak{g},\frak{g}^\sigma)
=(\frak{so}(2n,\bb{C}),
\frak{so}(m,\bb{C})\oplus\frak{so}(2n-m,\bb{C}))$

$(\frak{g},\frak{g}^{\theta\sigma})
=(\frak{so}(2n,\bb{C}),\frak{so}(m,2n-m))$

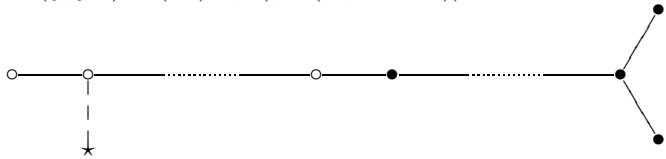
\begin{xy}
\ar@{-} (0,0) *{\circ}; (10,0) *{\circ}="A"
\ar@{--}"A"; (10,-10) *{\star}
\ar@{-} "A";  (20,0) 
\ar@{.} (20,0) ; (30,0)
\ar@{-} (30,0) ; (40,0) *{\circ}="C"
\ar@{-} "C" ; (50,0) *{\bullet}="D"
\ar@{-} "D";  (60,0) 
\ar@{.} (60,0) ; (70,0)
\ar@{-} (70,0) ; (80,0) *{\bullet}="E"
\ar@{-} "E" ; (85,-8.6) *{\bullet}
\ar@{-} "E" ; (85,8.6) *{\bullet}
\end{xy}

\subsubsection{}
\ 

$m,2n-m+1\geq 2,\quad |2n-2m+1|\geq 2$

$(\frak{g},\frak{g}^\sigma)
=(\frak{so}(2n+1,\bb{C}),
\frak{so}(m,\bb{C})\oplus\frak{so}(2n-m+1,\bb{C}))$

$(\frak{g},\frak{g}^{\theta\sigma})
=(\frak{so}(2n+1,\bb{C}),\frak{so}(m,2n-m+1))$

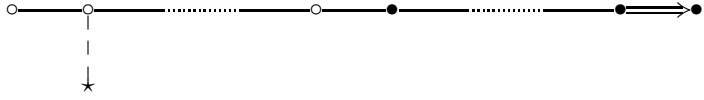
\begin{xy}
\ar@{-} (0,0) *{\circ}; (10,0) *{\circ}="A"
\ar@{--}"A"; (10,-10) *{\star}
\ar@{-} "A";  (20,0) 
\ar@{.} (20,0) ; (30,0)
\ar@{-} (30,0) ; (40,0) *{\circ}="C"
\ar@{-} "C" ; (50,0) *{\bullet}="D"
\ar@{-} "D";  (60,0) 
\ar@{.} (60,0) ; (70,0)
\ar@{-} (70,0) ; (80,0) *{\bullet}="E"
\ar@{=>} "E" ; (90,0) *{\bullet}
\end{xy}

\subsubsection{}
\ 

$n\geq 4$

$(\frak{g},\frak{g}^\sigma)
=(\frak{so}(2n,\bb{C}),\frak{gl}(n,\bb{C}))$

$(\frak{g},\frak{g}^{\theta\sigma})
=(\frak{so}(2n,\bb{C}),\frak{so}^*(2n))$

Case: $n$  even

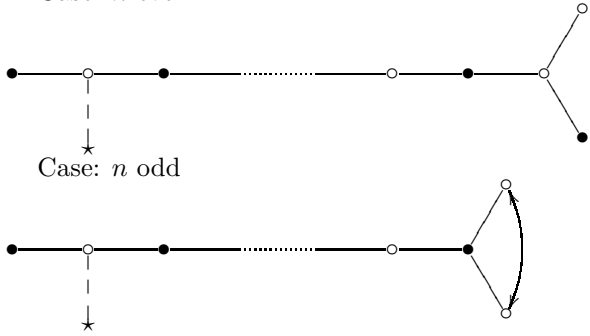
\begin{xy}
\ar@{-} (0,0) *{\bullet}; (10,0) *{\circ}="A"
\ar@{--}"A"; (10,-10) *{\star}
\ar@{-} "A";  (20,0) *{\bullet}="B"
\ar@{-} "B";  (30,0)
\ar@{.} (30,0) ; (40,0)
\ar@{-} (40,0) ; (50,0) *{\circ}="C"
\ar@{-} "C";  (60,0) *{\bullet}="D"
\ar@{-} "D";  (70,0) *{\circ}="E"
\ar@{-} "E" ; (75,-8.6) *{\bullet}
\ar@{-} "E";  (75,8.6) *{\circ}
\end{xy}

Case: $n$ odd

\begin{xy}
\ar@{-} (0,0) *{\bullet}; (10,0) *{\circ}="A"
\ar@{--}"A"; (10,-10) *{\star}
\ar@{-} "A";  (20,0) *{\bullet}="B"
\ar@{-} "B";  (30,0)
\ar@{.} (30,0) ; (40,0)
\ar@{-} (40,0) ; (50,0) *{\circ}="C"
\ar@{-} "C";  (60,0) *{\bullet}="D"
\ar@{-} "D" ; (65,-8.6) *{\circ}="E"
\ar@{-} "D";  (65,8.6) *{\circ}="F"
\ar@/_/@{<->} "E"; "F"
\end{xy}

\subsubsection{}
\ 

$(\frak{g},\frak{g}^\sigma)=
(\frak{e}_{6(6)},\frak{so}(5,5)\oplus\bb{R})$

$(\frak{g},\frak{g}^{\theta\sigma})
=(\frak{e}_{6(6)},\frak{sp}(2,2))$

\begin{xy}
\ar@{-} (0,0) *{\bullet}; (10,0) *{\circ}="A"
\ar@{-}"A"; (20,0) *{\bullet}="B"
\ar@{<=} "B";  (30,0) *{\circ}="C"
\ar@{--} "C";  (40,0) *{\star}
\end{xy}

\subsubsection{}
\ 

$(\frak{g},\frak{g}^\sigma)=
(\frak{e}_{6(2)},\frak{su}(3,3)\oplus\frak{sl}(2,\bb{R}))$

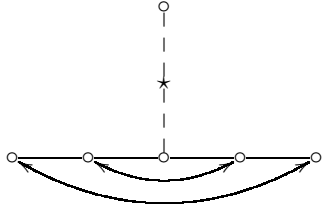
\begin{xy}
\ar@{-} (0,0) *{\circ}="A"; (10,0) *{\circ}="B"
\ar@{-}"B"; (20,0) *{\circ}="C"
\ar@{-} "C";  (30,0) *{\circ}="D"
\ar@{-} "D";  (40,0) *{\circ}="F"
\ar@{--} "C";  (20,10) *{\star}="E"
\ar@{--} "E";  (20,20) *{\circ}
\ar@/_{6mm}/@{<->} "A"; "F"
\ar@/_{3mm}/@{<->} "B"; "D"
\end{xy}

\subsubsection{}
\ 

$(\frak{g},\frak{g}^\sigma)=
(\frak{e}_{6(-26)},\frak{so}(1,9)\oplus\bb{R})$

$(\frak{g},\frak{g}^{\theta\sigma})
=(\frak{e}_{6(-26)}, \frak{f}_{4(-20)})$

\begin{xy}
\ar@{-} (0,0) *{\bullet}="A"; (10,0) *{\bullet}="B"
\ar@{=>}"B"; (20,0) *{\bullet}="C"
\ar@{-} "C";  (30,0) *{\circ}="D"
\ar@{--} "D";  (40,0) *{\star}
\end{xy}

\subsubsection{}
\ 

$(\frak{g},\frak{g}^\sigma)=
(\frak{e}_{7(7)},\frak{so}(6,6)\oplus\frak{sl}(2,\bb{R}))$

$(\frak{g},\frak{g}^{\theta\sigma})
=(\frak{e}_{7(7)}, \frak{su}(4,4))$

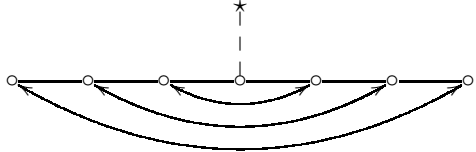
\begin{xy}
\ar@{-} (0,0) *{\circ}="A"; (10,0) *{\circ}="B"
\ar@{-}"B"; (20,0) *{\circ}="C"
\ar@{-} "C";  (30,0) *{\circ}="D"
\ar@{-} "D";  (40,0) *{\circ}="E"
\ar@{-} "E";  (50,0) *{\circ}="F"
\ar@{-} "F";  (60,0) *{\circ}="G"
\ar@{--} "D";  (30,10) *{\star}
\ar@/_{9mm}/@{<->} "A"; "G"
\ar@/_{6mm}/@{<->} "B"; "F"
\ar@/_{3mm}/@{<->} "C"; "E"
\end{xy}

\subsubsection{}
\ 

$(\frak{g},\frak{g}^\sigma)=
(\frak{e}_{7(7)},\frak{su}^*(8))$

$(\frak{g},\frak{g}^{\theta\sigma})
=(\frak{e}_{7(7)}, \frak{e}_{6(6)}\oplus\bb{R})$

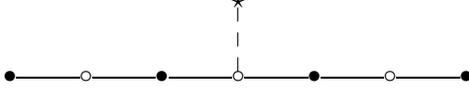
\begin{xy}
\ar@{-} (0,0) *{\bullet}="A"; (10,0) *{\circ}="B"
\ar@{-}"B"; (20,0) *{\bullet}="C"
\ar@{-} "C";  (30,0) *{\circ}="D"
\ar@{-} "D";  (40,0) *{\bullet}="E"
\ar@{-} "E";  (50,0) *{\circ}="F"
\ar@{-} "F";  (60,0) *{\bullet}="G"
\ar@{--} "D";  (30,10) *{\star}
\end{xy}

\subsubsection{}
\ 

$(\frak{g},\frak{g}^\sigma)=
(\frak{e}_{7(-5)},\frak{so}^*(12)\oplus\frak{sl}(2,\bb{R}))$

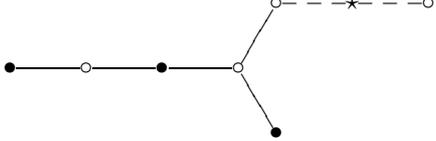
\begin{xy}
\ar@{-} (0,0) *{\bullet}="A"; (10,0) *{\circ}="B"
\ar@{-}"B"; (20,0) *{\bullet}="C"
\ar@{-} "C";  (30,0) *{\circ}="D"
\ar@{-} "D";  (35,8.6) *{\circ}="E"
\ar@{-} "D";  (35,-8.6) *{\bullet}
\ar@{--} "E"; (45,8.6) *{\star}="F"
\ar@{--} "F"; (55,8.6) *{\circ}
\end{xy}

\subsubsection{}
\ 

$(\frak{g},\frak{g}^\sigma)=
(\frak{e}_{7(-25)},\frak{e}_{6(-26)}\oplus\bb{R})$

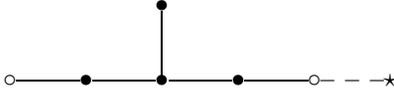
\begin{xy}
\ar@{-} (0,0) *{\circ}="A"; (10,0) *{\bullet}="B"
\ar@{-}"B"; (20,0) *{\bullet}="C"
\ar@{-} "C";  (30,0) *{\bullet}="D"
\ar@{-} "D";  (40,0) *{\circ}="E"
\ar@{-} "C";  (20,10) *{\bullet}
\ar@{--} "E"; (50,0) *{\star}
\end{xy}

\subsubsection{}
\ 

$(\frak{g},\frak{g}^\sigma)=
(\frak{e}_{8(8)},\frak{so}^*(16))$

$(\frak{g},\frak{g}^{\theta\sigma})
=(\frak{e}_{8(8)},\frak{e}_{7(7)}\oplus\frak{sl}(2,\bb{R}))$

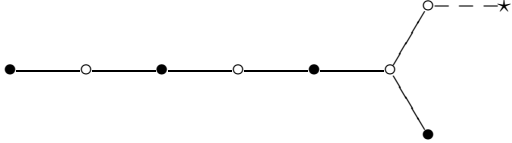
\begin{xy}
\ar@{-} (0,0) *{\bullet}="A"; (10,0) *{\circ}="B"
\ar@{-}"B"; (20,0) *{\bullet}="C"
\ar@{-} "C";  (30,0) *{\circ}="D"
\ar@{-} "D";  (40,0) *{\bullet}="E"
\ar@{-} "E";  (50,0) *{\circ}="F"
\ar@{-} "F";  (55,-8.6) *{\bullet}
\ar@{-} "F"; (55,8.6) *{\circ}="G"
\ar@{--}"G"; (65,8.6) *{\star}
\end{xy}

\subsubsection{}
\ 

$(\frak{g},\frak{g}^\sigma)=
(\frak{e}_{8(-24)},\frak{e}_{7(-25)}\oplus\frak{sl}(2,\bb{R}))$

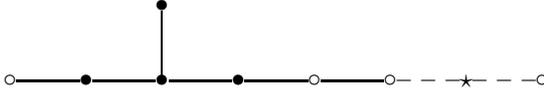
\begin{xy}
\ar@{-} (0,0) *{\circ}="A"; (10,0) *{\bullet}="B"
\ar@{-}"B"; (20,0) *{\bullet}="C"
\ar@{-} "C";  (20,10) *{\bullet}
\ar@{-} "C";  (30,0) *{\bullet}="D"
\ar@{-} "D";  (40,0) *{\circ}="E"
\ar@{-} "E";  (50,0) *{\circ}="F"
\ar@{--}"F"; (60,0) *{\star}="G"
\ar@{--}"G"; (70,0) *{\circ}
\end{xy}

\subsubsection{}
\ 

$(\frak{g},\frak{g}^\sigma)=
(\frak{e}_{6}^{\bb{C}},
\frak{sl}(6,\bb{C})\oplus\frak{sl}(2,\bb{C}))$

$(\frak{g},\frak{g}^{\theta\sigma})
=(\frak{e}_{6}^{\bb{C}}, \frak{e}_{6(2)})$

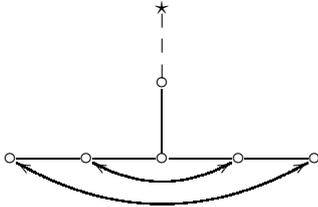
\begin{xy}
\ar@{-} (0,0) *{\circ}="A"; (10,0) *{\circ}="B"
\ar@{-}"B"; (20,0) *{\circ}="C"
\ar@{-} "C";  (30,0) *{\circ}="D"
\ar@{-} "D";  (40,0) *{\circ}="F"
\ar@{-} "C";  (20,10) *{\circ}="E"
\ar@{--} "E";  (20,20) *{\star}
\ar@/_{6mm}/@{<->} "A"; "F"
\ar@/_{3mm}/@{<->} "B"; "D"
\end{xy}

\subsubsection{}
\ 

$(\frak{g},\frak{g}^\sigma)=
(\frak{e}_{6}^{\bb{C}},
\frak{so}(10,\bb{C})\oplus \bb{C})$

$(\frak{g},\frak{g}^{\theta\sigma})
=(\frak{e}_{6}^{\bb{C}}, \frak{e}_{6(-14)})$

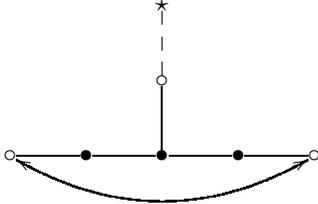
\begin{xy}
\ar@{-} (0,0) *{\circ}="A"; (10,0) *{\bullet}="B"
\ar@{-}"B"; (20,0) *{\bullet}="C"
\ar@{-} "C";  (30,0) *{\bullet}="D"
\ar@{-} "D";  (40,0) *{\circ}="F"
\ar@{-} "C";  (20,10) *{\circ}="E"
\ar@{--} "E";  (20,20) *{\star}
\ar@/_{6mm}/@{<->} "A"; "F"
\end{xy}

\subsubsection{}
\ 

$(\frak{g},\frak{g}^\sigma)=
(\frak{e}_{7}^{\bb{C}},
\frak{so}(12,\bb{C})\oplus \frak{sl}(2,\bb{C}))$

$(\frak{g},\frak{g}^{\theta\sigma})
=(\frak{e}_{7}^{\bb{C}}, \frak{e}_{7(-5)})$

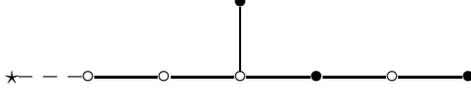
\begin{xy}
\ar@{--} (0,0) *{\star}="A"; (10,0) *{\circ}="B"
\ar@{-} "B"; (20,0) *{\circ}="C"
\ar@{-} "C"; (30,0) *{\circ}="D"
\ar@{-} "D"; (40,0) *{\bullet}="E"
\ar@{-} "E"; (50,0) *{\circ}="F"
\ar@{-} "F"; (60,0) *{\bullet}="G"
\ar@{-} "D"; (30,10)*{\bullet}
\end{xy}

\subsubsection{}
\ 

$(\frak{g},\frak{g}^\sigma)=
(\frak{e}_{7}^{\bb{C}},
\frak{e}_{6}^{\bb{C}}\oplus \bb{C})$

$(\frak{g},\frak{g}^{\theta\sigma})
=(\frak{e}_{7}^{\bb{C}}, \frak{e}_{7(-25)})$

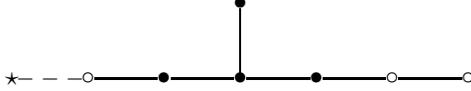
\begin{xy}
\ar@{--} (0,0) *{\star}="A"; (10,0) *{\circ}="B"
\ar@{-} "B"; (20,0) *{\bullet}="C"
\ar@{-} "C"; (30,0) *{\bullet}="D"
\ar@{-} "D"; (40,0) *{\bullet}="E"
\ar@{-} "E"; (50,0) *{\circ}="F"
\ar@{-} "F"; (60,0) *{\circ}="G"
\ar@{-} "D"; (30,10)*{\bullet}
\end{xy}

\subsubsection{}
\ 

$(\frak{g},\frak{g}^\sigma)=
(\frak{e}_{8}^{\bb{C}},
\frak{e}_{7}^{\bb{C}}\oplus \frak{sl}(2,\bb{C}))$

$(\frak{g},\frak{g}^{\theta\sigma})
=(\frak{e}_{8}^{\bb{C}}, \frak{e}_{8(-24)})$

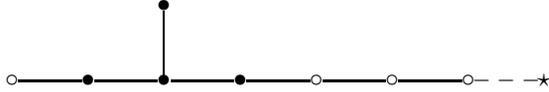
\begin{xy}
\ar@{-} (0,0) *{\circ}="A"; (10,0) *{\bullet}="B"
\ar@{-} "B"; (20,0) *{\bullet}="C"
\ar@{-} "C"; (20,10)*{\bullet}
\ar@{-} "C"; (30,0) *{\bullet}="D"
\ar@{-} "D"; (40,0) *{\circ}="E"
\ar@{-} "E"; (50,0) *{\circ}="F"
\ar@{-} "F"; (60,0) *{\circ}="G"
\ar@{--} "G"; (70,0) *{\star}
\end{xy}

\subsection{Case of holomorphic symmetric pairs}
\label{subsec:hol}

\subsubsection{}
\ 

$k,l,m-k,n-l\geq 1$

$(\frak{g},\frak{g}^\sigma)=
(\frak{su}(m,n),
\frak{su}(k,l)\oplus\frak{su}(m-k,n-l)\oplus\frak{u}(1))$

\smallskip

\begin{xy}
(7.5, 4)*{q},
(112.5, 4)*{s}
\ar@{-} (0,-1) *{\circ}="A"; (5,-1)
\ar@{.} (5,-1) ; (10,-1)
\ar@{-} (10,-1) ; (15,-1) *{\circ}="B"
\ar@{-} "B"; (25,-1) *+[F]{\quad p \quad}="D"
\ar@{-} "D"; (35,-1) *{\circ}="E"
\ar@{-} "E"; (40,-1) *{}
\ar@{.} (40,-1); (45,-1) *{}
\ar@{-} (45,-1); (50,-1) *{\circ}="F"
\ar@{--} (60,-1) *{\star}="C"; "F"
\ar@/_{7mm}/@{<->} "A"; "F"
\ar@/_{3mm}/@{<->} "B"; "E"
\ar@{--} "C"; (70,-1) *{\circ}="G"
\ar@{-} "G"; (75,-1)
\ar@{.} (75,-1) ; (80,-1)
\ar@{-} (80,-1) ; (85,-1) *{\circ}="H"
\ar@{-} "H"; (95,-1) *+[F]{\quad r \quad}="I"
\ar@{-} "I"; (105,-1) *{\circ}="J"
\ar@{-} "J"; (110,-1) *{}
\ar@{.} (110,-1); (115,-1) *{}
\ar@{-} (115,-1); (120,-1) *{\circ}="K"
\ar@/_{7mm}/@{<->} "G"; "K"
\ar@/_{3mm}/@{<->} "H"; "J"
$\overbrace{\hspace*{15mm}}$
$\hspace*{89mm}\overbrace{\hspace*{15mm}}$
\end{xy}

\begin{align*}
p=&\begin{cases}
m-2k-1&\text{if $k<m-k$}\\
2k-m-1&\text{if $k>m-k$}\\
-1&\text{if $k=m-k$}
\end{cases}&
\qquad 
q=&\begin{cases}
k&\text{if $k<m-k$}\\
m-k&\text{if $k>m-k$}\\
k-1&\text{if $k=m-k$}
\end{cases}&
\\
r=&\begin{cases}
n-2l-1&\text{if $l<n-l$}\\
2l-n-1&\text{if $l>n-l$}\\
-1&\text{if $l=n-l$}
\end{cases}&
\qquad 
s=&\begin{cases}
l&\text{if $l<n-l$}\\
n-l&\text{if $l>n-l$}\\
l-1&\text{if $l=n-l$}
\end{cases}&
\end{align*}

\subsubsection{}
\ 

$n\geq 2$

$(\frak{g},\frak{g}^\sigma)=
(\frak{su}(n,n),\frak{so}^*(2n))$

$(\frak{g},\frak{g}^{\theta\sigma})
=(\frak{su}(n,n),\frak{sp}(n,\bb{R}))$

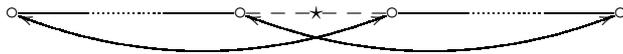
\begin{xy}
\ar@{-} (0,0) *{\circ}="A"; (10,0)
\ar@{.} (10,0); (20,0) 
\ar@{-} (20,0); (30,0)*{\circ}="B"
\ar@{--} "B"; (40,0) *{\star}="C"
\ar@{--} "C"; (50,0) *{\circ}="D"
\ar@{-} "D"; (60,0)
\ar@{.} (60,0); (70,0)
\ar@{-} (70,0); (80,0) *{\circ}="E"
\ar@/_{5mm}/@{<->}"A"; "D"
\ar@/_{5mm}/@{<->}"B"; "E"
\end{xy}

\subsubsection{}
\ 

$m,2n-m+1\geq 1$

$(\frak{g},\frak{g}^\sigma)=
(\frak{so}(2,2n+1),\frak{so}(2,m)\oplus\frak{so}(2n-m+1))$

\begin{xy}
\ar@{--} (0,0) *{\star}="A"; (10,0) *{\circ}="B"
\ar@{-} "B"; (20,0) 
\ar@{.} (20,0); (30,0)
\ar@{-} (30,0); (40,0) *{\circ}="C"
\ar@{-} "C"; (50,0) *{\bullet}="D"
\ar@{-} "D"; (60,0)
\ar@{.} (60,0); (70,0)
\ar@{-} (70,0); (80,0) *{\bullet}="E"
\ar@{=>} "E"; (90,0) *{\bullet}
\end{xy}

\subsubsection{}
\ 

$m,2n-m\geq 1$

$(\frak{g},\frak{g}^\sigma)=
(\frak{so}(2,2n),\frak{so}(2,m)\oplus\frak{so}(2n-m))$

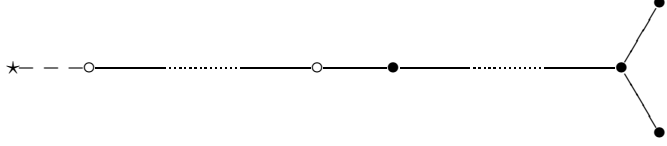
\begin{xy}
\ar@{--} (0,0) *{\star}="A"; (10,0) *{\circ}="B"
\ar@{-} "B"; (20,0) 
\ar@{.} (20,0); (30,0)
\ar@{-} (30,0); (40,0) *{\circ}="C"
\ar@{-} "C"; (50,0) *{\bullet}="D"
\ar@{-} "D"; (60,0)
\ar@{.} (60,0); (70,0)
\ar@{-} (70,0); (80,0) *{\bullet}="E"
\ar@{-} "E"; (85,-8.6) *{\bullet}
\ar@{-} "E"; (85,8.6) *{\bullet}
\end{xy}

\subsubsection{}
\ 

$m,n-m\geq 2$

$(\frak{g},\frak{g}^\sigma)=
(\frak{so}^*(2n),\frak{u}(m,n-m))$

$(\frak{g},\frak{g}^{\theta\sigma})
=(\frak{so}^*(2n),\frak{so}^*(2m)
\oplus\frak{so}^*(2n-2m))$

Case: $n\geq 5$

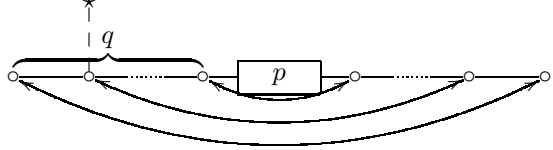
\begin{xy}
(12.5, 4)*{q},
\ar@{-} (0,-1) *{\circ}="A"; (10,-1) *{\circ}="B"
\ar@{-} "B" ; (15,-1)
\ar@{.} (15,-1) ; (20,-1)
\ar@{-} (20,-1) ; (25,-1) *{\circ}="C"
\ar@{-} "C"; (35,-1) *+[F]{\quad p \quad}="D"
\ar@{-} "D"; (45,-1) *{\circ}="E"
\ar@{-} "E"; (50,-1) *{}
\ar@{.} (50,-1); (55,-1) *{}
\ar@{-} (55,-1); (60,-1) *{\circ}="F"
\ar@{-} "F"; (70,-1) *{\circ}="G"
\ar@{--} (10,9) *{\star}="H"; "B"
\ar@/_{9mm}/@{<->} "A"; "G"
\ar@/_{6mm}/@{<->} "B"; "F"
\ar@/_{3mm}/@{<->} "C"; "E"
$\overbrace{\hspace*{25mm}}$
\end{xy}

\[
p=\begin{cases}
n-2m-1&\text{if $m<n-m$}\\
2m-n-1&\text{if $m>n-m$}\\
-1&\text{if $m=n-m$}
\end{cases}
\qquad 
q=\begin{cases}
m&\text{if $m<n-m$}\\
n-m&\text{if $m>n-m$}\\
m-1&\text{if $m=n-m$}
\end{cases}
\]

Case: $m=2, n=4$

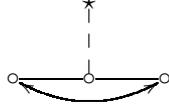
\begin{xy}
\ar@{-} (0,0) *{\circ}="A"; (10,0) *{\circ}="B"
\ar@{-} "B" ; (20,0) *{\circ}="C"
\ar@{--} "B"; (10,10) *{\star}
\ar@/_{3mm}/@{<->} "A"; "C"
\end{xy}

\subsubsection{}
\ 

$m,n-m\geq 1$

$(\frak{g},\frak{g}^\sigma)=
(\frak{sp}(n,\bb{R}),\frak{u}(m,n-m))$

$(\frak{g},\frak{g}^{\theta\sigma})
=(\frak{sp}(n,\bb{R}),\frak{sp}(m,\bb{R})
\oplus\frak{sp}(n-m,\bb{R}))$

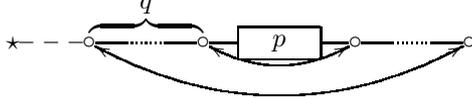
\begin{xy}
(17.5, 4)*{q},
\ar@{--} (0,-1) *{\star}; (10,-1) *{\circ}="A"
\ar@{-} "A" ; (15,-1) 
\ar@{.} (15,-1) ; (20,-1) 
\ar@{-} (20,-1) ; (25,-1) *{\circ}="B"
\ar@{-} "B"; (35,-1) *+[F]{\quad p \quad}="D"
\ar@{-} "D"; (45,-1) *{\circ}="E"
\ar@{-} "E"; (50,-1) *{}
\ar@{.} (50,-1); (55,-1) *{}
\ar@{-} (55,-1); (60,-1) *{\circ}="F"
\ar@/_{7mm}/@{<->} "A"; "F"
\ar@/_{3mm}/@{<->} "B"; "E"
$\hspace*{10mm}\overbrace{\hspace*{15mm}}$
\end{xy}

\[
p=\begin{cases}
n-2m-1&\text{if $m<n-m$}\\
2m-n-1&\text{if $m>n-m$}\\
-1&\text{if $m=n-m$}
\end{cases}
\qquad 
q=\begin{cases}
m&\text{if $m<n-m$}\\
n-m&\text{if $m>n-m$}\\
m-1&\text{if $m=n-m$}
\end{cases}
\]

\subsubsection{}
\ 

$(\frak{g},\frak{g}^\sigma)=
(\frak{e}_{6(-14)},\frak{su}(4,2)\oplus\frak{su}(2))$

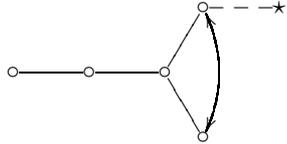
\begin{xy}
\ar@{-} (0,0) *{\circ}="A"; (10,0) *{\circ}="B"
\ar@{-} "B"; (20,0) *{\circ}="C"
\ar@{-} "C"; (25,8.6) *{\circ}="D"
\ar@{-} "C"; (25,-8.6) *{\circ}="E"
\ar@{--} "D"; (35,8.6)*{\star}
\ar@/^/@{<->} "D"; "E"
\end{xy}

\subsubsection{}
\ 

$(\frak{g},\frak{g}^\sigma)=
(\frak{e}_{6(-14)},\frak{so}^*(10)\oplus\frak{so}(2))$

$(\frak{g},\frak{g}^{\theta\sigma})
=(\frak{e}_{6(-14)},\frak{su}(5,1)\oplus\frak{sl}(2,\bb{R}))$

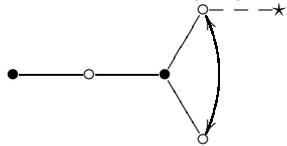
\begin{xy}
\ar@{-} (0,0) *{\bullet}="A"; (10,0) *{\circ}="B"
\ar@{-} "B"; (20,0) *{\bullet}="C"
\ar@{-} "C"; (25,8.6) *{\circ}="D"
\ar@{-} "C"; (25,-8.6) *{\circ}="E"
\ar@{--} "D"; (35,8.6)*{\star}
\ar@/^/@{<->} "D"; "E"
\end{xy}

\subsubsection{}
\ 

$(\frak{g},\frak{g}^\sigma)=
(\frak{e}_{7(-25)},\frak{e}_{6(-14)}\oplus\frak{so}(2))$

$(\frak{g},\frak{g}^{\theta\sigma})
=(\frak{e}_{7(-25)},\frak{so}(2,10)\oplus\frak{sl}(2,\bb{R}))$

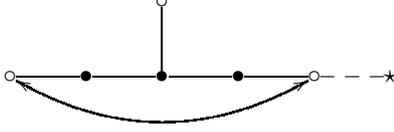
\begin{xy}
\ar@{-} (0,0) *{\circ}="A"; (10,0) *{\bullet}="B"
\ar@{-} "B"; (20,0) *{\bullet}="C"
\ar@{-} "C"; (30,0) *{\bullet}="D"
\ar@{-} "C"; (20,10) *{\circ}
\ar@{-} "D"; (40,0)*{\circ}="E"
\ar@{--} "E"; (50,0)*{\star}
\ar@/_{6mm}/@{<->} "A"; "E"
\end{xy}

\subsubsection{}
\ 

$(\frak{g},\frak{g}^\sigma)=
(\frak{e}_{7(-25)},\frak{su}(6,2))$

$(\frak{g},\frak{g}^{\theta\sigma})
=(\frak{e}_{7(-25)},\frak{so}^*(12)\oplus\frak{su}(2))$

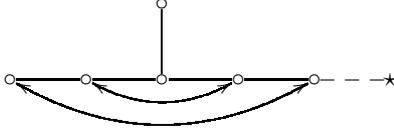
\begin{xy}
\ar@{-} (0,0) *{\circ}="A"; (10,0) *{\circ}="B"
\ar@{-} "B"; (20,0) *{\circ}="C"
\ar@{-} "C"; (30,0) *{\circ}="D"
\ar@{-} "C"; (20,10) *{\circ}
\ar@{-} "D"; (40,0)*{\circ}="E"
\ar@{--} "E"; (50,0)*{\star}
\ar@/_{6mm}/@{<->} "A"; "E"
\ar@/_{3mm}/@{<->} "B"; "D"
\end{xy}

\newpage


\begin{table}[H]
\begin{center}
\begin{tabular}{ccc}
\hline
\rule[0pt]{0pt}{12pt}
\rule[-5pt]{0pt}{5pt}
$\frak{g}$&
\multicolumn{2}{r}{
\hspace{15pt}
$a=a_1e_1+a_2e_2+\cdots\qquad$
See Settings \ref{sumn} to \ref{e82}.}\\

\hline

\rule[0pt]{0pt}{12pt}
\rule[-5pt]{0pt}{5pt}
& holomorphic & anti-holomorphic\\

\hline

\rule[0pt]{0pt}{12pt}
$\frak{su}(m,n)$&{
\rule[-5pt]{0pt}{5pt}
$a_m\geq a_{m+1}\ \ $}
& $a_{m+n}\geq a_{1}$
\\ \hline

\rule[0pt]{0pt}{12pt}
$\frak{so}(2,2n)$&{
\rule[-5pt]{0pt}{5pt}
$a_1\geq a_2$}
&{$-a_1\geq a_2$}\\ 
\hline

\rule[0pt]{0pt}{12pt}
$\frak{so}(2,2n+1)$&{
\rule[-5pt]{0pt}{5pt}
$a_1\geq a_2$}&
$-a_1\geq a_2$
\\ 
\hline

\rule[0pt]{0pt}{12pt}
$\frak{so}^*(2n)$&{
\rule[-5pt]{0pt}{5pt}
$a_{n-1}+a_n\geq 0$}&
$a_1+a_2\leq 0$\\ 
\hline

\rule[0pt]{0pt}{12pt}
$\frak{sp}(n,\bb{R})$&{
\rule[-5pt]{0pt}{5pt}
$a_{n}\geq 0$}
&$a_1\leq 0$\\ 
\hline

\rule[0pt]{0pt}{12pt}
$\frak{e}_{6(-14)}$&{
\rule[-5pt]{0pt}{5pt}
$a_6\geq a_1+a_2+a_3+a_4+a_5$}
&$-a_6\geq a_1+a_2+a_3+a_4-a_5$\\ 
\hline

\rule[0pt]{0pt}{12pt}
$\frak{e}_{7(-25)}$&{
\rule[-5pt]{0pt}{5pt}
$a_6\geq a_{5}$}
&$a_8\leq a_7$\\ 
\hline

\end{tabular}
\end{center}
\caption{holomorphic parabolic subalgebras}
\label{holparablist}
\vspace*{50pt}
\begin{center}
\begin{tabular}{cc}
\hline
\rule[0pt]{0pt}{12pt}
\rule[-5pt]{0pt}{7pt}
$\qquad\frak{g}\qquad$&$\qquad\frak{g}^\sigma\qquad$\\
\hline

\rule[0pt]{0pt}{12pt}
$\quad\frak{su}(m,n)\ m\neq n$&{
\rule[-5pt]{0pt}{5pt}
$\frak{su}(k,l)\oplus\frak{su}(m-k,n-l)\oplus\frak{u}(1)$}\\
\hline

\rule[0pt]{0pt}{12pt}
$\frak{su}(n,n)$&{
\rule[-5pt]{0pt}{5pt}
$\frak{su}(k,l)\oplus\frak{su}(n-k,n-l)\oplus\frak{u}(1)$}\\
&{\rule[-5pt]{0pt}{5pt}
$\frak{so}^*(2n)$}\\ 
&{\rule[-5pt]{0pt}{5pt}
$\frak{sp}(n,\bb{R})$}\\
\hline

\rule[0pt]{0pt}{12pt}
$\frak{so}(2,2n)$&{
\rule[-5pt]{0pt}{5pt}
$\frak{so}(2,k)\oplus\frak{so}(2n-k)$}\\
&{\rule[-5pt]{0pt}{5pt}
$\frak{u}(1,n)$}\\ 
\hline

\rule[0pt]{0pt}{12pt}
$\frak{so}(2,2n+1)$&{
\rule[-5pt]{0pt}{5pt}
$\frak{so}(2,k)\oplus\frak{so}(2n-k+1)$}\\ 
\hline

\rule[0pt]{0pt}{12pt}
$\frak{so}^*(2n)$&{
\rule[-5pt]{0pt}{5pt}
$\frak{u}(m,n-m)$}\\ 
&{\rule[-5pt]{0pt}{5pt}
$\frak{so}^*(2m)\oplus\frak{so}^*(2n-2m)$}\\ 
\hline

\rule[0pt]{0pt}{12pt}
$\frak{sp}(n,\bb{R})$&{
\rule[-5pt]{0pt}{5pt}
$\frak{u}(m,n-m)$}\\ 
&{\rule[-5pt]{0pt}{5pt}
$\frak{sp}(m,\bb{R})\oplus\frak{sp}(n-m,\bb{R})$}\\ 
\hline

\rule[0pt]{0pt}{12pt}
$\frak{e}_{6(-14)}$&{
\rule[-5pt]{0pt}{5pt}
$\frak{so}(10)\oplus \frak{so}(2)$}\\ 
&{\rule[-5pt]{0pt}{5pt}
$\frak{so}(2,8)\oplus \frak{so}(2)$}\\ 
&{\rule[-5pt]{0pt}{5pt}
$\frak{su}(4,2)\oplus \frak{su}(2)$}\\ 
&{\rule[-5pt]{0pt}{5pt}
$\frak{so}^*(10)\oplus \frak{so}(2)$}\\ 
&{\rule[-5pt]{0pt}{5pt}
$\frak{su}(5,1)\oplus \frak{sl}(2,\bb{R})$}\\ 
\hline

\rule[0pt]{0pt}{12pt}
$\frak{e}_{7(-25)}$&{
\rule[-5pt]{0pt}{5pt}
$\frak{e}_{6(-78)}\oplus \frak{so}(2)$}\\ 
&{\rule[-5pt]{0pt}{5pt}
$\frak{e}_{6(-14)}\oplus \frak{so}(2)$}\\ 
&{\rule[-5pt]{0pt}{5pt}
$\frak{so}(2,10)\oplus \frak{sl}(2,\bb{R})$}\\ 
&{\rule[-5pt]{0pt}{5pt}
$\frak{su}(6,2)$}\\ 
&{\rule[-5pt]{0pt}{5pt}
$\frak{so}^*(12)\oplus \frak{su}(2)$}\\ 
\hline

\end{tabular}
\end{center}
\caption{symmetric pairs of holomorphic type}
\label{holpairlist}
\end{table}

\newpage

\begin{table}[H]
\begin{center}
\begin{tabular}{ccr}
\hline
\rule[0pt]{0pt}{12pt}
$\frak{g}$&$\frak{g}^\sigma$&
$a=a_1e_1+a_2e_2+\cdots$
\\
\multicolumn{3}{r}{\rule[-5pt]{0pt}{5pt}
See Settings \ref{sumn} to \ref{e82}.}\\
\hline

\rule[0pt]{0pt}{12pt}
$\frak{su}(m,n)$&$\frak{su}(m,k)\oplus\frak{su}(n-k)\oplus\frak{u}(1)$&
\rule[-5pt]{0pt}{5pt}
$a_{m+n}\geq a_1$,\\ 
&&\rule[-5pt]{0pt}{5pt}$a_l \geq a_{m+1}$ and 
$a_{m+n} \geq a_{l+1} (1\leq \exists l\leq m-1)$,\\ 
&& \rule[-5pt]{0pt}{5pt}or $a_m\geq a_{m+1}$\\ \hline

\rule[0pt]{0pt}{12pt}
$\frak{su}(2,2n)$&$\frak{sp}(1,n)$&
\rule[-5pt]{0pt}{5pt}
$a_1\geq a_3$ and $a_{2n+2}\geq a_2$\\
\rule[-5pt]{0pt}{5pt}
$n\neq 1$&&\\ 
\hline

\rule[0pt]{0pt}{12pt}
$\frak{su}(2,2)$&$\frak{sp}(1,1)$&
\rule[-5pt]{0pt}{5pt}
$a_1\geq a_3\geq a_4\geq a_2$\\
&&
\rule[-5pt]{0pt}{5pt}
 or $a_3\geq a_1\geq a_2\geq a_4$\\
\hline

\rule[0pt]{0pt}{12pt}
$\frak{so}(2m,2n)$&$\frak{so}(2m,k)\oplus\frak{so}(2n-k)$&
\rule[-5pt]{0pt}{5pt}
$|a_m|\geq |a_{m+1}|$\\ \hline

\rule[0pt]{0pt}{12pt}
$\frak{so}(2m,2n+1)$&$\frak{so}(2m,k)\oplus\frak{so}(2n-k+1)$&
\rule[-5pt]{0pt}{5pt}
$|a_m|\geq a_{m+1}$\\
\hline

\rule[0pt]{0pt}{12pt}
$\frak{so}(4,2n)$&$\frak{u}(2,n)_1$&
\rule[-5pt]{0pt}{5pt}
$-a_2\geq |a_3|$\\
\rule[-5pt]{0pt}{5pt}
$n\neq 2$&$\frak{u}(2,n)_2$&
\rule[-5pt]{0pt}{5pt}
$a_2\geq |a_3|$\\
\hline

\rule[0pt]{0pt}{12pt}
$\frak{so}(4,4)$&$\frak{u}(2,2)_{11}$&
\rule[-5pt]{0pt}{5pt}
$-a_2\geq a_3$ or $-a_4\geq a_1$\\
&$\frak{u}(2,2)_{12}$&
\rule[-5pt]{0pt}{5pt}
$-a_2\geq a_3$ or $a_4\geq a_1$\\
&$\frak{u}(2,2)_{21}$&
\rule[-5pt]{0pt}{5pt}
$a_2\geq a_3$ or $-a_4\geq a_1$\\
&$\frak{u}(2,2)_{22}$&
\rule[-5pt]{0pt}{5pt}
$a_2\geq a_3$ or $a_4\geq a_1$\\
\hline

\rule[0pt]{0pt}{12pt}
$\frak{sp}(m,n)$&$\frak{sp}(m,k)\oplus\frak{sp}(n-k)$&
\rule[-5pt]{0pt}{5pt}
$a_m\geq a_{m+1}$\\
\hline

\rule[0pt]{0pt}{12pt}
$\frak{f}_{4(4)}$&$\frak{sp}(2,1)\oplus\frak{su}(2)$&
\rule[-5pt]{0pt}{5pt}
$a_1+a_2+a_3\leq a_4$\\
&$\frak{so}(5,4)$&
\rule[-5pt]{0pt}{5pt}
$a_1+a_2+a_3\leq a_4$\\
\hline

\rule[0pt]{0pt}{12pt}
$\frak{e}_{6(2)}$&$\frak{so}(6,4)\oplus\frak{so}(2)$&
\rule[-5pt]{0pt}{5pt}
$a_1+a_2+a_3-a_4-a_5-a_6\leq 2a_7$\\
&$\frak{su}(4,2)\oplus\frak{su}(2)$&
\rule[-5pt]{0pt}{5pt}
$a_1+a_2+a_3-a_4-a_5-a_6\leq 2a_7$\\
&$\frak{sp}(3,1)$&
\rule[-5pt]{0pt}{5pt}
$a_1+a_2+a_3-a_4-a_5-a_6\leq 2a_7$\\
&$\frak{f}_{4(4)}$&
\rule[-5pt]{0pt}{5pt}
$a_1+a_2+a_3-a_4-a_5-a_6\leq 2a_7$\\
\hline

\rule[0pt]{0pt}{12pt}
$\frak{e}_{7(-5)}$&$\frak{so}(8,4)\oplus \frak{su}(2)$&
\rule[-5pt]{0pt}{5pt}
$a_1+a_2+a_3+a_4+a_5-a_6\leq 2a_7$\\
&$\frak{su}(6,2)$&
\rule[-5pt]{0pt}{5pt}
$a_1+a_2+a_3+a_4+a_5-a_6\leq 2a_7$\\
&$\frak{e}_{6(2)}\oplus \frak{so}(2)$&
\rule[-5pt]{0pt}{5pt}
$a_1+a_2+a_3+a_4+a_5-a_6\leq 2a_7$\\
\hline

\rule[0pt]{0pt}{12pt}
$\frak{e}_{8(-24)}$&$\frak{so}(12,4)$&
\rule[-5pt]{0pt}{5pt}
$a_7\geq a_6$\\
&$\frak{e}_{7(-5)}\oplus \frak{su}(2)$&
\rule[-5pt]{0pt}{5pt}
$a_7\geq a_6$\\
\hline

\end{tabular}
\end{center}
\caption{$(\frak{g},\frak{g}^\sigma,\frak{q})$ of discrete series type}
\label{dslist}
\vspace*{140pt}
\end{table}

\newpage

\begin{table}[H]
\begin{center}
\begin{tabular}{ccr}
\hline
\rule[0pt]{0pt}{12pt}
$\qquad\frak{g}\qquad$&$\frak{g}^\sigma$&
\rule[-5pt]{0pt}{5pt}
$a=a_1e_1+a_2e_2+\cdots$
\\
\multicolumn{3}{r}{
\rule[-5pt]{0pt}{5pt}
See Settings \ref{sumn} to \ref{e82}.}
\\
\hline

\rule[0pt]{0pt}{12pt}
$\frak{su}(2m,2n)$&$\frak{sp}(m,n)$&
\rule[-5pt]{0pt}{5pt}
$(a_1,\dots,a_{2m} ; \, a_{2m+1},\dots,a_{2m+2n})$\\
\multicolumn{3}{r}{
\rule[-5pt]{0pt}{5pt}
$=(s,0,\dots,0;\, t,0,\dots,0)$, 
$(0,\dots,0,-s;\, 0,\dots,0,-t)$,}\\
\multicolumn{3}{r}{
\rule[-5pt]{0pt}{5pt}
$(s,0,\dots,0,-t;\, 0,\dots,0)$ 
or $(0,\dots,0;\, s,0,\dots,0,-t)$}\\
&&\rule[-5pt]{0pt}{5pt}
$\mod \bb{I}_{2m+2n} \quad (s,t\geq 0)$\\
\hline

\rule[0pt]{0pt}{12pt}
$\frak{so}(2m+1,2n)$&$\frak{so}(2m+1,k)\oplus\frak{so}(2n-k)$&
\rule[-5pt]{0pt}{5pt}
$a_{m+1}=\cdots= a_{m+n}=0$\\
\hline

\rule[0pt]{0pt}{12pt}
$\frak{so}(2m+1,2n+1)$&$\frak{so}(2m+1,k)\oplus\frak{so}(2n-k+1)$&
\rule[-5pt]{0pt}{5pt}
$a_{m+1}=\cdots= a_{m+n}=0$\\
\hline

\rule[0pt]{0pt}{12pt}
$\frak{so}(2m,2n)$&$\frak{u}(m,n)$&
\rule[-5pt]{0pt}{5pt}
$(a_1,\dots,a_m ; \, a_{m+1},\dots,a_{m+n})$\\
&\rule[-5pt]{0pt}{5pt}&
$=(s,0,\dots,0 ; \, 0,\dots,0)$\\ 
&\rule[-5pt]{0pt}{5pt}&
or $(0,\dots,0 ; \, s,0,\dots,0)$\\ 
\hline

\rule[0pt]{0pt}{12pt}
$\frak{so}^*(2n)$&$\frak{so}^*(2n-2)\oplus\frak{so}(2)$&
\rule[-5pt]{0pt}{5pt}
$(a_1,\dots,a_n)=(\underbrace{s,\dots,s}_k,-s,\dots,-s)$\\
&&
\rule[-5pt]{0pt}{5pt}
$(1\leq \exists k\leq n-1)$
\\
&$\frak{u}(n-1,1)$&
\rule[-5pt]{0pt}{5pt}
$(a_1,\dots,a_n)=(\underbrace{s,\dots,s}_k,-s,\dots,-s)$\\
&&
\rule[-5pt]{0pt}{5pt}
$(1\leq \exists k\leq n-1)$
\\
\hline

\rule[0pt]{0pt}{12pt}
$\frak{sp}(m,n)$
&$\frak{sp}(k,l)\oplus\frak{sp}(m-k,n-l)$&
\rule[-5pt]{0pt}{5pt}
$(a_1,\dots,a_m ; \, a_{m+1},\dots,a_{m+n})$\\
&\rule[-5pt]{0pt}{5pt}
$k,l,m-k,n-l\geq 1$&
$=(s,0,\dots,0 ; \, 0,\dots,0)$\\ 
&\rule[-9pt]{0pt}{5pt}
&
or $(0,\dots,0 ; \, s,0,\dots,0)$\\
&$\frak{sp}(m,k)\oplus\frak{sp}(n-k)$&
\rule[-5pt]{0pt}{5pt}
$(a_1,\dots,a_m ; \, a_{m+1},\dots,a_{m+n})$\\
&\rule[-5pt]{0pt}{5pt}&
$=(0,\dots,0 ; \, s,0,\dots,0),$\\
\multicolumn{3}{r}{
\rule[-5pt]{0pt}{5pt}
$a_{l-1}\geq a_{m+1}$ and $a_l=a_{m+2}=0\ (2\leq \exists l\leq m)$}\\
\hline

\rule[0pt]{0pt}{12pt}
$\frak{sl}(2n,\bb{C})$&$\frak{sp}(n,\bb{C})$&
\rule[-5pt]{0pt}{5pt}
$(a_1,\dots,a_n)=(s,0,\dots,0)$
or $(0,\dots,0,s)$
\\
&&\rule[-5pt]{0pt}{5pt}
$\mod \bb{I}_{2n}$\\
&$\frak{su}^*(2n)$&
\rule[-5pt]{0pt}{5pt}
$(a_1,\dots,a_n)=(s,0,\dots,0)$
or $(0,\dots,0,s)$\\
&&\rule[-5pt]{0pt}{5pt}
$\mod \bb{I}_{2n}$\\
\hline

\rule[0pt]{0pt}{12pt}
$\frak{so}(2n,\bb{C})$&$\frak{so}(2n-1,\bb{C})$&
\rule[-5pt]{0pt}{5pt}
$(a_1,\dots,a_n)=(s,\dots,s)$\\
&$\frak{so}(2n-1,1)$&
\rule[-5pt]{0pt}{5pt}
$(a_1,\dots,a_n)=(s,\dots,s)$\\
\hline

\rule[0pt]{0pt}{12pt}
$\frak{f}_{4(-20)}$&$\frak{so}(8,1)$&
\rule[-5pt]{0pt}{5pt}
$(a_1,a_2,a_3,a_4)=(s,s,s,s)$ or $(s,s,0,0)$\\
\hline

\rule[0pt]{0pt}{12pt}
$\frak{e}_{6(2)}$&$\frak{so}^*(10)\oplus\frak{so}(2)$&
\rule[-5pt]{0pt}{5pt}
$(a_1,\dots,a_7)=(s,s,s,s,t,t,0),$\\
&& \rule[-5pt]{0pt}{5pt}
or $(s,s,t,t,t,t,0)$\\
\hline

\rule[0pt]{0pt}{12pt}
$\frak{e}_{6(-14)}$&$\frak{so}(2,8)\oplus\frak{so}(2)$&
\rule[-5pt]{0pt}{5pt}
$(a_1,\dots,a_6)=(s,s,s,s,s,s),$\\
&& \rule[-5pt]{0pt}{5pt}
or $(s,s,s,s,-s,-s)$\\
\hline

\rule[0pt]{0pt}{12pt}
$\frak{e}_{6(-14)}$&$\frak{f}_{4(-20)}$&
\rule[-5pt]{0pt}{5pt}
$(a_1,\dots,a_6)=(s,s,0,0,0,0),$\\
&& \rule[-5pt]{0pt}{5pt}
or $(a_1,\dots,a_6)=(s,s,s,s,t,t)$\\
\hline

\rule[0pt]{0pt}{12pt}
$\frak{e}_{7(-5)}$&$\frak{e}_{6(-14)}\oplus \frak{so}(2)$&
\rule[-5pt]{0pt}{5pt}
$(a_1,\dots,a_7)=(s,s,s,s,s,s,0)$\\
\hline

&&
$s,t \in\bb{R},\quad \bb{I}_{n}=(\underbrace{1,\dots,1}_n)$\\

\end{tabular}
\end{center}
\caption{$(\frak{g},\frak{g}^\sigma,\frak{q})$ of isolated type}
\label{isolist}
\vspace*{50pt}
\end{table}

\newpage

\begin{table}[H]
\begin{center}
\begin{tabular}{ccc}
\hline
\rule[0pt]{0pt}{12pt}
\hspace{40pt}$\frak{g}$\hspace{40pt}
&$\frak{g}^\sigma$ \rule[-5pt]{0pt}{5pt}\\
\hline

\rule[0pt]{0pt}{12pt}
$\frak{su}(n,n)$&\rule[-5pt]{0pt}{5pt}
$\frak{sl}(n,\bb{C})\oplus\bb{R}$&$n\geq 1$\\
\hline

\rule[0pt]{0pt}{12pt}
$\frak{so}(m,n)$&\rule[-5pt]{0pt}{5pt}
$\frak{so}(k,l)\oplus \frak{so}(m-k,n-l)$
&\rule[-5pt]{0pt}{5pt}
$k,l,m-k,n-l\geq 1$\\
\hline

\rule[0pt]{0pt}{12pt}
$\frak{so}(n,n)$&
\rule[-5pt]{0pt}{5pt}
$\frak{so}(n,\bb{C})$&$n\geq 3$\\
&\rule[-5pt]{0pt}{5pt}
$\frak{gl}(n,\bb{R})$&$n\geq 3$\\
\hline

\rule[0pt]{0pt}{12pt}
$\frak{so}^*(4n)$&\rule[-5pt]{0pt}{5pt}
$\frak{su}^*(2n)\oplus\bb{R}$&
$n\geq 2$\\
\hline

\rule[0pt]{0pt}{12pt}
$\frak{sp}(n,n)$&\rule[-5pt]{0pt}{5pt}
$\frak{sp}(n,\bb{C})$&$n\geq 1$\\
&\rule[-5pt]{0pt}{5pt}
$\frak{su}^*(2n)\oplus\bb{R}$&$n\geq 1$\\
\hline

\rule[0pt]{0pt}{12pt}
$\frak{sp}(2n,\bb{R})$&\rule[-5pt]{0pt}{5pt}
$\frak{sp}(n,\bb{C})$&
$n\geq 2$
\\
\hline

\rule[0pt]{0pt}{12pt}
$\frak{sl}(2n,\bb{C})$&\rule[-5pt]{0pt}{5pt}
$\frak{sl}(m,\bb{C})\oplus
\frak{sl}(2n-m,\bb{C})\oplus\bb{C}$&
$m,2n-m\geq 1$\\
&\rule[-5pt]{0pt}{5pt}
$\frak{su}(m,2n-m)$&$m,2n-m\geq 1$\\
\hline

\rule[0pt]{0pt}{12pt}
$\frak{so}(2n,\bb{C})$&\rule[-5pt]{0pt}{5pt}
$\frak{so}(m,\bb{C})\oplus\frak{so}(2n-m,\bb{C})$
&$m,2n-m\geq 2$\\
&\rule[-5pt]{0pt}{5pt}
$\frak{so}(m,2n-m)$
&$m,2n-m\geq 2$\\
&\rule[-5pt]{0pt}{5pt}
$\frak{gl}(n,\bb{C})$&$n\geq 3$\\
&\rule[-5pt]{0pt}{5pt}
$\frak{so}^*(2n)$&$n\geq 3$\\
\hline

\rule[0pt]{0pt}{12pt}
$\frak{e}_{6(2)}$&\rule[-5pt]{0pt}{5pt}
$\frak{su}(3,3)\oplus\frak{sl}(2,\bb{R})$\\
\hline

\rule[0pt]{0pt}{12pt}
$\frak{e}_{7(-5)}$&\rule[-5pt]{0pt}{5pt}
$\frak{so}^*(12)\oplus\frak{sl}(2,\bb{R})$\\
\hline

\rule[0pt]{0pt}{12pt}
$\frak{e}_{7(-25)}$&\rule[-5pt]{0pt}{5pt}
$\frak{e}_{6(-26)}\oplus\bb{R}$\\
\hline

\rule[0pt]{0pt}{12pt}
$\frak{e}_{8(-24)}$&\rule[-5pt]{0pt}{5pt}
$\frak{e}_{7(-25)}\oplus\frak{sl}(2,\bb{R})$\\
\hline

\end{tabular}
\end{center}
\caption{}
\label{notddpair}
\end{table}


\bigskip

\begin{thebibliography}{99}
\bibitem{ber}
M. Berger, \href{http://www.numdam.org/item?id=ASENS_1957_3_74_2_85_0}{\it Les espaces sym\'{e}triques non compacts}, 
Ann.\ Sci.\ \'{E}cole\ Norm.\ Sup.\ 74 (1957), 85--177.
\bibitem{hel}
S. Helgason, ``Differential Geometry, Lie Groups, and Symmetric Spaces",
Pure and Appl.\ Math., Academic Press, 1978.
\bibitem{KnVo}
A. W. Knapp,  D. Vogan, Jr.,
``Cohomological Induction and Unitary 
Representations", 
Princeton  U.P., 1995.
\bibitem{kob94}
T. Kobayashi, \href{http://dx.doi.org/10.1007/BF01232239}{\it Discrete decomposability of the restriction of 
$A_\frak{q}(\lambda)$ with respect to reductive} subgroups 
and its applications, 
Invent.\ Math.\ 117 (1994), 181--205.
\bibitem{kob98i}
T. Kobayashi, \href{http://dx.doi.org/10.2307/120963}{\it Discrete decomposability of the restriction of 
$A_\frak{q}(\lambda)$, II.
---micro-local analysis} and asymptotic $K$-support,
Ann.\ of Math.\ 147 (1998), 709--729.
\bibitem{kob98ii}
T. Kobayashi, \href{http://dx.doi.org/10.1007/s002220050203}{\it Discrete decomposability of the restriction of 
$A_\frak{q}(\lambda)$, III.
---restriction of} Harish-Chandra modules and associated varieties, 
Invent.\ Math.\ 131 (1998), 229--256.
\bibitem{kob98iii}
T. Kobayashi, \href{http://dx.doi.org/10.1006/jfan.1997.3128}{\it Discrete series representations for the orbit spaces arising from two involutions} of real reductive Lie groups, 
J.\ Funct.\ Anal.\ 152 (1998), 100--135.
\end{thebibliography}
\end{document}